\documentclass[11pt]{amsart}

\usepackage[a4paper,hmargin=3cm,vmargin=3cm]{geometry}
\usepackage{amsfonts,amssymb,amscd,amstext,hyperref}
\usepackage{graphicx}
\usepackage{color}
\usepackage[dvips]{epsfig}
\usepackage[latin1]{inputenc}

\usepackage{fancyhdr}
\usepackage{times}

\usepackage{enumerate}
\usepackage{titlesec}
\usepackage{mathrsfs}
\usepackage{stmaryrd}

\def\R{\mathbb{R}}
\def\B{\mathbb{B}}
\def\N{\mathbb{N}}
\def\Z{\mathbb{Z}}
\def\C{\mathbb{C}}
\def\Q{\mathbb{Q}}

\newcommand{\ben}{\begin{enumerate}}
\newcommand{\bit}{\begin{itemize}}
\newcommand{\een}{\end{enumerate}}
\newcommand{\eit}{\end{itemize}}

\newcommand{\ed}{\end{document}}

\def\cA{\mathcal{A}}
\def\cU{\mathcal{U}}

\def\cD{\mathcal{D}}
\def\cR{\mathcal{R}}

\def\cH{\mathcal{H}}
\def\cL{\mathcal{L}}

\def\cG{\mathcal{G}}
\def\cF{\mathcal{F}}

\let\8=\infty \let\0=\emptyset  
\let\hat=\widehat

\let\landa=\lambda
\let\alfa=\alpha

\let\parc=\partial

\def\ep{\varepsilon}

\def\landa{\lambda}

\def\flecha{\rightarrow}
\def\esiz{\langle}
\def\esde{\rangle}

\def\S{\Sigma}

\def\cte.{\mathop{\rm cte.}\nolimits}

\def\cosh{\mathop{\rm cosh }\nolimits}

\def\N{\mathbb{N}}

\def\B{\mathbb{B}}
\def\Q{\mathbb{Q}}
\def\R{\mathbb{R}}
\def\Z{\mathbb{Z}}
\def\C{\mathbb{C}}
\def\A{\mathbb{A}}

\def\S{\mathbb{S}}

\newfont{\bb}{msbm10 at 12pt}

\titleformat{\section}
{\filcenter\bfseries\large} {\thesection{.}}{0.2cm}{}
\titleformat{\subsection}[runin]
{\bfseries} {\thesubsection{.}}{0.15cm}{}[.]
\titleformat{\subsubsection}[runin]
{\em}{\thesubsubsection{.}}{0.15cm}{}[.]

\usepackage[up,bf]{caption}

\newtheorem{theorem}{Theorem}[section]
\newtheorem{lemma}[theorem]{Lemma}
\newtheorem{proposition}[theorem]{Proposition}

\newtheorem{remark}[theorem]{Remark}
\newtheorem{corollary}[theorem]{Corollary}
\newtheorem{definition}[theorem]{Definition}

\theoremstyle{definition}

\numberwithin{equation}{section}
\numberwithin{figure}{section}

\usepackage{color}

\begin{document}
\fancyhead[LO]{Free boundary minimal annuli}
\fancyhead[RE]{Isabel Fernández, Laurent Hauswirth, Pablo Mira}
\fancyhead[RO,LE]{\thepage}

\thispagestyle{empty}

\begin{center}
{\bf \LARGE Free boundary minimal annuli immersed \\[0.2cm] in the unit ball}
\vspace*{5mm}

\hspace{0.2cm} {\Large Isabel Fernández, Laurent Hauswirth and Pablo Mira}
\end{center}

%



\footnote[0]{
\mbox{} \hspace{0.25cm} \emph{Keywords}: minimal surfaces, free boundary, critical catenoid, embeddedness, capillary surfaces.}

\vspace*{7mm}

\begin{quote}
{\small
\noindent {\bf Abstract}\hspace*{0.1cm}
We construct a family of compact free boundary minimal annuli immersed in the unit ball $\B^3$ of $\R^3$, the first such examples other than the critical catenoid. This solves a problem formulated by Nitsche in 1985. These annuli are symmetric with respect to two orthogonal planes and a finite group of rotations around an axis, and are foliated by spherical curvature lines. We show that the only free boundary minimal annulus embedded in $\B^3$ foliated by spherical curvature lines is the critical catenoid; in particular, the minimal annuli that we construct are not embedded. On the other hand, we also construct families of non-rotational compact embedded capillary minimal annuli in $\B^3$. Their existence solves in the negative a problem proposed by Wente in 1995.

\vspace*{0.1cm}

}
\end{quote}

\section{Introduction}
Amid the general theory of free boundary minimal surfaces, the case where the ambient space is the unit ball $\B^3$ of $\R^3$ is of special significance \cite{FS,Nit}. Here, we say that a compact minimal surface $\Sigma$ is \emph{free boundary} in $\B^3$ if it intersects $\parc \B^3$ orthogonally along $\parc \Sigma$. These surfaces appear as critical points of the area functional among all surfaces in $\B^3$ whose boundaries lie on $\parc \B^3$. After the seminal work of Fraser and Schoen \cite{FS0,FS}, the last decade has seen a great success in the construction of embedded free boundary minimal surfaces in $\B^3$ of different topological types, by using different methods; see \cite{CFS,CSW,FPZ,FS,KL,KW,KZ,Ke}.

A trivial example of such a free boundary surface is the flat equatorial disk of $\B^3$. In 1985, Nistche \cite{Nit} proved its topological uniqueness: \emph{any free boundary minimal disk immersed in $\B^3$ must be an equatorial disk}. 

The simplest non-trivial example of a free boundary minimal surface in $\B^3$ is the \emph{critical catenoid}, i.e., the only compact piece of a catenoid that intersects $\parc \B^3$ orthogonally along its boundary. This surface is rotational and has the topology of an annulus. The problem of the topological uniqueness of the critical catenoid among free boundary minimal annuli in $\B^3$ was already formulated by Nitsche \cite{Nit} in 1985, and it has been a relevant open problem of the theory for years, see e.g. \cite{L,F,So,Choe,LY,Pa,Nit,W2}. Our aim in this paper is to give a negative answer to this question. Specifically, we prove: 

\begin{theorem}\label{main}
There exists an infinite, countable, family of non-rotational free boundary minimal annuli immersed in $\B^3$. 

These annuli have one family of spherical curvature lines, and are not embedded. They are invariant under reflection through the planes $x_2=0$ and $x_3=0$, and under a finite group of rotations around the $x_3$-axis. The reflection with respect to $x_3=0$ interchanges the boundary components of the annulus. See Figures \ref{fig:eje1} and \ref{fig:eje2}.
\end{theorem}

A fundamental problem of the theory of free boundary minimal surfaces in $\B^3$ is the conjecture according to which the critical catenoid should be the only free boundary minimal annulus embedded in $\B^3$, see \cite{FL}, and \cite{L} for a detailed discussion. 
Theorem \ref{main} shows that the embeddedness assumption in the conjecture cannot be removed. Some partial affirmative answers to this conjecture have been recently obtained by several authors in \cite{AN,De,FS,KM,M,Seo,T}. 

The diversity of immersed minimal annuli provided by Theorem \ref{main} suggests to look for a potential counterexample to the critical catenoid conjecture with the same geometric structure, i.e., so that both boundary curves are elements of a foliation by spherical curvature lines of the annulus. However, we can actually show the following uniqueness result:

\begin{theorem}\label{th:uni}
The only free boundary minimal annulus embedded in $\B^3$ and foliated by spherical curvature lines is the critical catenoid.
\end{theorem}

Theorems \ref{main} and \ref{th:uni} give relevant insight not only towards the critical catenoid conjecture, but also about the general geometry of free boundary minimal annuli. One may regard our examples as free boundary versions of the constant mean curvature tori in space forms with planar or spherical curvature lines constructed in the 1980s by Wente \cite{W0} and then Abresch \cite{Ab} and Walter \cite{Wa1,Wa2}. In this sense, Theorem \ref{main} suggests in a natural way the interesting problem of classifying all free boundary minimal annuli immersed in $\B^3$. See Section \ref{subsec:discussion}.
The geometry of the examples in Theorem \ref{main} also has several formal similarities with the classical Riemann minimal examples foliated by circles in parallel planes; see Shiffman \cite{Shi} and Meeks-Perez-Ros \cite{MPR} for the fundamental uniqueness theorems of these Riemann examples.
\begin{figure}
\begin{center}
\vspace{-0.9cm}
\includegraphics[height=10cm]{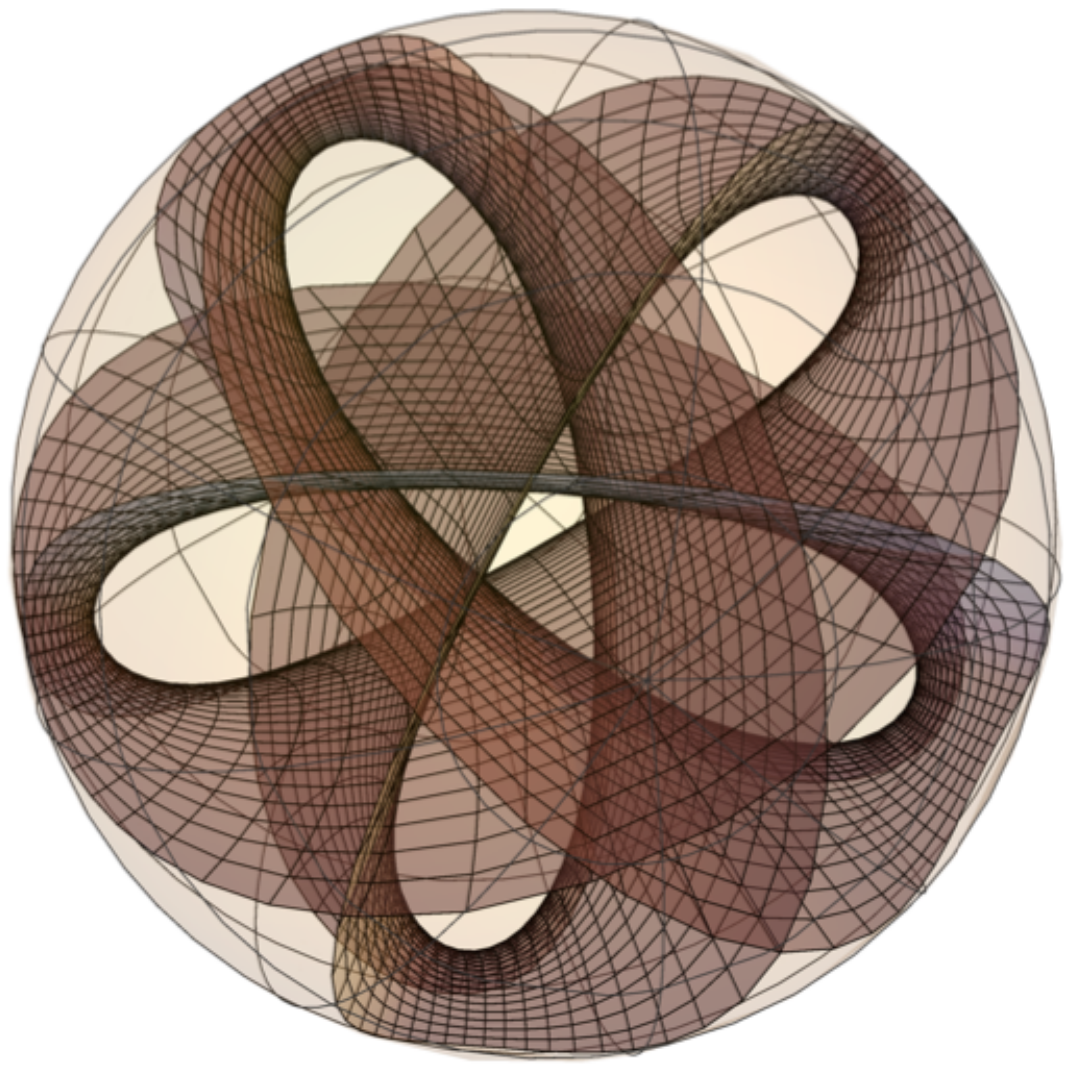} 
\includegraphics[height=10cm]{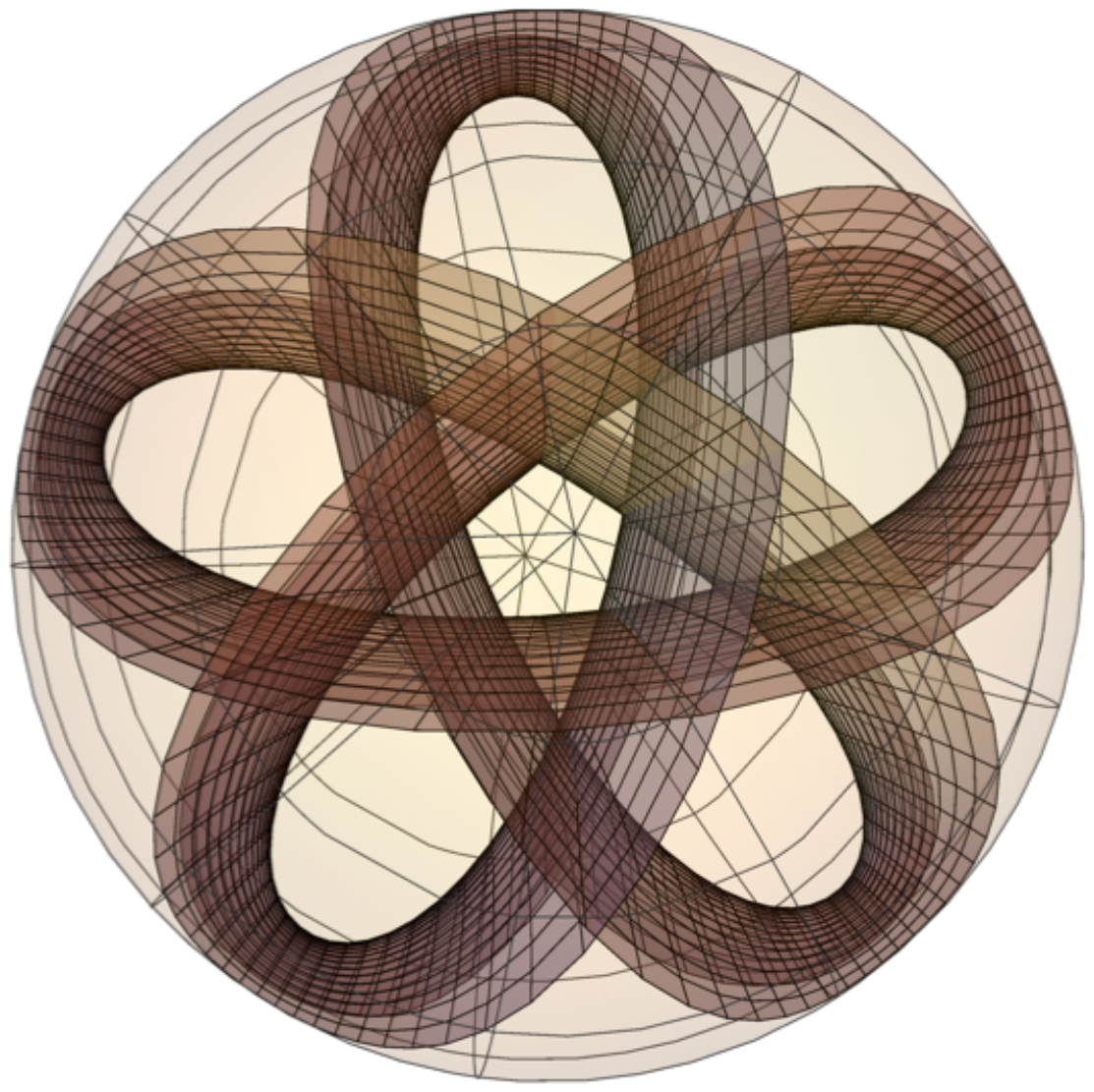} 
\vspace{-1.5cm}
\caption{Two views from above of a free boundary minimal annulus immersed in $\B^3$. The example is symmetric with respect to the $x_2=0$ and $x_3=0$ planes, and also with respect to the rotations with angles $6\pi k/5$, $k\in \{1,\dots, 5\}$, around the $x_3$-axis. Its symmetry group is isomorphic to the dihedral group $D_{10}$. The Gauss map along its central planar geodesic is a $3$-folded covering map of the great circle $\S^2\cap \{x_3=0\}$. See Figure \ref{fig:eje2} for two more views of the annulus.}\label{fig:eje1}
\end{center}
\end{figure}

The previous discussion can also be formulated in the more general \emph{capillary} context of compact constant mean curvature surfaces that meet $\parc \B^3$ at a constant angle along their boundary. By Nitsche's theorem \cite{Nit}, any capillary CMC disk immersed in $\B^3$ is an equatorial disk; see also Ros and Souam \cite{RS}. Regarding the topological uniqueness of capillary annuli, Wente \cite{W2} constructed for the case of non-zero mean curvature $H\neq 0$ examples of immersed, non-embedded, capillary and free boundary annuli in $\B^3$. For that, he used special properties of Abresch's solutions to the sinh-Gordon equation in \cite{Ab} that are not available in the minimal case that we treat here. Also in \cite{W2}, Wente asked whether any embedded capillary CMC annulus in $\B^3$ should be rotational. This can be seen as a capillary version of the critical catenoid conjecture. See also \cite{So,Pa}. 

In this paper we give a negative answer to Wente's problem; see Theorem \ref{capillary} for a more precise statement. See also Figure \ref{fig:eje3}.
\begin{theorem}\label{th:intro2}
There exist compact, embedded non-rotational minimal annuli in $\B^3$, with boundary contained in $\parc \B^3$, that are foliated by spherical curvature lines. In particular, they are embedded capillary minimal annuli in $\B^3$.
\end{theorem}

\begin{figure}
\begin{center}
\includegraphics[height=6.8cm]{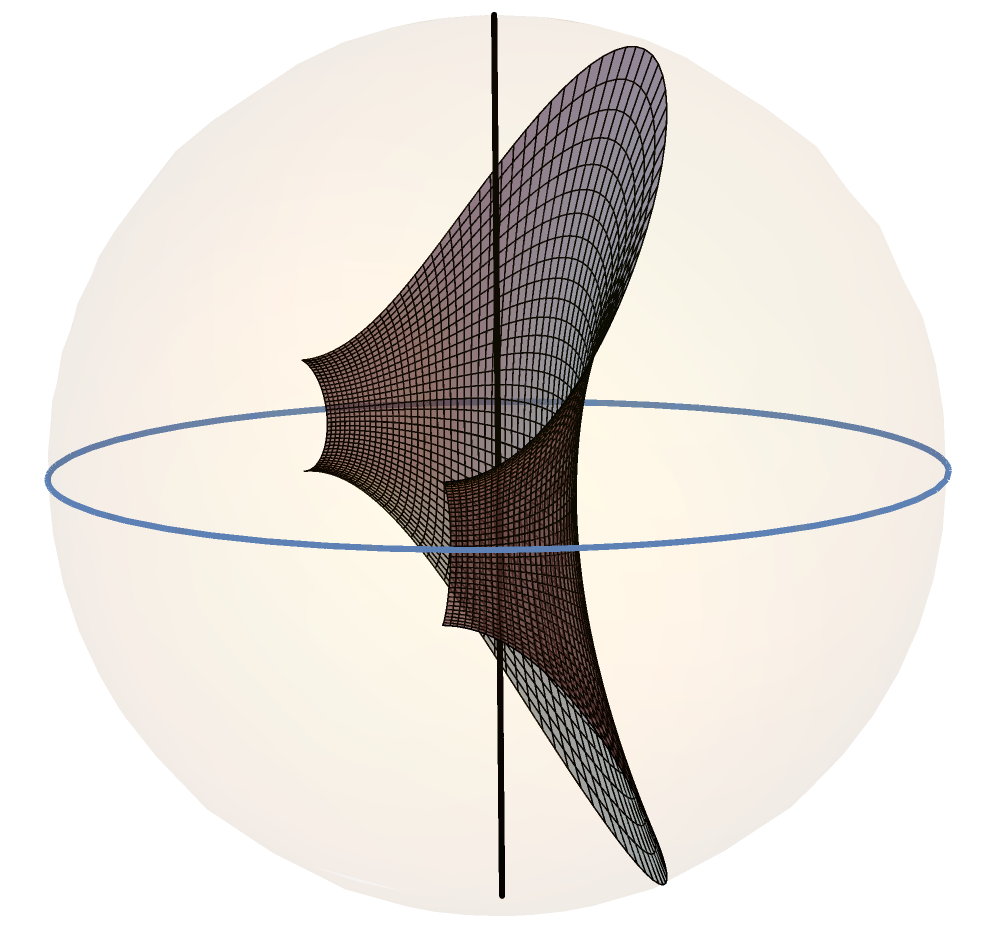} 
\includegraphics[height=6.8cm]{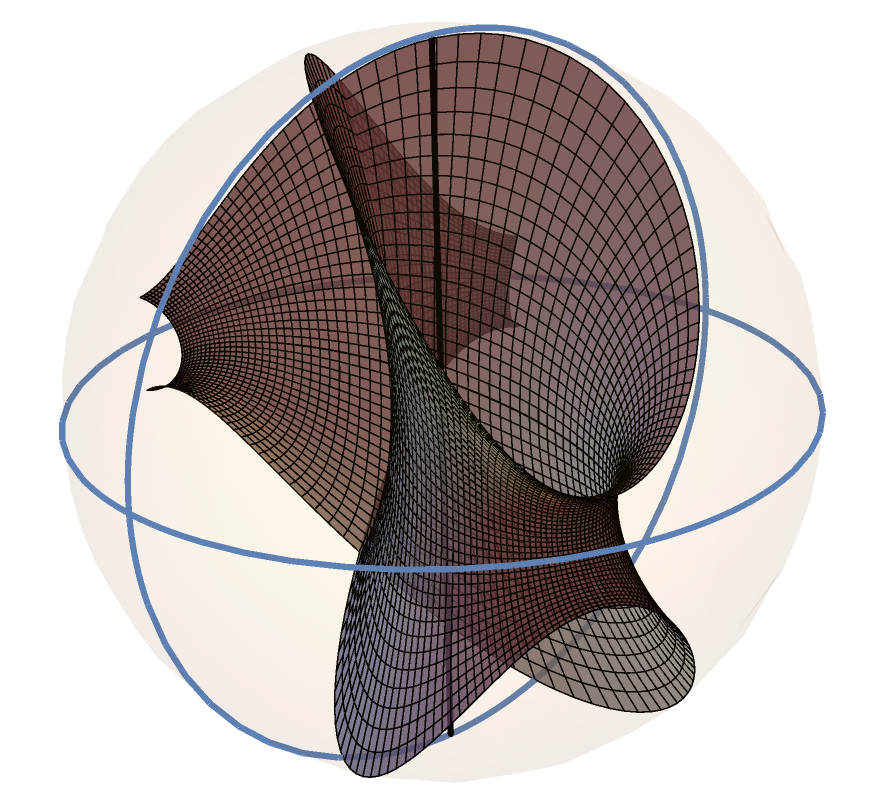} 
\caption{Left: an embedded fundamental piece of the free boundary minimal annulus in Figure \ref{fig:eje1}. It has two symmetry planes, and the whole annulus is obtained by the union of this piece and its rotations around the $x_3$-axis of angles $6\pi k /5$, $k=1,\dots, 4$. Right: union of the fundamental piece and its rotation of angle $6\pi /5$.}\label{fig:eje2}
\end{center}
\end{figure}

The general family of minimal surfaces in $\R^3$ with spherical curvature lines was described by Dobriner \cite{D} in the 19th century in terms of elliptic theta functions. Much more recently, in 1992, Wente \cite{W} recovered and reformulated Dobriner's classification using the solutions to a certain Hamiltonian planar system. Wente described these minimal surfaces as \emph{catenoids, perhaps covered infinitely often, from which a number of flat ends have been extruded}. Since these flat ends are placed along the planar curvature lines of the surface, this picture seems to forbid that such a surface could contain a compact free boundary minimal annulus in $\B^3$. In general, if a minimal annulus with spherical curvature lines has its two boundary curves on the same sphere, one would expect following Wente's description that it will have one or more flat ends in the middle.

The surprising realization that started this work is that, for \emph{some} minimal surfaces with spherical curvature lines, one can perform a phase shift of a half period in its Weierstrass data to avoid such flat ends. This allows the construction of compact minimal annuli with a central planar geodesic, similar to a catenoidal neck, but perhaps with an immersed dihedral \emph{flower} structure as the one depicted in Figure \ref{fig:eje1}. Still, in order to prove Theorem \ref{main}, one needs to control several aspects simultaneously: compactness, periods, center of the spheres that contain the curvature lines, and orthogonal intersection along the boundary.

Our construction also produces interesting non-compact examples.

\begin{corollary}\label{cor:intro}
There exist complete, non-compact minimal strips $\Sigma$ with free boundary in $\B^3$, foliated by spherical curvature lines. That is, $\parc \Sigma$ has two non-compact connected components, both of them contained in $\parc \B^3$, and $\Sigma$ intersects $\parc \B^3$ orthogonally along $\parc \Sigma$.
\end{corollary}

\begin{figure}
\begin{center}
\includegraphics[height=6.8cm]{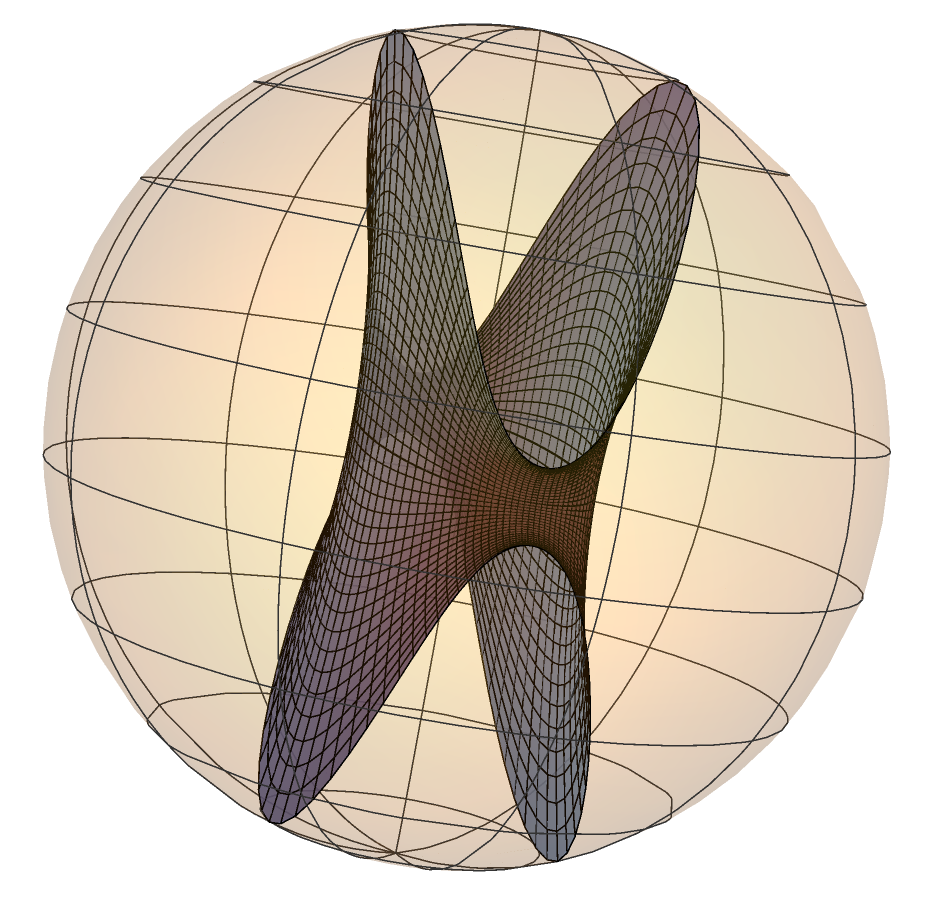}
\includegraphics[height=6.8cm]{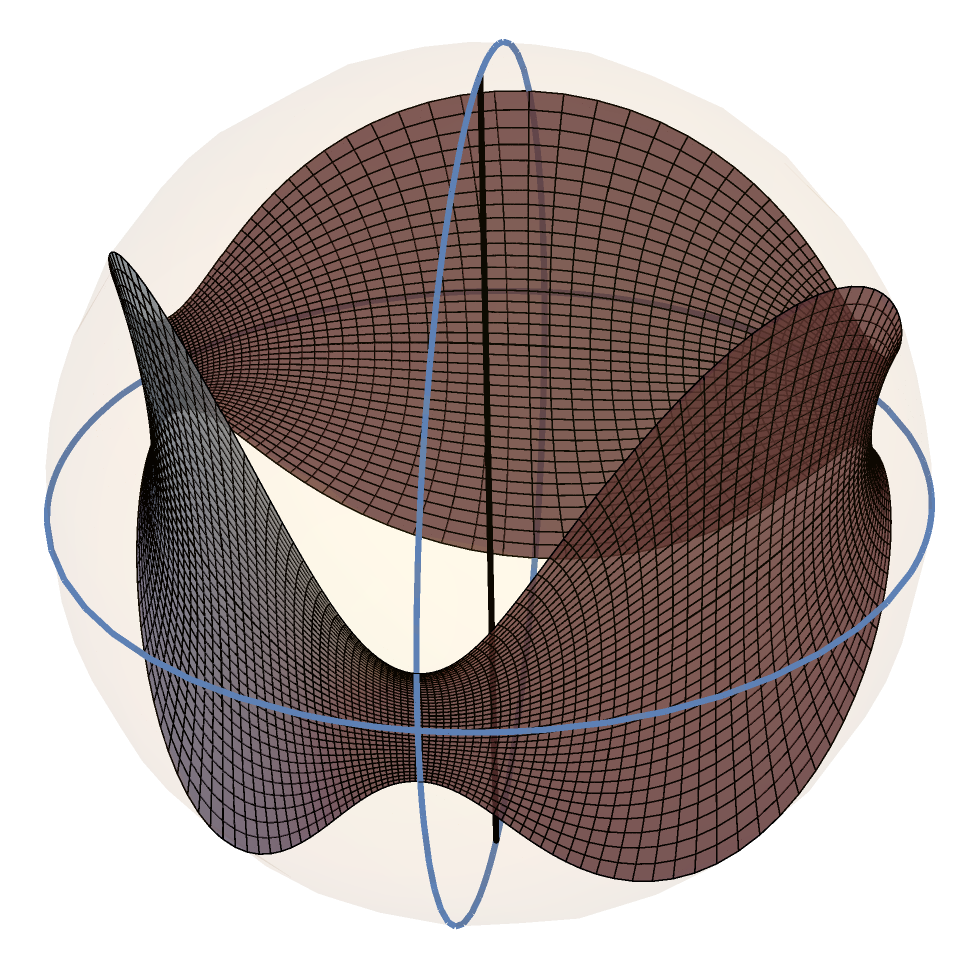} 
\caption{Two examples of embedded, capillary minimal annuli in the unit ball $\B^3$. They do not have free boundary but intersect $\parc \B^3$ with a constant angle.}\label{fig:eje3}
\end{center}
\end{figure}

We next outline the paper. In Section \ref{sec:prelim} we review, following Wente \cite{W}, the geometry of minimal surfaces in $\R^3$ with spherical curvature lines, and the integration of their associated Hamiltonian system. 
In Section \ref{sec:orbits} we analyze the phase space of this system in order to control later on the free boundary condition. In Section \ref{sec:dege} we solve explicitly the Hamiltonian system in a degenerate case that appears as a limit of the examples that we will construct. 

In Section \ref{sec:wei} we provide Weierstrass data in terms of elliptic functions for a family of complete minimal surfaces in $\R^3$ foliated by spherical curvature lines, having the additional property that one of such curvature lines is a \emph{bounded} planar geodesic. This possibility was not considered in \cite{W}. The phase space analysis in Section \ref{sec:orbits} ensures that, for some of these minimal surfaces $\Sigma$, one of their curvature lines intersects some sphere $S(p,R)$ of $\R^3$ orthogonally.

In Section \ref{sec:period} we study the period map that indicates when are the spherical curvature lines of $\Sigma$ periodic. We view it as an $\R$-valued map from the $2$-dimensional space of rectangular lattices in $\C$ where the Weierstrass data of $\Sigma$ are defined, and show that its level sets are connected, regular, analytic curves. 

In Section \ref{sec:height} we show that there exists an analytic $1$-parameter family of the examples $\Sigma$ constructed in Section \ref{sec:wei} for which the sphere of orthogonal intersection $B(p,R)$ has its center in the plane where the bounded planar geodesic of $\Sigma$ lies; so, by symmetry, $\Sigma$ has free boundary in some sphere of $\R^3$. For that, we use our study of the degenerate Hamiltonian system of Section \ref{sec:dege}. Our analysis of the period map in Section \ref{sec:period} implies that there is a dense family of examples within that $1$-parameter family that, after homothety and translation, are compact free boundary minimal annuli in the unit ball $\B^3$. This proves Theorem \ref{main}. 

In Section \ref{sec:examples} we prove Theorem \ref{th:intro2} on the existence of embedded capillary minimal annuli in $\B^3$, using the results of Sections  \ref{sec:wei}, \ref{sec:period} and \ref{sec:height}. Finally, in Section \ref{sec:uniqueness} we prove the uniqueness result stated in Theorem \ref{th:uni}.

\section{Minimal surfaces foliated by spherical curvature lines}\label{sec:prelim}

We say that a minimal surface in $\R^3$ has \emph{spherical curvature lines} if one its two families of curvature lines has the property that each of its elements lies in some sphere $S$ of $\R^3$, and hence, intersects this sphere at a constant angle. Here, we allow that the sphere $S$ has infinite radius, i.e., that it is a plane, for some of these curvature lines. In this section we review some aspects of their geometry, following Wente \cite{W}.

Let $\Sigma$ be a minimal surface with spherical curvature lines. Then, around each point of $\Sigma$ there exists a local conformal parameter $z=u+iv$ on a domain $D\subset \C$ with the following properties:
\begin{enumerate}
\item
The second fundamental form of $\Sigma$ is $II=-du^2+dv^2$, and its Hopf differential is $\esiz \psi_{zz},N\esde = -1/2$. Here, $\psi:D\flecha \R^3$ is a conformal parametrization of $\Sigma$, and $N$ is the unit normal. 
\item
The $v$-curves $\psi(u_0,v)$ are spherical curvature lines of $\Sigma$.
\item
The metric of the surface is $ds^2 =e^{2\omega} |dz|^2$, where $\omega$ is a solution to the Liouville equation $\Delta \omega - e^{-2\omega} =0$. The principal curvature associated to the $v$-curves is $\kappa_2=e^{-2\omega}$.
\item
There exist functions $\alfa(u),\beta(u)$ such that 
\begin{equation}\label{om1}
2 \omega_u = \alfa(u) e^{\omega} + \beta(u) e^{-\omega}.
\end{equation}
\end{enumerate}
Equation \eqref{om1} characterizes the property that the $v$-curves of $\psi$ are spherical curvature lines. The local function $\omega$ can actually be extended analytically to a global solution to the Liouville equation, defined on $\C$ minus a discrete set of points $\mathcal{Z}$, at which $\omega\to \8$. This extended function, that will also be denoted by $\omega$, satisfies \eqref{om1} globally on $\C\setminus \mathcal{Z}$ for adequate functions $\alfa,\beta:\R\flecha \R$, and defines a conformal minimal immersion 
\begin{equation}\label{conf}
\psi(u,v):\C\setminus \mathcal{Z}\flecha \R^3
\end{equation}
which extends our original local immersion, and that has flat ends at the points of $\mathcal{Z}$. The values of $\alfa(u),\beta(u)$ are finite for every $u\in \R$, and they describe the radius $R(u)$ of the sphere $S(u)\subset \R^3$ where $\psi(u,v)$ lies, and the intersection angle $\theta(u)$ of $\Sigma$ with $S(u)$ along $\psi(u,v)$ by the equations 
\begin{equation}\label{radan}
R^2=\frac{4+\beta^2}{\alfa^2}, \hspace{1cm} \theta = {\rm arctan}(2/\beta).
\end{equation}
As for the center $c(u)$ of $S(u)$, it is given (when $\alfa(u)\neq 0$) by
 \begin{equation}\label{gencen}
c(u)=\psi(u,v)  - \frac{2}{\alfa(u)} \frac{\psi_u(u,v)}{|\psi_u(u,v)|} - \frac{\beta(u)}{\alfa(u)} N(u,v).
 \end{equation}
At the points in $\mathcal{Z}$, since $\omega\to \8$, one must have $R=\8$, because the curves $\psi(u_0,v)$ are trivially bounded if $R(u_0)\neq \8$.

An important property is that all the centers $c(u)$ of the spheres $S(u)\subset \R^3$ lie in a common line $L\subset\R^3$.

Another important equation satisfied by $\psi(u,v)$ is  \begin{equation}\label{pos}
4(e^{\omega})_v^2 = p(u,e^{\omega}),
\end{equation}
where, for each $u\in \R$, $p(u,X)$ is the (at most) fourth degree polynomial defined by 
  \begin{equation}\label{def:p}
  p(u,X):= -\alfa^2 X^4 - 4 \alfa' X^3 + 6 \gamma X^2 + 4 \beta' X - \left(4 + \beta^2\right), \hspace{1cm} 6\gamma:=6\alfa \beta -4\delta.
  \end{equation}

\subsection{A Hamiltonian system} The functions $(\alfa(u),\beta(u))$ above satisfy the autonomous system 
\begin{equation}\label{system}\left\{ \def\arraystretch{1.5}\begin{array}{lll} \alfa'' & = &  \delta \alfa - 2 \alfa^2 \beta, \\ \beta'' & = &  \delta \beta - 
2 \alfa \beta^2 - 2\alfa, \end{array} \right.
 \end{equation}
with respect to some constant $\delta\in \R$. The system \eqref{system} has a Hamiltonian nature, and in particular has some preserved quantities. More specifically, any solution $(\alfa(u),\beta(u))$ to \eqref{system} is defined on $\R$, and associated to it there exist constants $h,k\in \R$ such that 
\begin{equation}\label{def:h}
h= \alfa' \beta' +\alfa^2 - \delta\alfa \beta + \alfa^2 \beta^2
\end{equation}
and
\begin{equation}\label{def:k}
4k=(\alfa \beta'-\alfa'\beta)^2 + 4 \alfa'^2 + 4 \alfa^3 \beta - 4 \delta\alfa^2\end{equation}
hold for every $u\in\R$. The constants $h,k,\delta$ let us define, associated to any minimal surface foliated by spherical curvature lines, the polynomial 
\begin{equation}\label{spo}
q(x):=-x^3+\delta x^2+ h x +k
\end{equation}
that will play an important role. This polynomial is strongly related to the one introduced in \eqref{def:p}. For instance, a computation using  \eqref{def:h} and \eqref{def:k} gives the relation $${\rm discriminant}(q(x)) = 16\, {\rm discriminant}(p(u,X)),$$ for every $u\in \R$. In particular, the discriminant of $p(u,X)$ is actually independent from $u$, and described by the constants $h,k,\delta$.


The system \eqref{system} can be integrated by separated variables following a procedure by Jacobi. Consider the change of coordinates 
\begin{equation}\label{change}
\alfa \beta = s+t, \hspace{1cm} \alfa^2 = -st.
\end{equation}
It defines a diffeomorphism between the halfplane $\alfa>0$ (or $\alfa<0$) of the $(\alfa,\beta)$-plane and the quadrant $s>0, t<0$ of the $(s,t)$-plane. Then, \eqref{system} can be rewritten as
\begin{equation}\label{system2}\left\{ \def\arraystretch{1.5}\begin{array}{lll} s'(\landa)^2& = &  s^2 q(s), \hspace{1cm} (s>0), \\ t'(\landa)^2 & = &  t^2 q (t), \hspace{1.1cm} (t<0),\end{array} \right.
 \end{equation}
where $q$ is the polynomial \eqref{spo}, and the parameter $\landa$ is related to $u$ by 
\begin{equation}\label{repa}
2u'(\landa) = s(\landa)-t(\landa)>0.
\end{equation}
Note that if $(\alfa,\beta)$ satisfies \eqref{change} for some $(s,t)$, then $(-\alfa,-\beta)$ also does. 

Conversely, given a solution $(\alfa,\beta)$ to \eqref{system} with respect to some constant $\delta\in \R$, one can seek to integrate \eqref{om1} to find a solution $\omega$ to the Liouville equation, and then obtain from $\omega$ a minimal surface $\Sigma$ that satisfies \eqref{om1} with respect to $(\alfa,\beta)$, and in particular has spherical curvature lines. 
In order to do this, in the view of \eqref{pos}, it is also necessary to impose the additional condition that $p(u,X)>0$ for some $u\in \R$ and some $X>0$. 

It turns out that this positivity condition is also sufficient for the existence, as explained in the next lemma, that is contained in \cite[Theorem 2.3]{W}.
%
%
%
%
%
\begin{lemma}
If a solution $(\alfa(u),\beta(u))$ to \eqref{system} satisfies 
\begin{equation}\label{positi}
p(u_0,X)>0 \hspace{0.5cm} \text{for some pair $(u_0,X)\in \R\times (0,\8)$,} 
\end{equation}
then there exists a conformal minimal immersion $\psi(u,v)$ with spherical curvature lines that satisifies \eqref{om1} with respect to $\alfa(u),\beta(u)$.
\end{lemma}

\begin{remark}\label{rem21}
Assume that $\alfa(u_0)=0$. By \eqref{radan}, the curvature line $\psi(u_0,v)$ intersects a plane at a constant angle, and so it can be unbounded. If $\psi(u_0,v)$ is unbounded, the surface has a flat end at some point in the $(u_0,v)$ line, and $e^{\omega}$ is unbounded along that line.  However, assume that additionally to $\alfa(u_0)=0$ we also have $\alfa'(u_0)>0$. In that situation, the polynomial $p(u_0,X)$ in \eqref{def:p} has degree three and negative leading coefficient. Thus, by \eqref{pos}, $e^{\omega}$ must be bounded, i.e., the planar curvature line $\psi(u_0,v)$ is bounded.
\end{remark}

\subsection{Weierstrass representation}\label{sec:were}
Minimal surfaces with spherical curvature lines can be described by elliptic Weierstrass data. Specifically, given $g_2,g_3\in \R$, let $\wp(z)=\wp(z;g_2,g_3)$ be the (possibly degenerate) Weierstrass $P$-function associated to $g_2,g_3$, so that it satisfies its standard differential equation $\wp'^2 = 4\wp^3 -g_2 \wp -g_3$. There are three cases:

\begin{enumerate}
\item
$g_2=g_3=0$. In that degenerate case, $\wp(z)=1/z^2$, which is holomorphic in $\C\setminus \Lambda$, with $\Lambda=\{0\}$.
 \item
The \emph{modular discriminant} $\Delta_{\rm mod}:=g_2^3-27 g_3^2$ is zero, with $g_2g_3\neq 0$. In that case, $\wp$ is a degenerate Weierstrass P-function that is singly periodic with respecto to either a real period or a purely imaginary period. See \cite{Ch}. The function $\wp(z)$ is holomorphic in $\C\setminus \Lambda$, where $\Lambda$ is the set of multiples of this fundamental period, and has double poles at the points of $\Lambda$.
\item
$\Delta_{\rm mod} \neq 0$. Then, $\wp(z)$ is doubly periodic with respect to a lattice $\Lambda \subset \C$. The map $\wp(z)$ is holomorphic in $\C\setminus \Lambda$, and has double poles at the points of $\Lambda$.
\end{enumerate}
More specifically, in the case $\Delta_{\rm mod} <0$, since $g_2,g_3\in \R$, the lattice $\Lambda$ has generators $\{2\omega_1,2\omega_2\}$ with $\omega_2= \overline{\omega_1}$, i.e., $\Lambda$ is a \emph{rhombic lattice}. Likewise, in the case $\Delta_{\rm mod} >0$, we have a \emph{rectangular lattice} $\Lambda$ with generators $\{2\omega_1,2\omega_2\}$, being $\omega_1>0$ and $\omega_2\in i \R$, with ${\rm Im}(\omega_2)>0$. See Figure \ref{fig:34}.

\begin{figure}[h]
\begin{center}
\includegraphics[height=4.1cm]{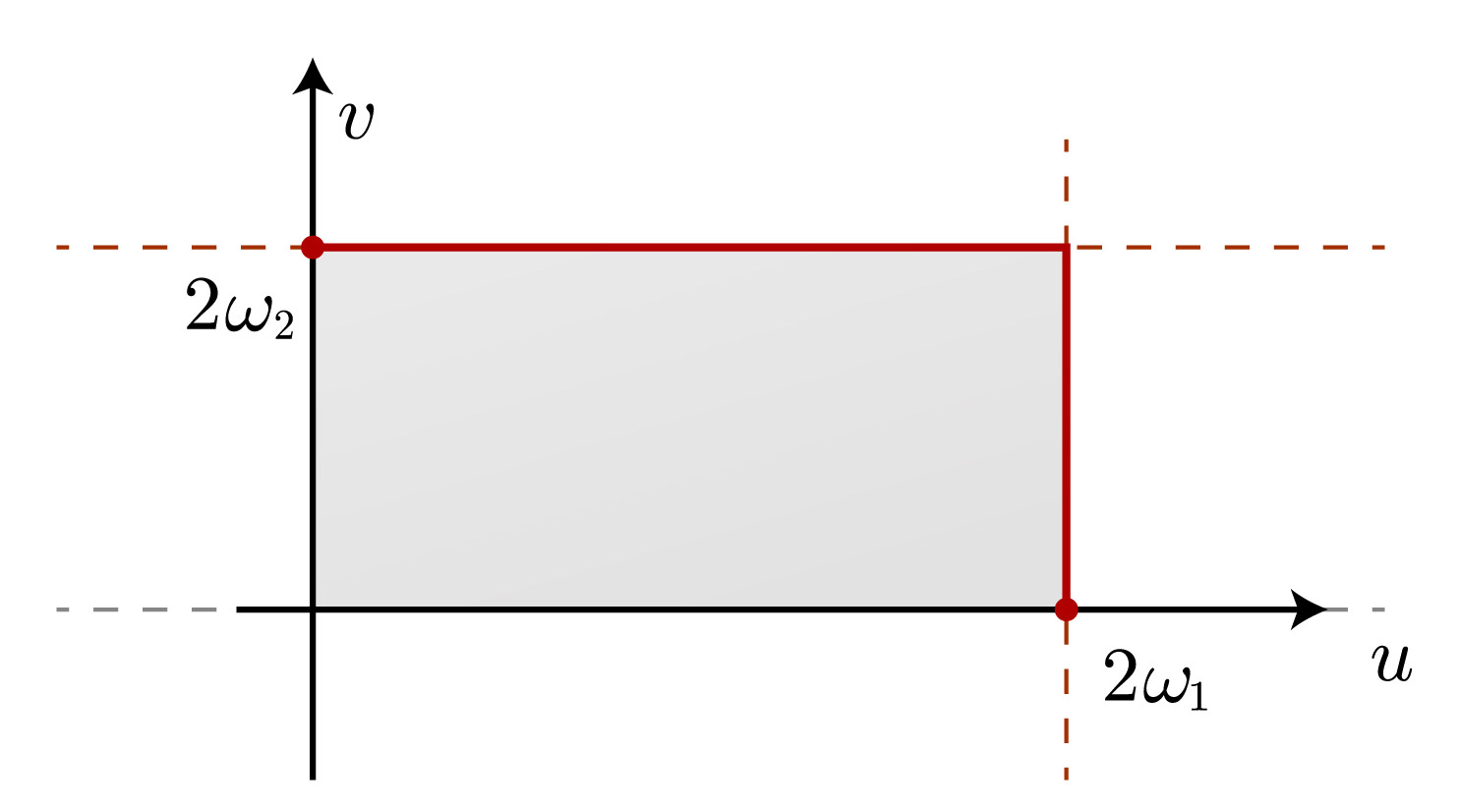} 
\includegraphics[height=4.5cm]{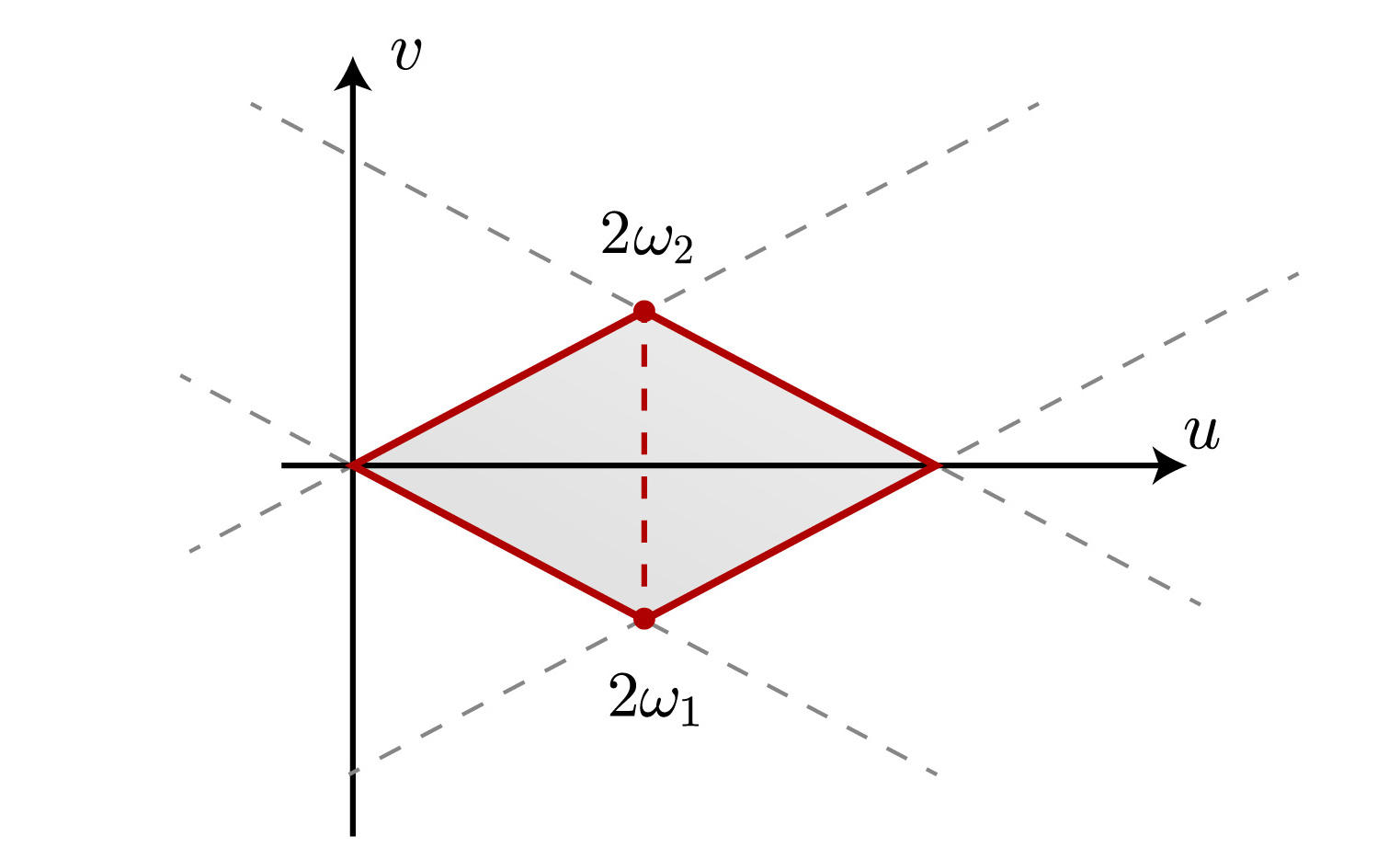}
 \caption{Fundamental domains of the lattice $\Lambda$ for $\Delta_{\rm mod} >0$ and $\Delta_{\rm mod} <0$.}\label{fig:34}
\end{center}
\end{figure}

In what follows, let $\Lambda\subset \C$ be the set of poles of $\wp$, as detailed above. Let $b\in \R$ satisfy the cubic equation 
 \begin{equation}\label{cubicb}
 b^3 -4 g_2 b - 16 g_3=1.
 \end{equation}
Define the Weierstrass data $\Phi:=(\phi_1,\phi_2,\phi_3)$ given by 
\begin{equation}\label{weidata}
\Phi=\frac{\phi}{2}\left(\frac{1}{g}-g,i\left(\frac{1}{g}+g\right),2\right):\C\flecha \C\cup \{\8\}. \end{equation}
Here, $\phi(z):=b-4\wp(z)$, and $g$ is a meromorphic map on $\C$ such that $g/g' =\phi$. Note that $g$ is defined up to a multiplicative constant.

These Weierstrass data $\Phi$ define a complete, conformally immersed minimal surface $\psi:\C\setminus \Lambda\flecha \R^3$ given for $z=u+iv$ by
\begin{equation}\label{weipar0}
\psi(u,v)\equiv \psi(z) = 2 {\rm Re} \int_0^z \Phi(z) dz.
\end{equation}
Note that $\phi$ has poles at the points of $\Lambda$. At each $z_0\in \Lambda$, $\psi$ has an embedded flat end. By construction, the second fundamental form of $\psi$ is $II=-du^2+dv^2$,
where $z=u+iv$, and so the $v$-curves $\psi(u_0,v)$ are curvature lines of the surface associated to the positive principal curvature. It can be checked that, by our choice of $\Phi$, these are actually spherical curvature lines, and the centers $c(u)$ of the spheres $S(c(u),R(u))$ that contain the curves $\psi(u,v)$ all lie in a common vertical line $L$ of $\R^3$. If a $v$-curve contains a point $z_0=u_0+iv_0$ of the lattice $\Lambda$, then the curvature line $\psi(u_0,v)$ is unbounded (the surface $\psi$ has an end at $z_0$), and in particular it is contained in a plane, that is, $R(u_0)=\8$.

Conversely, it is proved by Wente, see Theorems 4.1, 4.2 and 4.3 in \cite{W}, that up to an isometry and a homothety of $\R^3$, and except for some degenerate cases, any minimal surface in $\R^3$ with spherical curvature lines can be constructed from the above Weierstrass data. Here, the degenerate cases correspond to the surfaces with planar curvature lines (including the catenoid), which appear when $\alfa(u)\equiv 0$.
%
%

\section{Phase space analysis}\label{sec:orbits}
In this section we analyze the systems \eqref{system}, \eqref{system2} analytically. To start, we fix the polynomial $q(x)$ in \eqref{spo}. While the next discussion can be carried out more generally, we will directly work in the conditions of our construction and make the assumption that $q(x)$ can be factorized as 
 \begin{equation}\label{spo2}
 q(x)=-(x-r_1)(x-r_2)(x-r_3),
 \end{equation}
where $r_1< r_2<0<r_3$.

Let $(\alfa(u),\beta(u))$ be a solution to \eqref{system} whose invariants $h,k,\delta$ are the coefficients of $q(x)$, written as in \eqref{spo}. 
Let $\mathcal{J}\subset \R$ be an open interval on which $\alfa$ does not vanish. As explained in the previous section, the restriction of $(\alfa(u),\beta(u))$ to $\mathcal{J}$ determines a solution $(s(\landa),t(\landa))$ to system \eqref{system2}. For this solution, by \eqref{spo2}, we must have $s(\landa)\in [0,r_3]$, and $t(\landa)\in (-\8,r_1]\cup [r_2,0]$, in order for \eqref{system2} to be defined. But as a matter of fact, if $t(\landa)$ is contained in $(-\8,r_1]$, it can be proved that $(\alfa(u),\beta(u))$ does not satisfy the positivity condition \eqref{positi}. So, it does not generate a minimal surface, and is unimportant to our study. Hence, we will assume that $t(\landa)$ takes values in $[r_2,0]$.

In this way, the trajectory $(s(\landa),t(\landa))$ is contained in the rectangle $\cR:= [0,r_3]\times [r_2,0]$ of the $(s,t)$-plane. If $s(\landa)$ or $t(\landa)$ is constant, this trajectory is contained in one of the edges of $\cR$. In any other case, the trajectory meets the interior of $\cR$. Thus, for any connected arc of this trajectory that lies in the interior of $\cR$, there exist $\ep_1,\ep_2 \in \{-1,1\}$ such that
\begin{equation}\label{system3}\left\{ \def\arraystretch{1.5}\begin{array}{lll} s'(\landa)& = &  \ep_1 s \sqrt{q(s)}, \hspace{1cm} (s>0), \\ t'(\landa) & = &  \ep_2 t\sqrt{q(t)}, \hspace{1.1cm} (t<0).\end{array} \right.
 \end{equation}
While this gives four different autonomous systems in normal form, the change $\landa \mapsto -\landa$ reverses both signs of $\ep_1,\ep_2$, and so the corresponding systems have the same orbits, with opposite orientation. So, we can regard the rectangle $\cR$ as two different phase spaces $\cR_+, \cR_-$, associated respectively to the systems 
\begin{equation}\label{system4} (a) \hspace{0.2cm} \left\{ \def\arraystretch{1.5}\begin{array}{lll} s'(\landa)& = &  \ep s \sqrt{q(s)},  \\ t'(\landa) & = &  \ep t\sqrt{q(t)}, \end{array} \right. \hspace{0.5cm} (b) \hspace{0.2cm} \left\{ \def\arraystretch{1.5}\begin{array}{llr} s'(\landa)& = &  -\ep s \sqrt{q(s)},  \\ t'(\landa) & = & \ep t\sqrt{q(t)},\end{array} \right.
 \end{equation}
where $\ep = \pm 1$, $s>0$ and $t<0$. Note that while each of these systems is formally decoupled, and can be integrated as two independent first order autonomous ODEs in normal form, the common parameter $\landa$ establishes a link between both solutions. We describe next the behavior of the orbits of $\cR_+$ and $\cR_-$.

Any trajectory $(s(\landa),t(\landa))$ of $\cR_+$ has negative slope at every point of ${\rm int}(\cR)$. Assume that $\ep =1$ for definiteness. Then, $s,t$ are defined (up to a translation in the $\landa$ parameter) for every $\landa \leq 0$, so that $s(\landa),t(\landa)\to 0$ with the order of $\exp(-|\landa|)$ as $\landa\to -\8$, and either $s(0)=r_3$ or $t(0)=r_2$ at the other end. So, the orbits $\Gamma$ of $\cR_+$ that touch the interior of $\cR$ start (in infinite $\landa$-time) at the equilibrium $(0,0)$, and end up (in finite $\landa$-time) at a point in the \emph{outer rim} $\cR\cap (\{s=r_3\}\cup \{t=r_2\})$ of the rectangle $\cR$. Moreover, $\Gamma$ is tangent to either $s=r_3$ or $t=r_2$ at this point in the outer rim, unless $\Gamma$ is the special orbit $\Gamma_1$ of $\cR_+$ that joins the origin to the \emph{opposite vertex} $(r_2,r_3)$ of $\cR$, that is also an equilibrium of \eqref{system4}.

One can discuss similarly the trajectories of $\cR_-$. This time, any such trajectory $(s(\landa),t(\landa))$ that meets the interior of $\cR$ has positive slope in ${\rm int}(\cR)$, and joins in finite $\landa$-time a point $(r_3,t)$, $t\in (r_2,0)$ with a point $(s,r_2)$, $s\in (0,r_3)$. The orbits in $\cR_-$ foliate the interior of $\cR$, and none of them passes through the vertex $(r_3,r_2)$.

\begin{figure}
\begin{center}
\includegraphics[height=5cm]{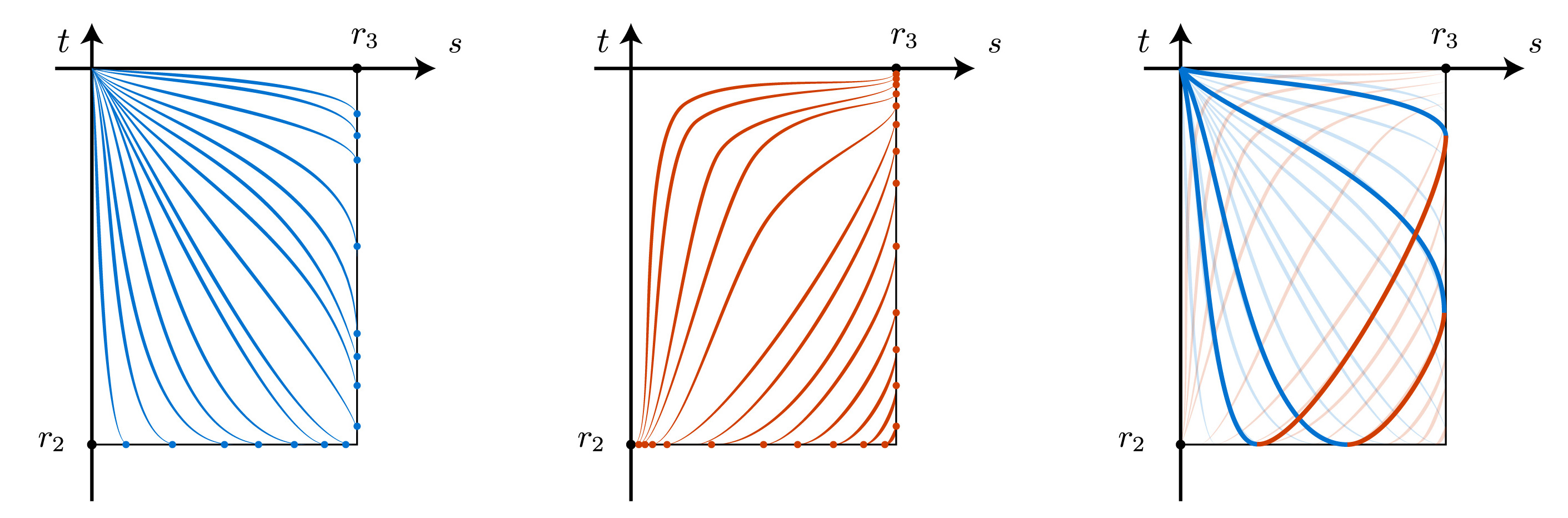} 
\caption{Left: Trajectories in the phase space $\cR_+$. Middle: Trajectories in the phase space $\cR_-$. Right: Trajectories of \eqref{system2}, obtained by analytic continuation.}\label{fig:fases}
\end{center}
\end{figure}

The orbits of $\cR_+$ and $\cR-$ can be joined in a tangential, analytic way at any point of the outer rim of $\cR$ (with the vertex $(r_3,r_2)$ excluded), to create a trajectory of system \eqref{system2}. In this way, we obtain trajectories $\Gamma(\landa)=(s(\landa),t(\landa))$ of \eqref{system2} that \emph{bounce at the walls} $s=r_3$ and $t=r_2$ of $\cR$ following the orbits of $\cR_+$, $\cR_-$. See Figure \ref{fig:fases}. These trajectories $\Gamma(\landa)$ are thus defined for every $\landa \in \R$, with $\Gamma(\landa)\to (0,0)$ as $\landa\to \pm \8$. 

The special orbit $\Gamma_1(\landa)$ is a degenerate case that bounces back in a symmetric way after hitting the vertex $(r_3,r_2)$; say, if $\Gamma_1(0)=(r_3,r_2)$ after a translation in the $\landa$-parameter, then $\Gamma_1(\landa)=\Gamma_1(-\landa)$ for every $\landa$. 

We summarize this discussion in the following Lemma.

\begin{lemma}\label{orbija}
Any solution $\Gamma(\landa):=(s(\landa),t(\landa))$ to \eqref{system2} is real analytic, and defined for all $\landa \in \R$, with $\Gamma(\landa)\to (0,0)$ with order $\exp(-|\landa|)$ as $\landa \to \pm \8$. 

Any such trajectory $\Gamma(\landa)$ is contained in the rectangle $\cR=[0,r_3]\times [r_2,0]$, and it is regular with non-zero, finite slope at every point $p_0=\Gamma(\landa_0)\in {\rm int}(\cR)$. Also, it intersects exactly once each of the boundary segments $\{r_3\}\times [r_2,0)$ and $(0,r_3]\times \{r_2\}$ of $\cR$, unless either $s(\landa)$ or $t(\landa)$ is constant.

These unique intersections with the boundary are tangential (i.e. either zero or infinite slope) unless $\Gamma$ is the unique solution $\Gamma_1(\landa)$ to \eqref{system2} that passes through $(r_3,r_2)$ at, say, $\landa=0$. In that situation, the curve $\Gamma_1(\landa)$ has a unique singular point at $\landa=0$, it has negative slope at any other point, and satisfies $\Gamma_1(\landa)=\Gamma_1(-\landa)$ for every $\landa \in \R$.
\end{lemma}

We next go back to system \eqref{system}, and study the parameter $u$ of the solution $(\alfa(u),\beta(u))$ to \eqref{system} associated to the trajectory $\Gamma(\landa)$ to \eqref{system2} that we just discussed. The parameters $u,\landa$ are related by \eqref{repa}. Since $s(\landa) t(\landa)\neq 0$ for every $\landa\in \R$, it follows by \eqref{change} that we can choose our initial interval $\mathcal{J}\subset \R$ as $\mathcal{J}= \{u(\landa): \landa \in \R\}$. Now, since $\Gamma(t)\to (0,0)$ with order $\exp(-|\landa| )$ as $\landa\to \pm \8$ (Lemma \ref{orbija}), this interval is bounded. So, up to a translation in the $u$-parameter, we can set $\mathcal{J}=(0,L)$ for some $L>0$, with $u=0$ (resp. $u=L$) corresponding to $\landa =-\8$ (resp. $\landa=\8$).

Note that, by \eqref{change}, we have $\alfa(0)=\alfa(L)=0$. So, it follows from \eqref{def:h} and \eqref{def:k} that, at $u=0$, the solution $(\alfa(u),\beta(u))$ to \eqref{system} is uniquely determined by the value $\beta_0:= \beta(0)$ and the sign of $\alfa'(0)$. Once we fix the sign of $\alfa'(0)$, any choice of $\beta_0$ determines a trajectory $\Gamma(t)$ to \eqref{system2}. Conversely, after fixing a sign for $\alfa'(0)$, any trajectory of \eqref{system2} corresponds to some  unique choice of $\alfa(0)=0$, $\beta(0)=\beta_0\in \R$ in system \eqref{system}. To this respect, it should be noted that the \emph{dual} solutions $(\alfa,\beta)$ and $(-\alfa,-\beta)$ to \eqref{system} project to the same orbit $(s(\landa),t(\landa))$ of \eqref{system2}, by \eqref{change}.

We are specially interested in one particular orbit of \eqref{system2}. Note that by the previous discussion, we do not need to fix the sign of $\alfa'(0)$ in the definition below.

\begin{definition}\label{def:gamma0}
In the conditions above, we denote by $\Gamma_0$ the orbit of \eqref{system2} that corresponds to the choice $\alfa(0)=0$, $\beta(0)=0$.
\end{definition}

\begin{remark}\label{iniconi}
In the case that the solution $(\alfa(u),\beta(u))$ to \eqref{system} defines a minimal surface $\Sigma$, the condition $\alfa(0)=\beta(0)=0$ is equivalent, by \eqref{radan}, to the property that $\Sigma$ intersects orthogonally a plane along the curvature line $\psi(0,v)$.
\end{remark}

Consider next the following assumptions on the polynomial $q(x)$ in \eqref{spo} (written as \eqref{spo2}):

\begin{enumerate}
\item
$r_1< r_2<0<r_3$.
\item
$r_1 r_2 r_3 =1$.
\item
$q'(0)>0$, that is,  $r_1 r_2 + r_1 r_3 + r_2 r_3 <0$.
\end{enumerate}
The first condition had already been imposed on $q(x)$. The second assumption is just a normalization that can be attained after a homothety and a conformal reparametrization of the surface. The third one will be fundamental to control the free boundary condition for minimal surfaces (see Corollary \ref{cor:orbits} below). By a basic algebraic manipulation we have:

\begin{lemma}\label{conde}
For $(r_1,r_2,r_3)\in \R^3$, the following two claims are equivalent:
\begin{enumerate}
\item[i)]
$(r_1,r_2,r_3)$ satisfy conditions (1)-(3) above.
\item[ii)]
$r_3=1/(r_1 r_2)$, where $(r_1,r_2)$ satisfy
\begin{equation}\label{ine}
-\sqrt[3]{2}< r_2< 0, \hspace{1cm}  \frac{-1-\sqrt{1-4 r_2^3}}{2 r_2^2}< r_1 < r_2.
\end{equation}
\end{enumerate}
\end{lemma}

Along the rest of the paper, we will consider the open sets $\Omega_0\subset \Omega\subset W\subset \R^2$ given  by the following relations for $(r_1,r_2)$ (see Figure \ref{fig:fig1}): 
 \begin{equation}\label{omw}
 W\equiv \{r_1<r_2<0\},\hspace{0.5cm} \Omega\equiv \{ \eqref{ine} \text{holds}\}, \hspace{0.5cm} \Omega_0\equiv\Omega\cap \{r_1 r_2^2<-1\}.
 \end{equation}
We remark that the condition $r_1r_2^2 <-1$ for $\Omega_0$ is equivalent to $r_2+r_3<0$.

\begin{figure}
\begin{center}
\includegraphics[height=6.5cm]{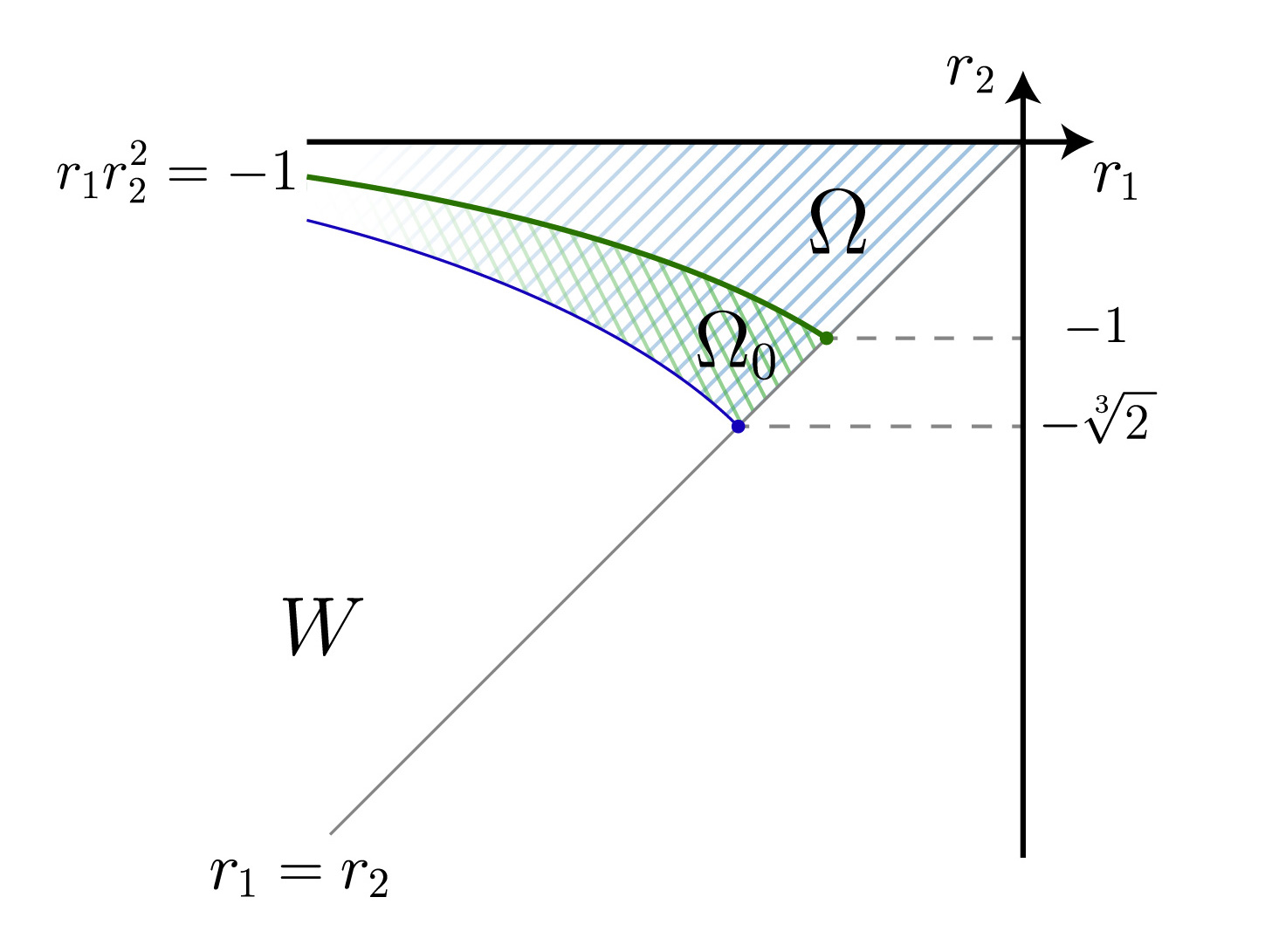}
\caption{The domains $W$, $\Omega$ and $\Omega_0$.}\label{fig:fig1}
\end{center}
\end{figure}

\begin{proposition}\label{pro:orbits}
Let $q(x)$ be the polynomial \eqref{spo2}, where $(r_1,r_2)\in \Omega$ and $r_3=1/(r_1 r_2)$. Then:
\begin{enumerate}
\item
The orbit $\Gamma_0$ intersects the line $s+t=0$ in \emph{at most} one point before hitting the horizontal segment $\cR\cap \{t=r_2\}$.
\item
If $(r_1,r_2)\in \Omega_0$, this intersection point always exists.
\end{enumerate}
\end{proposition}
\begin{proof}
Since $\alfa(0)=0$, we get from \eqref{def:h} and $(r_1,r_2)\in \Omega$ that $$\alfa'(0) \beta'(0)=h=q'(0)>0.$$ So, by $\alfa(0)=\beta(0)=0$, both $\alfa(u),\beta(u)$ have the same sign for $u>0$ small enough. By \eqref{change}, the orbit $\Gamma_0=\Gamma_0(\landa)$ is above $s+t=0$ near the origin, for $\landa \approx -\8$. By our previous analysis, $\Gamma_0(\landa)\to (0,0)$ as $\landa \to \8$, and in the process it intersects the segment $\cR\cap \{t=r_2\}$ exactly once. 

Assume that $(r_1,r_2)\in \Omega_0$. Then, we have $r_2+r_3< 0$ and so the segment $\cR\cap \{t=r_2\}$ lies in the halfplane $s+t< 0$. Thus, $\Gamma_0$ intersects $s+t=0$ before reaching $\cR\cap \{t=r_2\}$. So, we only have left to prove item (1).

For this, we look at the slope $m(s_0)$ of the orbit of \eqref{system4}-(a) that passes through a point $(s_0,-s_0)$ in the diagonal $s+t=0$. By \eqref{system4}-(a), it is given by $$m(s_0)=-\sqrt{\frac{q(-s_0)}{q(s_0)}}.$$ We have then that $m(0)=-1$.

On the other hand, a simple computation from \eqref{spo} shows that the function $q(-x)-q(x)$ is negative in $[0,\sqrt{h})$ and positive in $(\sqrt{h},\8)$, where $h=q'(0)>0$. Thus, for small positive values of $s$, we see that $m(s)>-1$, and we have two possibilities:

{\bf Case 1:} ${\rm min}\{r_3,-r_2\} \leq \sqrt{h}$. In that situation, $m(s)> -1$ for every $s\in (0,{\rm min}\{r_3,-r_2\})$.

{\bf Case 2:} ${\rm min}\{r_3,-r_2\} > \sqrt{h}$. This time, $m(s)>-1$ for every $s\in (0,\sqrt{h})$, while $m(s)<-1$ for every $s \in (\sqrt{h},{\rm min}\{r_3,-r_2\})$.

Recall that the orbit $\Gamma_0$ starts above the diagonal $s+t=0$. So, as long as $\Gamma_0$ stays in the phase space \eqref{system4}-(a), it can only cross $s+t=0$ at a point $(s_0,-s_0)$ at which $m(s_0) \leq -1$. Thus, by the previous dichotomy, $\Gamma_0$ can only cross once $s+t=0$ while in the phase space \eqref{system4}-(a). From here and the monotonicity of the orbits of \eqref{system4}-(b), we see that item (1) holds. This completes the proof.
\end{proof}

\begin{remark}\label{rema:orbits}
If in Proposition \ref{pro:orbits} we assume additionally that $r_2+r_3>0$, i.e., that $(r_1,r_2)\in \Omega\setminus \overline\Omega_0$, then $\Gamma_0$ only intersects $s+t=0$ at most once. Indeed, from $r_2+r_3>0$ and the arguments in the proof of Proposition \ref{pro:orbits} we see that the slope $m(s)$ is $>-1$ for every $s\in (0,-r_2]$, due to the fact that $m(s)\to 0$ as $s\to -r_2$. In particular, $\Gamma_0$ can only intersect $s+t=0$ at most once after reaching the segment  $\cR\cap \{t=r_2\}$. But on the other hand, if $\Gamma_0$ had already intersected $s+t=0$ before hitting this horizontal segment, the condition $m(s)>-1$ prevents it from intersecting $s+t=0$ again. 
\end{remark}

\begin{corollary}\label{cor:orbits}
Let $(\alfa(u),\beta(u))$ be the (unique up to sign) solution to system \eqref{system} with initial conditions $\alfa(0)=\beta(0)=0$. Assume that its associated polynomial $q(x)$ in \eqref{spo}, written as in \eqref{spo2}, satisfies $(r_1,r_2)\in \Omega_0$ with $r_3=1/(r_1 r_2)$. Then, there exists a unique $\tau>0$ such that $\beta(\tau)=0$ and both $\alfa(u)\neq 0$ and $\beta(u)\neq 0$ hold for every $u\in (0,\tau]$.
\end{corollary}

\section{Solution of the system in the degenerate case}\label{sec:dege}

We now solve system \eqref{system2} in that case that $q(x)$ is given by \eqref{spo2} for the (degenerate) case that $r_1=r_2=r<0$. Our aim in doing so is to have an exact control of the orbit $\Gamma_0$ in this situation. 

The fact that $q(x)=-(x-r)^2(x-1/r^2)$ has a double root allows to solve \eqref{system2} in a more explicit way. Specifically, let $H$ be the $C^{\8}$ function in $(-\8,0)\cup (0,1/r^2)$ given by 
\begin{equation}\label{defache}
H(x):= \frac{1+ \sqrt{1- x r^2}}{1- \sqrt{1-x r^2}} \, \left( \frac{\sqrt{1-r^3}-\sqrt{1- x r^2}}{\sqrt{1-r^3}+\sqrt{1-x r^2}}\right)^{\frac{1}{\sqrt{1-r^3}}}.
\end{equation}
The function $-{\rm ln} (H(x))$ is a primitive of $1/(x \sqrt{q(x)})$. Then, the general solution $(s(\landa),t(\landa))$ to system $(a)$ in \eqref{system4} (assume $\ep =1$ for definiteness) is implicitly given by
\begin{equation}\label{solude}
H(s(\landa)) =  c_1 e^{-\landa}, \hspace{1cm} H(t(\landa)) = c_2 e^{-\landa},
\end{equation}
where $c_1,c_2\in \R$. Likewise, the general solution to \eqref{system4}-$(b)$ with $\ep=1$ is
\begin{equation}\label{solude2}
H(s(\landa)) =  c_1 e^{\landa}, \hspace{1cm} H(t(\landa)) = c_2 e^{-\landa}.
\end{equation}

In our present degenerate case, we can still carry out a qualitative analysis of the system following our study of Section \ref{sec:orbits}. There is, however, an important difference in that, due to the existence of the double root of $q(x)$ at $x=r<0$, the function $t(\landa)$ can only approach the value $t(\landa)=r$ as $\landa\to \8$.  This means that the trajectories $(s(\landa),t(\landa))$ of system \eqref{system2} start (for $\landa=-\8)$ at $(0,0)$, they reach in finite time the \emph{wall} $\{(1/r^2,t): t\in (r,0)\}$, and end (for $\landa=\8$) at the point $(0,r)$. See Figure \ref{fig:fasesdeg}. In particular, any orbit of \eqref{system2} that starts at the origin above the diagonal $s+t=0$ must eventually intersect again this diagonal, as it must end at $(0,r)$.

\begin{figure}[h]
\begin{center}
\includegraphics[height=5cm]{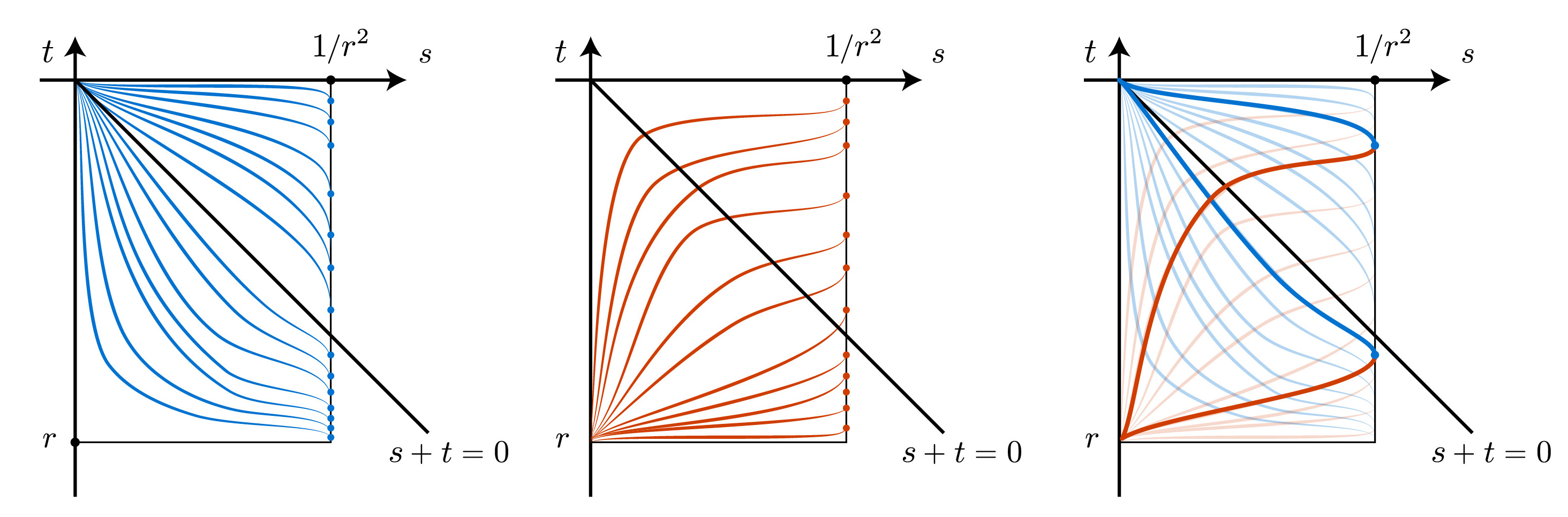} 
 \caption{Trajectories of \eqref{system2} in the $(s,t)$-plane when $r_1=r_2$. If they start initially above the diagonal $s+t=0$ they eventually cross it, maybe after \emph{bouncing} at the right wall.}\label{fig:fasesdeg}
\end{center}
\end{figure}

From now on, we assume the conditions 
\begin{equation}\label{degecon}
r_1=r_2=:r<0, \hspace{0.5cm} r_3=\frac{1}{r^2}, \hspace{0.5cm} r\in (-\sqrt[3]{2},-1],
\end{equation}
which mean that $(r,r)\in \parc \Omega_0$.

Let $\Gamma_0$ denote, as usual, the orbit of \eqref{system2} that corresponds to the initial conditions $\alfa(0)=\beta(0)=0$, for the original system \eqref{system}; see Definition \ref{def:gamma0}. By our analysis of Section \ref{sec:orbits}, $\Gamma_0$ starts from $(0,0)$ and stays initially above $s+t=0$, by \eqref{degecon}. Thus, as we have just discussed, $\Gamma_0$ intersects $s+t=0$ at some point.

\begin{definition}\label{def:dege}
For any $r\in (-\sqrt[3]{2},-1]$, we denote by $\hat{\alfa}=\hat{\alfa}(r)\in (0,1/r^2]$ the value such that $\Gamma_0$ intersects the diagonal $s+t=0$ at $(\hat{\alfa},-\hat{\alfa})$. See Figure \ref{fig:fig8}.
\end{definition}
\begin{figure}[h]
\begin{center}
\includegraphics[height=5.6cm]{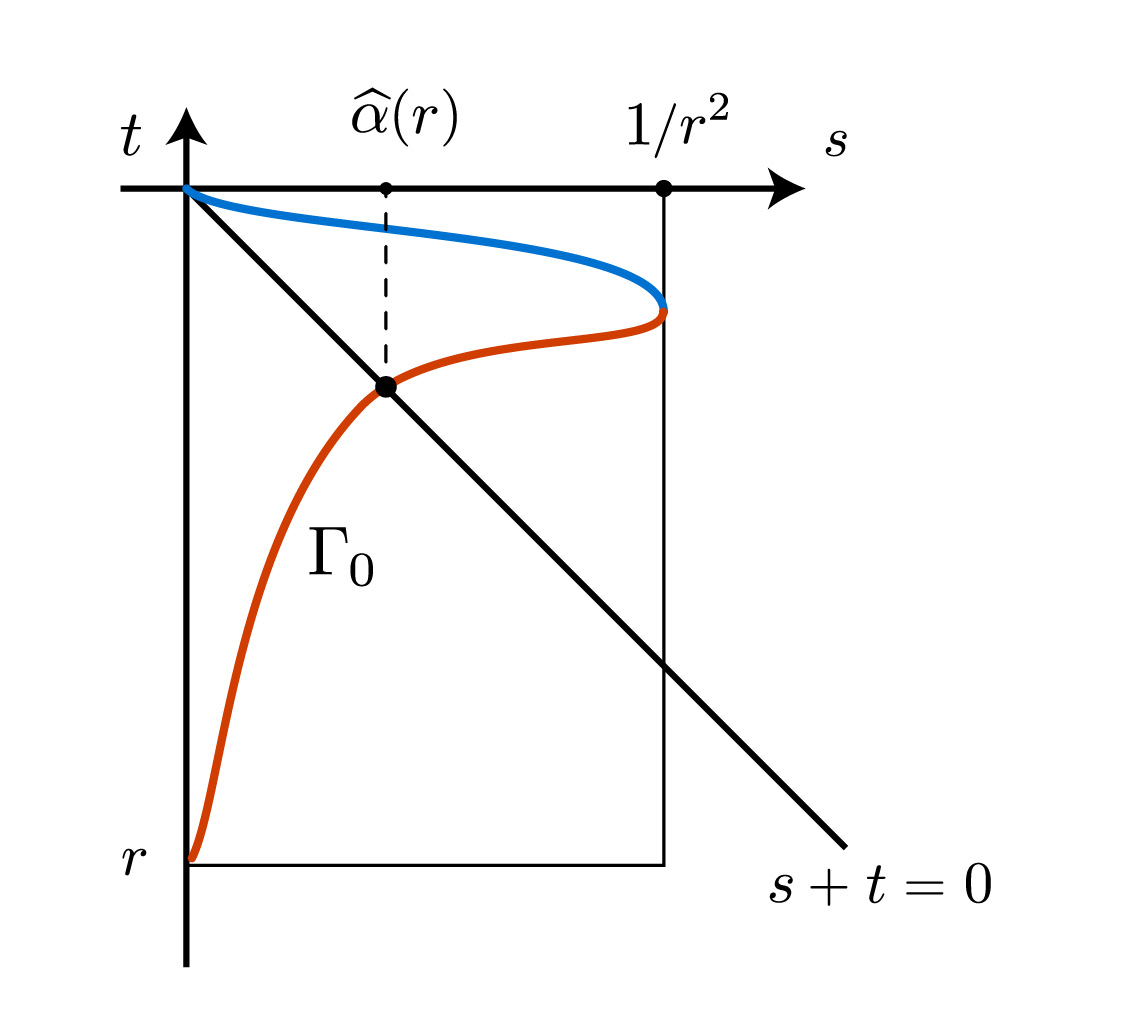} 
 \caption{The trajectory $\Gamma_0$ and the definition of $\hat{\alfa}(r)$.}\label{fig:fig8}
\end{center}
\end{figure}

We next compute the value of $\hat{\alfa}(r)$, and show that it is unique.
\begin{theorem}\label{th:dege}
Given $r\in (-\sqrt[3]{2},-1]$, let us define in terms of $H(x)$ in \eqref{defache} the functions 
\begin{equation}\label{def:hg}
\cF(x):= H(x) H(-x), \hspace{1cm} \cG(x):= \frac{H(x)}{H(-x)},
 \end{equation} 
both of them defined on $(0, 1/r^2].$ 

Also, let $r^{\sharp} \approx -1.155867$ be the unique solution to $\theta(x)=-1$, where 
\begin{equation}\label{defteta}
\theta(x):=-(3+2\sqrt{2})\left( \frac{\sqrt{1-x^3}-\sqrt{2}}{\sqrt{1-x^3}+\sqrt{2}}\right)^{\frac{1}{\sqrt{1-x^3}}}.
\end{equation}
\begin{enumerate}
\item
If $r\in [r^{\sharp},-1]$, then $\hat{\alfa}(r)$ is the unique value such that $\cF(\hat{\alfa}(r))=-1$.
\item
If $r\in (-\sqrt[3]{2},r^{\sharp}]$, then $\hat{\alfa}(r)$ is the unique value such that $\cG(\hat{\alfa}(r))=-1$.
\end{enumerate}
Moreover, if $r=r^{\sharp}$, then $\hat{\alfa}(r)=1/(r^{\sharp})^2$.
\end{theorem}
\begin{proof}
Initially, i.e. for $\landa\approx -\8$, $\Gamma_0$ is an orbit of system \eqref{system4}-$(a)$, with $\ep=1$. By the general solution \eqref{solude} to this system we have the first integral
\begin{equation}\label{first}
\frac{H(s)}{H(t)} = {\rm const}
\end{equation}
along any trajectory. We now look at the orbits of the phase space $\cR_+$ associated to \eqref{system4}-$(a)$ that intersect the diagonal $s+t=0$. For any such point $(s_0,-s_0)$ in the diagonal, let $c(s_0)$ be its corresponding constant $c(s_0):=H(s_0)/H(-s_0)=\cG(s_0)$ for \eqref{first}.
Let $c_0\in \R$ denote the integration constant for $\Gamma_0$ associated to \eqref{first}. Clearly, $c_0$ is the limit as $s_0\to 0$ of the constants $c(s_0)$.  It follows directly from \eqref{defache} that $$c_0= \lim_{s_0\to 0}\frac{H(s_0)}{H(-s_0)}=-1.$$ So, one has $H(s)=-H(t)$ along the orbit $\Gamma_0$, while we stay in the phase space $\cR_+$ of \eqref{system4}-$(a)$, i.e., as long as the orbit $\Gamma_0$ does not hit the \emph{wall} $s=1/r^2$.

The function $\cF(x)$ is strictly increasing, with $\cF(x)\to -\8$ as $x\to 0$, and so its maximum value is $$\cF(1/r^2)= \theta(r),$$ where $\theta(x)$ is given by \eqref{defteta}. The function $\cG(x)$ has a unique critical point, a minimum, at $$x_G= \sqrt{-\frac{1}{r}(2+r^3)},$$ and so the maximum of $\cG$ in $[0,1/r^2]$ is the largest value between $$\cG(0)=\lim_{x\to 0^+} \cG(x)=-1,\hspace{1cm} \cG(1/r^2)= \frac{1}{\theta(r)}.$$ 

{\bf Case 1:} \emph{$r\in (r^{\sharp},-1]$, where $r^{\sharp}:=\theta^{-1}(-1)$ for $\theta(x)$ in \eqref{defteta}.}

The function $\theta(x)$ in \eqref{defteta} is increasing in $[-\sqrt[3]{2},-1]$. Thus, in this Case 1 we have $\theta(r)>-1$, and so $\cG(x)<-1$ for every $x\in [0,1/r^2]$. In particular $H(x)\neq -H(-x)$. Since the relation $H(s)=-H(t)$ holds along the orbit $\Gamma_0$ until it leaves the phase space $\cR_+$, we deduce then that $\Gamma_0$ cannot intersect the diagonal $s+t=0$ while in $\cR_+$.

Let $(1/r^2,t_0)$ be the point of the \emph{wall} of $\cR$ that is reached by $\Gamma_0$. Since $H(1/r^2)=1$ and $H(s)=-H(t)$ along $\Gamma_0$, we have $H(t_0)=-1$.

By the analysis in Section \ref{sec:orbits}, the orbit $\Gamma_0$ extends analytically beyond $(1/r^2,t_0)$ by passing to the phase space $\cR_-$ associated to system \eqref{system4}-$(b)$. Since $\Gamma_0$ must intersect $s+t=0$ and did not do it while in $\cR_+$, there must exist some point 
$(\hat{\alfa},-\hat{\alfa}) \in \cR_-$ met by $\Gamma_0$. But now, the general solution \eqref{solude2} to \eqref{system4}-$(b)$ provides the first integral 
\begin{equation}\label{firs2}H(s) H(t) ={\rm constant}
\end{equation}
along any such trajectory in $\cR_-$. Since $\Gamma_0$ passes through $(1/r^2,t_0)$ and $H(1/r^2)=-H(t_0)=1$, the constant in \eqref{firs2} associated to $\Gamma_0$ is $-1$. Thus, $\cF(\hat{\alfa})=-1$, as claimed in the statement.

{\bf Case 2:} \emph{$r\in (-\sqrt[3]{2},r^{\sharp})$, where $r^{\sharp}:=\theta^{-1}(-1)$ for $\theta(x)$ in \eqref{defteta}.}

The proof is analogous, so we merely sketch it. This time $\theta(r)<-1$, and so $\cF(x)< -1$ for every $x$. This is used to show by means of the previous arguments that $\Gamma_0$ cannot intersect $s+t=0$ after reaching the \emph{wall}, i.e., it must intersect it while still in the phase space $\cR_+$. At that intersection point $(\hat{\alfa},-\hat{\alfa}) \in \cR_+$, using that $H(s)=-H(t)$ along $\Gamma_0$ when in $\cR_+$, we must have $\cG(\hat{\alfa})=-1$, as stated. This finishes the proof of Case 2.

Finally, when $r=r^{\sharp}$, the orbit $\Gamma_0$ intersects $s+t=0$ exactly at the point $(1/r^2,-1/r^2)$. That is, if $r=r^{\sharp}$, we have 
$\hat{\alfa}(r^{\sharp})=1/(r^{\sharp})^2$. 
\end{proof}

\section{Compact minimal annuli with spherical curvature lines}\label{sec:wei}

\subsection{Weierstrass data}\label{subsec:wei}

Fix $(r_1,r_2)\in W$, where $W\subset \R^2$ is given by \eqref{omw}, and define $r_3:=1/(r_1 r_2)$.
Denoting $b:=\frac{1}{3}(r_1+r_2+r_3)$, let $e_1,e_2,e_3\in \R$ be 
\begin{equation}\label{def:ej}
e_j= \frac{b-r_j}{4}, \hspace{1cm} j=1,2,3.
\end{equation}
Note that $e_1+e_2+e_3=0$. Denote \begin{equation}\label{def:gs}
g_2 = -4(e_1 e_2 + e_2 e_3 + e_3 e_1), \hspace{0.5cm} g_3 = 4 e_1 e_2 e_3.
\end{equation}
It follows that $\Delta_{\rm mod} :=g_2^3-27 g_3^2 >0$, and so these values define a rectangular lattice $\Lambda$ in $\C$.

Let $\wp$ be the Weierstrass $P$-function associated to $\Lambda$. It satisfies 
 \begin{equation}\label{odepw}
\wp'^2 = 4 \wp^3- g_2 \wp - g_3 = 4(\wp -e_1)(\wp -e_2)(\wp-e_3).
\end{equation} 
For all the properties of $\wp$ and other associated Weierstrass elliptic functions that will be used in this paper, see e.g. \cite{WW,Ch}.

Choose generators $2\omega_1\in \R$ and $2\omega_2 \in i\R$ for the rectangular lattice $\Lambda$, so that $\omega_1>0$, and ${\rm Im}(\omega_2)>0$. At the half-periods, we have 
\begin{equation}\label{wehalf}
\wp(\omega_1)=e_1>0, \hspace{0.4cm} \wp(\omega_2)= e_3<0, \hspace{0.4cm} \wp(\omega_1+\omega_2)=e_2\in (e_3,e_1).
\end{equation}
Also, $\wp'(\omega_1)=\wp'(\omega_2)=\wp'(\omega_1+\omega_2)=0$. By construction, $b$ satisfies \eqref{cubicb} with respect to $g_2,g_3$.

Define the doubly periodic meromorphic function 
\begin{equation}\label{defF}
\phi(z):= b- 4 \wp (z+ \omega_1).
\end{equation}
Note that $\phi$ is real along both $\R$ and $i\R$, and it has no poles in the strip $|{\rm Re}(z)|<\omega_1$. Also, $\phi$ has poles at the points of the lattice $\mathcal{Z}:=\Lambda - \omega_1$ for which it is periodic. The values of $\phi$ along $i\R$ are contained in the interval $[r_1,r_2]$, with $\phi(0)=r_1$ and $\phi(\omega_2)=r_2$. In particular, $\phi(z)<0$ for every $z\in i\R$.

\begin{remark}
An alternative expression for $\phi$ in terms of $r_1,r_2,r_3$ is
\begin{equation}\label{Falt}
\phi(z) = r_1 + \frac{(r_1-r_3)(r_1-r_2)}{b-r_1-4 \wp(z)}.
\end{equation}
This follows from the general formula (\cite[p. 444]{WW}) $$\wp(z+\omega_1)=e_1 +\frac{(e_1-e_2)(e_1-e_3)}{\wp(z)-e_1}.$$
\end{remark}

\begin{lemma}\label{odeF}
$\phi$ satisfies the differential equation $\phi'(z)^2 = \hat{q}(\phi(z))$, where
\begin{equation}\label{pora}
\hat{q}(x):= -(x -r_1)(x-r_2)(x-r_3).
\end{equation}
\end{lemma}
\begin{proof}
It follows after a computation from \eqref{Falt}, \eqref{odepw}, and the definition of the constants. 
\end{proof}

Define next \begin{equation}\label{w6}
g(z) = \hat{g}_0 \, {\rm exp} \left(\int_0^z \frac{1}{\phi(\nu)} d\nu \right),
\end{equation}
where $\hat{g}_0>0$ is a positive constant. So, $g$ is the unique solution to $g/g'=\phi$ with $g(0)=\hat{g}_0$. 

Denoting $z=u+iv$, define the minimal surface $\Sigma=\Sigma(r_1,r_2;\hat{g}_0)$ given by
\begin{equation}\label{weipar}
\psi(u,v)\equiv \psi(z) = 2 {\rm Re} \int_0^z \Phi(z) dz : \cU \subset \R^2 \flecha \R^3,
\end{equation}
where $\cU=\{(u,v): |u|<\omega_1\}$ and $\Phi$ is given by \eqref{weidata} in terms of $\phi$ in \eqref{defF} and $g$ in \eqref{w6}. Note that $\psi(0,0)$ is the origin of $\R^3$. Since $\phi$ has no poles on $\cU$, $\psi$ is a conformally immersed minimal surface, with Gauss map $g$. 

From our discussion in Section \ref{sec:prelim} we have:

\begin{lemma}\label{lem:sigen}
The surface $\Sigma$ given by \eqref{weipar} has the following properties:
\begin{enumerate}
\item
Each curve $\psi(u_0,v)$ is a curvature line contained in some sphere $S(c(u_0),R(u_0))$. 
\item
All the centers $c(u)$ lie in a common vertical line $L$ of $\R^3$.
\item
$\psi(0,v)$ is a planar curvature line lying in the $x_3=0$ plane, and for every other $u_0\in (-\omega_1,\omega_1)$, with $u_0\neq 0$, the curvature line $\psi(u_0,v)$ lies in a sphere of finite radius $R(u_0)$.
\item
$\Sigma$ is symmetric with respect to the $x_2=0$ plane.
 \item
$\Sigma$ intersects the $x_3=0$ along $\psi(0,v)$ with a constant angle $\theta$ given by 
\begin{equation}\label{anginfor}
\cos \theta=\frac{\hat{g}_0^2-1}{\hat{g}_0^{2}+1}.
\end{equation}
\item
If $\hat{g}_0=1$, i.e. $\theta=\pi/2$, then $\Sigma$ is symmetric with respect to the $x_3=0$ plane.
\end{enumerate}
\end{lemma}
\begin{proof}
Items (1) and (2) hold by Wente \cite{W}, see our discussion in Section \ref{sec:prelim}. Note that we have made a translation $z\mapsto z+\omega_1$ in the conformal parameter $z=u+iv$ with respect to the formulas presented in Section \ref{sec:were}, but this does not affect the properties of $\Sigma$ detailed there. Since $\phi$ is real along $i\R$, it follows by \eqref{weidata} that $\psi(0,v)$ lies in the $x_3=0$ plane. Since the imaginary part of $\phi$ is never identically zero along $u_0 + i\R$ for any $u_0\in (-\omega_1,\omega_1)$ with $u_0\neq 0$, it follows that $\psi(u_0,v)$ does not lie in a horizontal plane, and so $R(u_0)$ is finite. This proves item (3).

Item (4) follows directly from \eqref{weidata}, since $\phi,g$ are real along $\R$. Also, by \eqref{w6}, we have that $|g|=\hat{g}_0$ along $i\R$, and so \eqref{anginfor} holds. Item (6) follows again from \eqref{weidata}, since if $\hat{g}_0=1$, we must have $\phi(z)\in \R$ and $|g(z)|=1$ along $i\R$.
\end{proof}

In Lemma \ref{gex} below we give an explicit expression of the Gauss map $g$ in terms of the Weierstrass zeta and sigma functions. Recall that these classical functions satisfy $\zeta'(z)=-\wp(z)$ and $\sigma'(z)/\sigma(z) =\zeta(z)$.

It is a classical property of Weierstrass functions that $\wp(z)$ is real and injective along the boundary of the rectangle generated by the half-periods $\omega_1,\omega_2$. Thus, there exists exactly one value $\mu$ in the boundary of that rectangle where $4\wp(\mu)=b$. By \eqref{def:ej}, and taking into account the values of $\wp$ at the half-periods, see \eqref{wehalf}, we have $$4 \wp(\omega_1+\omega_2) - b = -r_2>0, \hspace{0.5cm} 4\wp(\omega_2)-b = -r_3 <0.$$
So, $\mu$ lies in the horizontal segment between $\omega_2$ and $\omega_1+\omega_2$. 
The function $\wp(z)$ is increasing along that segment, and hence, from \eqref{odepw} and \eqref{cubicb} we have $\wp'(\mu)=1/4$.

\begin{lemma}\label{gex}
Let $\mu\in \C$ be the unique number of the form $\mu=x+\omega_2$, with $x\in (0,\omega_1)$, such that $b = 4\wp(\mu)$. Then, $g$ in \eqref{w6} can be written alternatively as 
\begin{equation}\label{w80}
g(z):= g_0 \, {\rm exp}\left(-2\zeta(\mu) \left( z + \omega_1\right) \right) \left[ \frac{\sigma(\mu+  \left( z + \omega_1\right))}{\sigma(\mu- \left( z + \omega_1\right))}\right]
\end{equation}
where the constant $g_0\in \R$ is given by 
 \begin{equation}\label{w100}
 g_0:= -\hat{g}_0 \, {\rm exp} \big( 2( \omega_1  \zeta(\mu) - \mu \, \zeta(\omega_1))\big) \in \R.
 \end{equation}
\end{lemma}
\begin{proof}
Differentiating \eqref{w80} we have
$$\frac{g'(z)}{g(z)} = -2 \zeta(\mu) + \zeta(\mu-\left( z + \omega_1\right) ) + \zeta(\mu+\left( z + \omega_1\right) ).$$ We now use the general identity 
\begin{equation}\label{zetarel}
\zeta(z_1+z_2) + \zeta(z_1-z_2)= 2\zeta(z_1) + \frac{\wp'(z_1)}{\wp(z_1)-\wp(z_2)}
\end{equation}
with $z_1:=\mu$ and $z_2:= z+ \omega_1$, to obtain $$\frac{g'(z)}{g(z)} = \frac{\wp'(\mu)}{\wp(\mu)-\wp(z+\omega_1)}.$$ Using that $\wp'(\mu)=1/4$ and $b=4\wp(\mu)$, we deduce from there that $$\frac{g(z)}{g'(z)}= b - 4 \wp (z+ \omega_1)=\phi(z).$$
We now check the value of $g_0$ in order to have $g(0)=\hat{g}_0$. We use the following property of the Weierstrass $\sigma$ function, for $j,k\in \Z$:

\begin{equation}\label{forsigma}\frac{\sigma(z+ 2 j \omega_1 + 2 k \omega_2)}{\sigma(z)} = (-1)^{j + k + jk} {\rm exp} \big[ (2 j \zeta(\omega_1) + 2 k \zeta(\omega_2))(j \omega_1 + k \omega_2 + z)\big].
\end{equation}

If in \eqref{forsigma} we choose $z= \mu- \omega_1$ and $j=1$, $k=0$, we get 
\begin{equation}\label{ffsig}
 \frac{\sigma(\mu+ \omega_1)}{\sigma(\mu- \omega_1)} = - {\rm exp} (2 \mu\,  \zeta(\omega_1)).
 \end{equation} 
Thus, from \eqref{w80}, we obtain \eqref{w100}.
\end{proof}

\subsection{Period and symmetries of minimal annuli}

We now consider the period problem for $\Sigma$ along $i\R$, i.e., we discuss when are the spherical curvature lines $\psi(u_0,v)$ of $\Sigma$ periodic. We show in Lemma \ref{lemper} below that this is controlled by the number $\kappa\in \R$ given by
\begin{equation}\label{deper}
\kappa:= \frac{-1}{2\pi i} \int_0^{2\omega_2} \frac{1}{\phi(\nu)} d\nu.
\end{equation}
Note that, by \eqref{w6}, $\kappa$ measures the variation of $g(z)$ as $z$ varies from $0$ to $2\omega_2\in i\R$: recall that $|g(z)|=\hat{g}_0$, constant, if $z\in i\R$. Also, we observe that if $\psi(u,v)$ is periodic in the $v$-direction, this period must be a multiple of ${\rm Im}(2\omega_2)$, by the periodicity of $\phi$ and \eqref{weidata}.

\begin{lemma}\label{lemper}
The following three conditions are equivalent, for any $n\in \N^*$:
\begin{enumerate}
\item
$\psi(u+iv)=\psi(u+iv+2n \omega_2)$.
\item
$g(2n \omega_2)=g(0).$
\item 
There is some $m\in \Z$ such that $\kappa=m/n\in \Q$.
\end{enumerate}
\end{lemma}
\begin{proof}
If (1) holds, then obviously $g(2n\omega_2)=g(0)$. Besides, from \eqref{w6}, we obtain directly 
\begin{equation}\label{interm}
g(2n\omega_2)=g(0) {\rm exp}(-2 \pi i \kappa n).
\end{equation}
Thus, $g(2n\omega_2)=g(0)$ is clearly equivalent to $\kappa=m/n$ for some $m\in \Z$, i.e., items (2) and (3) are equivalent.

Finally, we show that if (2) or (3) hold, then (1) must also hold. Indeed, in that case the Weierstrass data $\Phi$ in \eqref{weidata} satisfy $\Phi(z+2n\omega_2)=\Phi(z)$, and so, by \eqref{weipar}, $\psi(u+iv+2n \omega_2)=\psi(u+iv)+w_0$ for some translational period $w_0\in \R^3$. Since the curvature lines $\psi(u_0,v)$ are spherical (and spheres are compact), we have $w_0=0$, and hence (1) follows. 
\end{proof}

\begin{remark}
The value $\kappa$ in \eqref{deper} does not depend on the choice of $\hat{g}_0$.

\end{remark}

The quantity $\kappa$ in \eqref{deper} allows to control not only the period problem for $\psi(u,v)$ in the $v$-direction as explained in Lemma \ref{lemper}, but also the symmetries of the resulting minimal annuli when the period closes.


Specifically, with our current notations, assume that $\Sigma=\Sigma(r_1,r_2,\hat{g}_0)$ satisfies $\kappa=\frac{m}{n}$, where $n,m\in \N^*$ and $m/n$ is irreducible. Consider $\Sigma$ parametrized by $\psi(u,v):(-\omega_1,\omega_1)\times \R\flecha \R^3$ as in \eqref{weipar}. Then, its quotient by the isomorphism 
\begin{equation}\label{quotient}
(u,v)\mapsto (u,v+ 2 n {\rm Im}(\omega_2))
\end{equation}
is an immersed minimal annulus, by Lemma \ref{lemper}. \emph{We denote it by $\Sigma^*$}. We have from Lemma \ref{lem:sigen} and the Weierstrass representation that $\Sigma^*$ is symmetric with respect to:
\begin{enumerate}
\item
The plane $x_2=0$.
\item
The plane $x_3=0$, if $\hat{g}_0=1$, i.e., if $\beta(0)=0$.
\item
The rotations of angles $2\pi k/n$, with $k\in \{0,\dots, n-1\}$, around the vertical line $L$ of $\R^3$ that contains the centers of the spheres $S(c(u),R(u))$.
\end{enumerate}
Therefore, we have:

\begin{corollary}\label{rem:symmetries}
The symmetry group of the minimal annulus $\Sigma^*$ is isomorphic to the dihedral group $D_n$ if $\hat{g}_0\neq 1$, and to $D_n\times \Z_2$ if $\hat{g}_0=1$.
\end{corollary}
\begin{proof}
It suffices to show that $\Sigma^*$ is only invariant with respect to the group of isometries of $\R^3$ generated by those listed above. Let $\gamma=\Sigma^*\cap \{x_3=0\}$ be the central planar curvature line of $\Sigma^*$. Since the $u$-curvature lines of $\Sigma^*$ are not closed, and $\gamma$ is the only $v$-curvature line that is planar, we deduce that any isometry of $\R^3$ that leaves $\Sigma^*$ invariant also leaves $\gamma$ invariant. Since $\gamma$ has a line of symmetry in the $x_3=0$ plane and it is not a circle, its (planar) symmetry group is a dihedral group $D_{n'}$. Thus, $n'$ must be a divisor of $n$ and by the argument in Lemma \ref{lemper}, we have that $\kappa=m'/n'$ for some $m'\in \Z$. But since $\kappa=m/n$ irreducible, we obtain $n=n'$ and the result follows.
\end{proof}
In the case $\hat{g}_0=1$, $\Sigma^*$ is invariant with respecto to a prismatic group of order $2n$; see \cite{CSW} for a description of such groups.

In addition, the numerator $m$ of $\kappa=\frac{m}{n}$ describes the number of times that the planar curvature line $\gamma=\Sigma^*\cap \{x_3=0\}$ wraps around. This can be seen as follows. First, note that the Gauss map $g$ of $\Sigma^*$ maps $\gamma\equiv \psi(0,v)$ into the circle $|z|=\hat{g}_0$ of $\C$, and that $g'(iv)\neq 0$ for any $v$, since $\phi(iv)\neq \8$. Thus, $g$ defines a regular covering map of this circle. The value $m$ describes the degree of this map. This implies that $m$ gives the rotation index of the unit normal of the (locally convex) planar curve $\gamma$.

\subsection{The free boundary condition}

We next study the intersection angle of $\Sigma=\Sigma(r_1,r_2;\hat{g}_0)$ with the family of spheres $S(c(u),R(u))$. As a consequence of Proposition \ref{pro:orbits} we have:
\begin{proposition}\label{taucero}
Assume that $(r_1,r_2)\in \Omega_0\subset W$, and that $\hat{g}_0=1$. Then, there exists a value $\tau\in (0,\omega_1)$ such that the surface $\Sigma$ intersects the sphere $S(c(\tau),R(\tau))$ orthogonally along the curvature line $\psi(\tau,v)$.
\end{proposition}
\begin{proof}
Let $(\alfa(u),\beta(u))$ be the solution to the system \eqref{system} associated to $\Sigma$. Since $\Sigma$ intersects orthogonally the $x_3=0$ plane along the curve $u=0$, we have $\alfa(0)=\beta(0)=0$, and so the polynomial $p(u,X)$ in \eqref{def:p} is 
 \begin{equation}\label{ecp}
 p(0,X)=-4\alfa'(0) X^3 - 4\delta X^2 + 4 \beta'(0)X - 4.
 \end{equation}
The metric $e^{2\omega}|dz|^2$ of $\Sigma$ is given in terms of the Weierstrass data $(\phi,g)$ by 
 \begin{equation}\label{eome}
 e^{\omega}= \frac{|\phi|}{2}(|g|+ |g|^{-1}).
 \end{equation} 
In particular, we have $e^{\omega(0,v)} =-\phi(i v),$ since $\phi<0$ and $|g|=1$ along $i\R$. Now, differentiating this equation with respect to $v$ and using \eqref{pos}, we obtain from \eqref{ecp} the differential equation 
\begin{equation}\label{odf}
\phi'(z)^2 = -\alfa'(0) \phi(z)^3 + \delta \phi(z)^2+ \beta'(0) \phi(z) +1.
\end{equation}
But on the other hand, $\phi$ satisfies $\phi'(z)^2 =\hat{q}(\phi(z))$, by Lemma \ref{odeF}. Thus, 
\begin{equation}\label{alfa01}
\alfa'(0)=1, \hspace{0.5cm} \delta =\frac{\hat{q}''(0)}{2}, \hspace{0.5cm} \beta'(0)=\hat{q}'(0),
\end{equation} 
since $\phi$ is not constant. By \eqref{def:h} and \eqref{def:k}, the constants $h,k$ associated to $(\alfa(u),\beta(u))$ are $h=\beta'(0)$ and $k=1$. This implies from \eqref{alfa01} that $q(x)=\hat{q}(x)$, where $q(x)$ is the polynomial constructed from $h,k,\delta$ in \eqref{spo}.

In particular, we can apply the analytic study of Section \ref{sec:orbits} to our geometric situation. By our present hypotheses on $\hat{q}(x)$, the polynomial $q(x)=\hat{q}(x)$ associated to $\Sigma$ is in the conditions of Proposition \ref{pro:orbits}. In this way, by Corollary \ref{cor:orbits}, there exists a value $\tau>0$ where $\beta(\tau)=0$, and both $\alfa(u)\neq 0$ and $\beta(u)\neq 0$ hold for every $u\in (0,\tau)$. In particular, $\tau\in (0,\omega_1)$, since $\alfa(\omega_1)=0$. This proves the result.
\end{proof}
\begin{definition}\label{sita}
Assume that $(r_1,r_2)\in \Omega_0\subset W$, and that $\hat{g}_0=1$. We denote by $\Sigma_{\tau}$ the closed subset of $\Sigma$ parametrized by $\psi(u,v)$ as in \eqref{weipar}, with $(u,v)\in [-\tau,\tau]\times \R$, where $\tau\in (0,\omega_1)$ is the number in Proposition \ref{taucero}. 

Note that, by construction, $\tau$ is the smallest number in $(0,\omega_1)$ at which $\beta=0$.
\end{definition}
The boundary of $\Sigma_{\tau}$ consists of two curves that intersect orthogonally two spheres $S(c(\tau),R(\tau))$ and $S(c(-\tau),R(\tau))$. By the symmetry of $\Sigma_{\tau}$ with respect to the $x_3=0$ plane along $u=0$, these spheres have the same radius $R(\tau)>0$, and their centers $c(\tau),c(-\tau)\in \R^3$ are symmetric with respect to the $x_3=0$ plane.

\begin{theorem}\label{altufor}
The height $c_3(\tau)$ of the center of the sphere $S(c(\tau),R(\tau))$ is
\begin{equation}\label{centro}
c_3(\tau)=b \tau + 4 \zeta(\tau) + \frac{2\wp'(\tau)}{\wp(\tau)-e_1} - \frac{(1+g(\tau)^2) \, \phi(\tau)^2}{(1+g(\tau)^2)\phi'(\tau) + g(\tau)^2 -1} . 
\end{equation}
\end{theorem}
\begin{proof}
For any fixed $u\in (0,\omega_1)$, the center $c(u)$ of the sphere $S(c(u),R(u))$ where $\psi(u,v)$ lies is given by \eqref{gencen}. As $\phi<0$ and $g>0$ in $(0,\omega_1)$, it holds by \eqref{eome} 
\begin{equation}\label{eome2}
e^{\omega}=\frac{-\phi (1+g^2)}{2g}
\end{equation} along that interval. By differentiating with respect to $u$ this expression, we find 
 \begin{equation}\label{centro5}
\omega_u (u,0)= \frac{\phi'(u)+N_3(u)}{\phi(u)}, \hspace{1cm} N_3(u):=\frac{g(u)^2-1}{g(u)^2+1}.
\end{equation} 
Thus, noting that $(\psi_3)_u(u,0)=\phi(u)$ by \eqref{weidata} and \eqref{weipar}, we have the expression for the third coordinate $c_3(u)$ of \eqref{gencen}:
\begin{equation}\label{3corn}
c_3(u)= \int_0^{u} \phi(\nu) d\nu+ \frac{4}{\alfa(u)}\left( \frac{g(u)}{1+g(u)^2}\right) - \frac{\beta(u)}{\alfa(u)} \left(\frac{g(u)^2-1}{g(u)^2+1}\right).
\end{equation}
Now, from \eqref{defF},
\begin{equation}\label{33corn} \int_0^{u} \phi(\nu) d\nu= b  u + 4(\zeta(u+\omega_1)-\zeta(\omega_1)) = b  u + 4 \zeta(u) +  \frac{2\wp'(u)}{\wp(u)-e_1},
\end{equation}
 where for the second equality we have used the identity 
$$\zeta(z_1+z_2)-\zeta(z_1)-\zeta(z_2)=\frac{1}{2} \cdot \frac{\wp'(z_1)-\wp'(z_2)}{\wp(z_1)-\wp(z_2)}$$ applied to $z_1=u$, $z_2=\omega_1$.

We now take $u=\tau$. Since $\beta(\tau)=0$, we can obtain the value of $\alfa(\tau)$ in terms of $g, \phi$ from \eqref{centro5} and \eqref{om1}. Using this, we have from \eqref{3corn} and \eqref{33corn} that
\begin{equation}\label{heta}
\def\arraystretch{2.8} \begin{array}{lll} c_3(\tau) & =& \displaystyle \int_0^{\tau} \phi(\nu) d\nu - \frac{(1+g(\tau)^2) \, \phi(\tau)^2}{(1+g(\tau)^2)\phi'(\tau) + g(\tau)^2 -1}\\ & = & \displaystyle b \tau + 4 \zeta(\tau) + \frac{2\wp'(\tau)}{\wp(\tau)-e_1} - \frac{(1+g(\tau)^2) \, \phi(\tau)^2}{(1+g(\tau)^2)\phi'(\tau) + g(\tau)^2 -1} . 
\end{array}
\end{equation}
This completes the proof.
\end{proof}

\begin{remark}\label{rem:fb}
If $c_3(\tau)=0$, $\Sigma_{\tau}$ is a free boundary minimal strip in the sphere $S(c(\tau),R(\tau))$, possibly with a non-compact image (the period might not close, since $\kappa$ in \eqref{deper} can be irrational).
\end{remark}

The next proposition follows quite directly from Wente \cite{W}:

\begin{proposition}\label{monocen}
The height $c_3(u)$ of the center of the spheres $S(c(u),R(u))$ defines a strictly decreasing diffeomorphism from $(0,\omega_1)$ to $(-\8,\8)$.
\end{proposition}
\begin{proof}
By equation (4.14) of \cite{W}, we have 
\begin{equation}\label{wentecenter}
c'(u)= \frac{1}{\alfa^2} (A e^{-\omega} \psi_u + B e^{-\omega}\psi_v + C N) =: \frac{1}{\alfa^2} {\bf F}
\end{equation}
where $$A=2 \alfa' + \alfa^2 e^{\omega} - \alfa \beta e^{-\omega},\hspace{0.5cm} B= 2\alfa \omega_v, \hspace{0.5cm} C= 2\alfa e^{-\omega} + \alfa'\beta- \beta'\alfa,$$ and ${\bf F}= {\bf F}(u)=(0,0,F_3(u))$, since all the centers $c(u)$ lie in the same vertical line. Moreover, also from \cite{W}, 
\begin{equation}\label{wence}
\esiz {\bf F},{\bf F}\esde = F_3(u)^2 = 4k,
\end{equation} 
where $k\in \R$ is the Hamiltonian constant in \eqref{def:k}. In our case, $k=1$, and so we have from \eqref{wentecenter} 
\begin{equation}\label{c3u}
c_3'(u)^2 = \frac{4}{\alfa(u)^2}.
\end{equation} In order to check the sign of $c'(u)$, we note that at $u=0$ we have $\alfa(0)=0$ and $\alfa'(0)=1$. This shows that $A=2$ and $B=C=0$ at $(u,v)=(0,0)$. Now, since the third coordinate of $\psi_u(0,0)$ is $\phi(0)<0$, we deduce from \eqref{wentecenter} that $c_3'(u)<0$.

From here, we have since $\alfa(0)=0$, that $c_3(u)\to \8$ as $u\to 0$. Similarily, since $\alfa(\omega_1)=0$ (recall that $\phi$ has a pole at $\omega_1$), we have $c_3(u)\to -\8$ as $u\to \omega_1$. This completes the proof.
\end{proof}

It is important to observe that, for any $u\in (0,\omega_1)$, all the fundamental quantities $\alfa(u),\beta(u),c_3(u)$ can be explicitly computed in terms of the Weierstrass elliptic functions $\wp,\zeta,\sigma$. We explain this next.

Fix $u\in (0,\omega_1)$, and consider $v_0\neq 0$ such that $\omega(u,0)\neq \omega(u,v_0)$.  Then, by \eqref{om1}, 
\begin{equation}\label{sisome}
\left\{\def\arraystretch{1.7} \begin{array}{lll} 2\omega_u (u,0) &=& \alfa(u) e^{\omega(u,0)} + \beta(u) e^{-\omega(u,0)} \\ 2\omega_u (u,v_0) &=& \alfa(u) e^{\omega(u,v_0)} + \beta(u) e^{-\omega(u,v_0)}\end{array}. \right.
\end{equation}
Thus,
\begin{equation}\label{albe0}
\alfa (u)= \frac{2(\omega_u(u,0) e^{\omega(u,0)} - \omega_u(u,v_0) e^{\omega (u,v_0)})}{e^{2\omega(u,0)}-e^{2\omega(u,v_0)}}
\end{equation}
and 
\begin{equation}\label{albe}
\beta (u)= \frac{2(\omega_u(u,0) e^{-\omega(u,0)} - \omega_u(u,v_0) e^{-\omega (u,v_0)})}{e^{-2\omega(u,0)}-e^{-2\omega(u,v_0)}}.
\end{equation}
Now, the function $\omega(u,v)$ is explicitly given by \eqref{eome} in terms of $\phi,g$, which are themselves given in terms of $\wp,\zeta,\sigma$ by \eqref{Falt} and \eqref{w80}. So, we obtain an explicit expression for $\alfa,\beta$ in terms of these elliptic functions. Also, since the function $c_3(u)$ is given by \eqref{3corn} and \eqref{33corn}, the same is true for it. This proves our claim.

\section{The period map}\label{sec:period}

We now vary the values $(r_1,r_2)\in W$ that we fixed at the start of Section \ref{sec:wei}. Below, we follow the notations in that section.
\begin{definition}
We define the \emph{period map} ${\rm Per}(r_1,r_2):W\flecha (0,1)$ as 
\begin{equation}\label{permap}
{\rm Per}(r_1,r_2)=\kappa=\frac{-1}{2\pi i} \int_0^{2\omega_2} \frac{d\nu}{\phi(\nu)}, 
\end{equation}
see \eqref{deper}. By construction, ${\rm Per}(r_1,r_2)$ depends analytically on $(r_1,r_2)$.
\end{definition}

Our next objective is to control the level sets of the map ${\rm Per}$. We first prove:

\begin{theorem}\label{th:cerouno}
For any $(r_1,r_2)\in W$ we have 
\begin{equation}\label{peres0}
\frac{1}{\sqrt{1-r_1^2 r_2}} \leq {\rm Per} (r_1,r_2) \leq \frac{1}{\sqrt{1-r_1 r_2^2}}.
\end{equation} 
In particular, ${\rm Per}(r_1,r_2)\in (0,1)$, and ${\rm Per}(r_1,r_2)$ extends $C^1$-smoothly to the diagonal $\cD:=\{(r,r):r<0\}$, with values 
 \begin{equation}\label{per:boundary}
 {\rm Per} (r,r)= \frac{1}{\sqrt{1-r^3}}.
 \end{equation} 
\end{theorem}
\begin{proof}
Recall that the function $\phi(z)$ in \eqref{defF} is real on $i\R$, with $\phi(0)=\phi(2\omega_2)=r_1$ and $\phi(\omega_2)=r_2$. Moreover, $\phi$ is a strictly increasing (resp. decreasing) diffeomorphism from the segment $[0,\omega_2]$ (resp. $[\omega_2,2\omega_2]$) into $[r_1,r_2]$. Then, we can make in the integral expression of ${\rm Per}$ in \eqref{permap} the change of variable $t=\phi(\nu)$ on each of these two segments, and we have by Lemma \ref{odeF}
\begin{equation}\label{perdire}\def\arraystretch{2.9}\begin{array}{lll} 
{\rm Per}(r_1,r_2) & =  &\displaystyle \frac{-1}{2\pi i} \left\{ \int_0^{\omega_2} \frac{1}{\phi(\nu)} d\nu + \int_{\omega_2}^{2\omega_2} \frac{1}{\phi(\nu)} 
d\nu \right\} \\
& = & \displaystyle \frac{-1}{\pi} \int_{r_1}^{r_2} \frac{1}{t\sqrt{-q(t)}} dt,
\end{array}
\end{equation}
where $q(x)=\hat{q}(x)$ is given by \eqref{spo2}. Writing $t=r_1 + (r_2-r_1)s$ we obtain the more convenient expression
\begin{equation}\label{perotra}
{\rm Per}(r_1,r_2)=\int_0^1 Q(s;r_1,r_2)ds, 
\end{equation}
where
\begin{equation}\label{Qper}
\def\arraystretch{2.5}\begin{array}{lll} Q(s;r_1,r_2) & =&  \displaystyle \frac{r_1-r_2}{\pi(r_1+(r_2-r_1)s)\, \sqrt{-q(r_1+(r_2-r_1)s)}} \\ 
& = & \displaystyle \frac{-\sqrt{r_1 r_2}}{\pi \sqrt{s (1-s)} (r_1+(r_2-r_1)s)\, \sqrt{1- r_1 r_2((r_1+(r_2-r_1)s)}}.
\end{array}
\end{equation}
The function $Q(s,r_1,r_2)$ is positive. Also, 
\begin{equation}\label{peres1}
\frac{1}{\sqrt{1-r_1^2 r_2}} \leq \frac{1}{ \sqrt{1- r_1 r_2(r_1+(r_2-r_1)s)}}\leq \frac{1}{\sqrt{1-r_2^2 r_1}},
\end{equation}
since it can be easily checked that the function in the middle of these inequalities is increasing. So, by \eqref{perotra}, \eqref{Qper} and the first inequality in \eqref{peres1},
\begin{equation}\label{peres11}{\rm Per}(r_1,r_2)\geq \frac{1}{\pi \sqrt{1-r_1^2 r_2}} \int_0^1 \frac{-\sqrt{r_1 r_2}}{\sqrt{s (1-s)} (r_1+(r_2-r_1)s)} \, ds.
\end{equation} The function $$\varphi(s):= 2 \,{\rm arctan}\left(\sqrt{\frac{r_2 s}{r_1(1-s)}}\right)$$ is a primitive of the above integral. Thus, from \eqref{peres11}, we obtain the left inequality of expression \eqref{peres0}. Operating in the same way using the second inequality of \eqref{peres1}, we obtain \eqref{peres0}.

From \eqref{peres0} we see directly that, as $(r_1,r_2)\to (r,r)$, 
$${\rm Per}(r_1,r_2) \longrightarrow  \frac{1}{\sqrt{1-r^3}}.$$ This proves the continuous extension of ${\rm Per}(r_1,r_2)$ to the diagonal $\cD$, as well as \eqref{per:boundary}. In particular, ${\rm Per}(r,r)$ is an increasing diffeomorphism from $(-\8,0)$ to $(0,1)$.

Finally, by differentiating \eqref{Qper} with respect to $r_1$, we have that as $(r_1,r_2)\to (r,r)$,
$$\frac{\parc {\rm Per}(r_1,r_2)}{\parc r_1} \longrightarrow -\frac{1}{\pi}\int_0^1 \frac{1-3r^3+s (3 r^3 -2)}{2 r (1-r^3)^{3/2} \sqrt{s (s-1)}} \, ds = \frac{3 r^2}{4 (1-r^3)^{3/2}}.$$ By an analogous computation, the derivative of ${\rm Per}(r_1,r_2)$ with respect to $r_2$ converges to the same value as $(r_1,r_2)\to (r,r)$. Thus, ${\rm Per}(r_1,r_2)$ extends $C^1$ to the diagonal $\cD$. 

Note that, by the implicit function theorem, from each $(r,r)\in \cD$ starts a unique regular, real analytic level curve of ${\rm Per}$, that intersects $\cD$ orthogonally at $(r,r)$.
\end{proof}

\begin{figure}[h]
\begin{center}
\includegraphics[height=6.2cm]{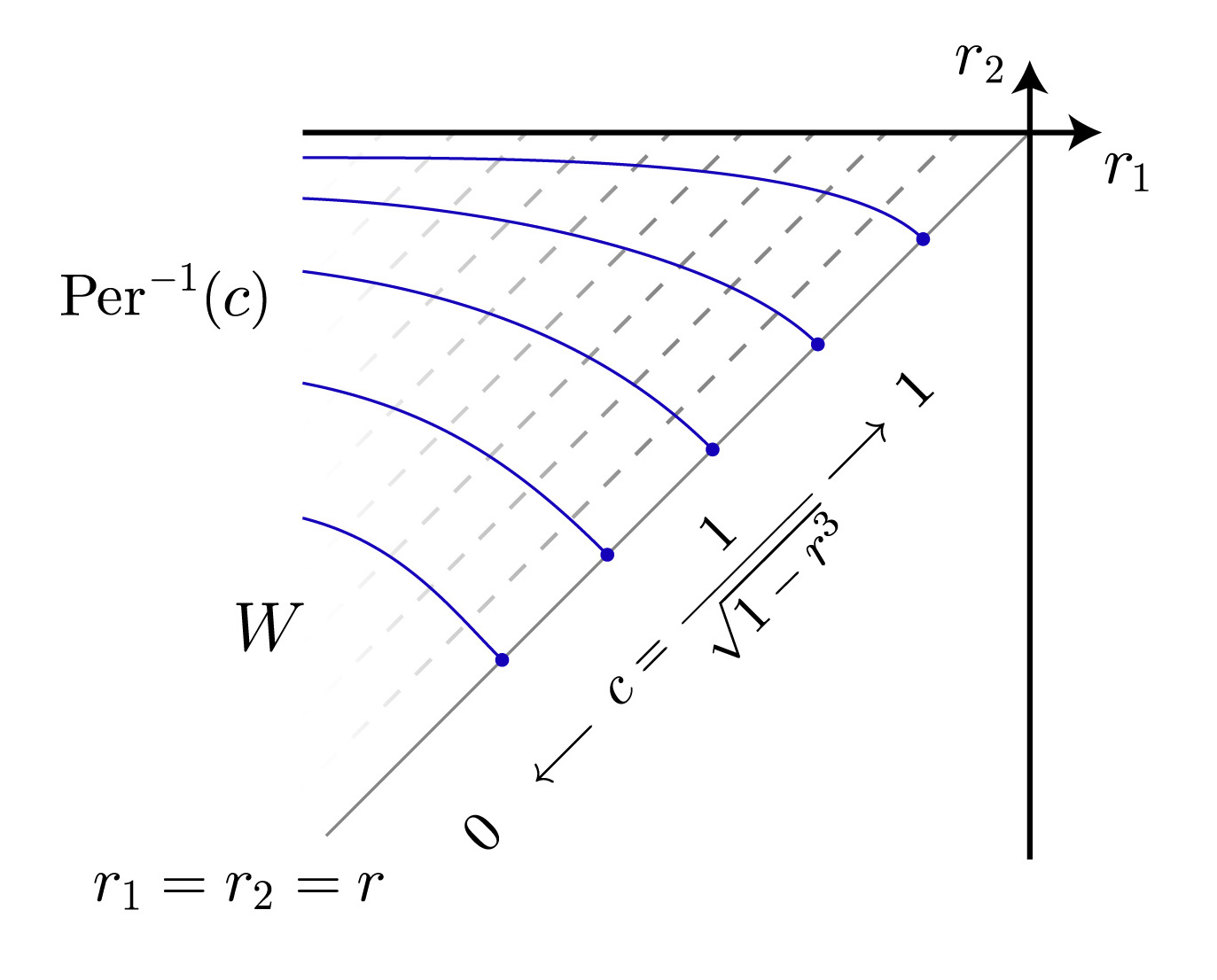} 
 \caption{Level sets of the period map ${\rm Per}(r_1,r_2)$.}\label{fig:fig2}
\end{center}
\end{figure}

\begin{theorem}\label{th:peri}
 For any $c\in (0,1)$, the level set ${\rm Per}^{-1}(c)$ is a connected, regular, real analytic curve in $W$ that meets $\cD$ orthogonally at the point $(r_c,r_c)$, where $r_c=\sqrt[3]{1-1/c^2}$.
 
Moreover, 
\begin{equation}\label{deriperr}
\frac{\parc {\rm Per}(r_1,r_2)}{\parc(r_1+r_2)} >0
\end{equation}
holds for every $(r_1,r_2)\in W$. In particular, the level curves of ${\rm Per}$ intersect transversely the lines $r_1-r_2 = {\rm const}$, and ${\rm Per}(r_1,r_2)$ is strictly increasing along them.
\end{theorem}
\begin{proof} 
It follows from Theorem \ref{th:cerouno} that, for any $c\in (0,1)$, the level set ${\rm Per}^{-1}(c)$ contains a regular, real analytic curve that intersects $\cD$ orthogonally at $(r_c,r_c)$. We wish to show that any level set ${\rm Per}^{-1}(c)$ is one of such regular, real analytic curves starting at $\mathcal{D}$; in particular, these level sets are connected.

To start we will prove \eqref{deriperr}. First, by differentiation of $Q(s,r_1,r_2)$ in \eqref{Qper} we have
\begin{equation}\label{peres3}
Q_{r_1}+Q_{r_2} =\frac{\cL(s;r_1,r_2)}{\Upsilon(s;r_1,r_2)},
\end{equation}
where $$\cL(s;r_1,r_2):=r_1(r_1-r_2 + 3 r_1^2 r_2^2) - s (r_1 -r_2) (r_1+r_2 + 3 r_1^2 r_2^2)$$ and
$$\Upsilon (s;r_1,r_2):= -2\pi \sqrt{r_1 r_2} \sqrt{s (1-s)} (r_1+(r_2-r_1)s)^2 \, (1- r_1 r_2((r_1+(r_2-r_1)s)^{3/2}<0.$$ The numerator $\cL(s;r_1,r_2)$ is linear in $s$, and is always negative at $s=1$, since $r_1<r_2<0$. As $\Upsilon(s;r_1,r_2)$ is negative, $Q_{r_1}+Q_{r_2}$ is always positive at $s=1$.

Assume first of all that $$r_1 -r_2 + 3r_1^2 r_2^2 \geq 0.$$ Thus, $Q_{r_1}+Q_{r_2}\geq 0$ at $s=0$, and so it is non-negative at every point. By \eqref{perotra} we obtain directly that \eqref{deriperr} holds.

We assume then in what follows that 
\begin{equation*}
r_1 -r_2 + 3r_1^2 r_2^2 < 0.
\end{equation*} This implies that, for any $(r_1,r_2)\in W$, the (linear) numerator $\cL(s;r_1,r_2)$ in \eqref{peres3} has a unique change of sign between $s=0$ and $s=1$. Specifically, there exist $s^*\in (0,1)$ such that $\cL(s^*)=0$, with $\cL(s)>0$ if $s\in [0,s_0^*)$ and $\cL(s)<0$ if $s\in (s^*,1]$.

In addition, the function 
\begin{equation}\label{peres4}
\frac{1}{(1-r_1 r_2 (r_1+(r_2-r_1)s))^{3/2}}
\end{equation} is strictly increasing, since $r_1<r_2<0$. Call $R^*>0$ to the value of \eqref{peres4} at $s=s_0^*$. Hence, we have:

\begin{equation}
\frac{1}{(1-r_1^2 r_2)^{3/2}} \leq \frac{1}{(1-r_1 r_2 (r_1+(r_2-r_1)s))^{3/2}} \leq R^* \hspace{0.5cm} \text{ in \hspace{0.1cm} $[0,s_0^*]$,}
\end{equation}
and
\begin{equation}
R^*\leq  \frac{1}{(1-r_1 r_2 (r_1+(r_2-r_1)s))^{3/2}} \leq \frac{1}{(1-r_2^2 r_1)^{3/2}}  \hspace{0.5cm} \text{ in \hspace{0.1cm} $[s_0^*,1]$.}
\end{equation}
Thus, taking into account the sign of $\cL$, we have
$$\int_{s_0^*}^1 \frac{\cL(s;r_1,r_2)}{\Upsilon(s;r_1,r_2)} \, ds \geq R^* \int_{s_0^*}^1 \frac{\cL(s;r_1,r_2)}{-2\pi \sqrt{r_1 r_2} \sqrt{s (1-s)}(r_1+(r_2-r_1)s)^2} \, ds$$ and 
$$\int_0^{s_0^*} \frac{\cL(s;r_1,r_2)}{\Upsilon(s;r_1,r_2)} \, ds \geq R^* \int_0^{s_0^*} \frac{\cL(s;r_1,r_2)}{-2\pi \sqrt{r_1 r_2} \sqrt{s (1-s)}(r_1+(r_2-r_1)s)^2} \, ds.$$ 

Therefore, by \eqref{perotra} and \eqref{peres3}, we have

\begin{equation}\label{arre}\def\arraystretch{2.4}\begin{array}{lll}\displaystyle \frac{\parc {\rm Per}(r_1,r_2)}{\parc (r_1+r_2)}& =& \displaystyle \int_0^{s_0^*} \frac{\cL(s;r_1,r_2)}{\Upsilon(s;r_1,r_2)} \, ds +  \int_{s_0^*}^1 \frac{\cL(s;r_1,r_2)}{\Upsilon(s;r_1,r_2)} \, ds \\ & \geq & \displaystyle R^* \int_0^1 \frac{\cL(s;r_1,r_2)}{-2\pi \sqrt{r_1 r_2} \sqrt{s (1-s)}(r_1+(r_2-r_1)s)^2} \, ds
\end{array}
\end{equation}
The function $$\frac{(r_1-r_2) \sqrt{s (1-s)}}{\pi \sqrt{r_1 r_2}} + \frac{3}{\pi} {\rm arctan} \left(\sqrt{\frac{r_2 s}{r_1 (1-s)}}\right)$$ is a primitive of the above integral. Thus, from \eqref{arre} we obtain $$\frac{\parc {\rm Per}(r_1,r_2)}{\parc (r_1+r_2)} \geq \frac{3 r_1 r_2}{2} R^* >0.$$ This proves \eqref{deriperr}.

Hence, the gradient of ${\rm Per}(r_1,r_2)$ does not vanish in $W$, and so the level sets of ${\rm Per}$ are regular, real analytic curves, at first maybe not connected, that intersect transversely the lines $r_1-r_2={\rm const}$, and ${\rm Per}(r_1,r_2)$ increases along any such line as $r_1$ (or $r_2$) increases.

Finally, we can see that the union of all the level curves of ${\rm Per}(r_1,r_2)$ that start at the diagonal $\cD$ is equal to $W$, by a connectedness argument. Indeed, such set is trivially closed, and it is also open due to the regularity of all the level curves of ${\rm Per}$.
\end{proof}

\section{Free boundary minimal annuli: proof of Theorem \ref{main}}\label{sec:height}

\subsection{Minimal annuli with free boundary in two spheres}

Let $(r_1,r_2)\in \Omega_0$, where $\Omega_0\subset \R^2$ is the open set defined by \eqref{omw}, and let $\Sigma_{\tau}=\Sigma_{\tau}(r_1,r_2)$ denote the minimal surface in $\R^3$ of Definition \ref{sita}. In this way, $\Sigma_{\tau}$ is parametrized as  map $\psi(u,v):[-\tau,\tau]\times \R\flecha \R^3$, and the curve $\psi(0,v)$ is a horizontal planar geodesic contained in the $x_3=0$-plane. 

The next proposition follows from our analysis in Section \ref{sec:wei}, and collects some of the most important properties of $\Sigma_{\tau}$ proved there:

\begin{proposition}\label{2centers}
$\Sigma_{\tau}$ is an immersed minimal strip in $\R^3$ that satisfies the following properties:
\begin{enumerate}
\item
The curvature lines $v\mapsto \psi(\cdot,v)$ of $\Sigma_{\tau}$ are spherical curvature lines.
 \item
$\Sigma_{\tau}$ is symmetric with respect to the planes $x_3=0$ and $x_2=0$. Moreover, the symmetry with respect to $x_3=0$ interchanges the boundary curves of $\parc \Sigma_{\tau}$.
 \item
Each component of $\parc \Sigma_{\tau}$ intersects orthogonally a sphere of a certain radius $R(\tau)>0$, and the centers $c(\tau),c(-\tau)$ of these two spheres are symmetric with respect to $x_3=0$. In particular, $c_3(-\tau)=-c_3(\tau)$.
\item
If additionally $c_3(\tau)=0$, then a suitable homothety and translation of $\Sigma_{\tau}$ is a free boundary minimal strip immersed in the unit ball $\B$ of $\R^3$.
\end{enumerate}
Assume also that ${\rm Per}(r_1,r_2)=\frac{m}{n}\in \Q$, where $m,n$ have no common divisors. Then:
\begin{enumerate}
\item[i)]
The quotient of $\Sigma_{\tau}$ by the conformal projection \eqref{quotient} defines a compact minimal annulus $\Sigma_{\tau}^*$ in $\R^3$, with all the above properties.
\item[ii)]
The Gauss map of $\Sigma_{\tau}^*$ defines a regular $m$-fold covering of the great circle $\S^2\cap \{x_3=0\}$ along the horizontal planar geodesic $\psi(0,v)$ of $\Sigma_{\tau}^*$.
\item[iii)]
$\Sigma_{\tau}^*$ has a prismatic symmetry group $D_n\times \Z_2$ of order $4n$.
\end{enumerate}
\end{proposition}

In this section we show that some of the compact minimal annuli $\Sigma_{\tau}^*$ of Proposition \ref{2centers} are actually free boundary in the unit ball (after a homothety and a translation). For that, by item (4) of Proposition \ref{2centers}, we need to control the height $c_3(\tau)$ of the boundary curve $\psi(\tau,v)$ of $\parc \Sigma_{\tau}$. This will be done by studying the nodal set of the height map $\mathfrak{h}(r_1,r_2)$ that we introduce below.

\subsection{The height map}\label{subsec:height}

Let $\Omega_0\subset \R^2$ be the open set defined by \eqref{omw}. By Proposition \ref{taucero}, the height $c_3(\tau)$ of the center of the sphere that contains the boundary curve $\psi(\tau,v)$ of $\Sigma_{\tau}$ is given by \eqref{centro}. 

\begin{definition}\label{def:height}
The \emph{height map} $\mathfrak{h}(r_1,r_2):\Omega_0\flecha \R$ is defined as
\begin{equation}\label{heimap}
\mathfrak{h}(r_1,r_2) = -c_3(\tau),  
\end{equation}
where $c_3(\tau)$ is given by \eqref{centro}. All quantities in \eqref{centro} are those established in Section \ref{sec:wei}, and they depend analytically on $(r_1,r_2)$. So, $\mathfrak{h}(r_1,r_2)$ is an analytic map.
\end{definition}

In order to control the nodal set $\mathfrak{h}^{-1}(0)\subset \Omega_0$, we will study how $\mathfrak{h}(r_1,r_2)$ behaves when $(r_1,r_2)$ approaches the segment $$L_0:=\{(r,r): r\in (-\sqrt[3]{2},-1]\}\subset \parc \Omega_0\cap \{r_1=r_2\}.$$
For this, we will use the analysis of the degenerate case of system \eqref{system2} in Section \ref{sec:dege}.

\begin{remark}
As $(r_1,r_2)\to (r,r)$, the fundamental lattice generated by $(2\omega_1,2\omega_2)$ degenerates, and in the limit we have 
\begin{equation}\label{omein}
\omega_1 = \8, \hspace{1cm} \omega_2 = \frac{-i \pi r}{\sqrt{1-r^3}}.
\end{equation} 
This follows from the general expression (see \cite[p. 444]{WW}) $$\omega_1 =\int_{e_1}^{\8} \frac{1}{\sqrt{4x^3-g_2 x - g_3}} dx, \hspace{0.5cm} \omega_2 = -i \int_{-\8}^{e_3} \frac{1}{\sqrt{-4x^3+g_2 x + g_3}} dx,$$ taking into account in our situation the relations \eqref{def:ej} and \eqref{def:gs}.

\end{remark}

\begin{theorem}\label{th:hei}
The height map $\mathfrak{h}:\Omega_0\flecha \R$ extends continuously to $\Omega_0\cup L_0$. Moreover, there exists a value $r^*\approx -1.078124$ such that:
\begin{enumerate}
\item
$\mathfrak{h}(r^*,r^*)=0$.
\item
$\mathfrak{h}(r,r)<0$ if $r\in (-\sqrt[3]{2}, r^*)$ and $\mathfrak{h}(r,r)>0$ if $r\in (r^*,-1]$.
 \item
The set $\mathfrak{h}^{-1}(0)\subset \Omega_0$ contains a regular, real analytic curve $\gamma^*$ with endpoint $(r^*,r^*)$.
\item
The period map ${\rm Per} (r_1,r_2)$ is not constant along $\gamma^*$.
\end{enumerate}
\end{theorem}
\begin{proof}
Assume that $(r_1,r_2)\in \Omega_0 \flecha (r,r)\in L_0$. As explained in Section \ref{sec:wei}, the elliptic function $\phi(z)$ in \eqref{defF} for $(r_1,r_2)$ satisfies that $\phi(i\R)$ is contained in the real interval $[r_1,r_2]$. Thus, in the limit we have
\begin{equation}\label{weidatalimit}
\phi(z)= r,\hspace{0.5cm} g(z) = {\rm exp} (z/r).
\end{equation}
These are the Weierstrass data of the universal covering of a catenoid with necksize $r^2$ and vertical axis, and so, the surfaces $\Sigma_{\tau}=\Sigma_{\tau}(r_1,r_2)$ converge uniformly on compact sets to some translation of it. 

Unluckily, the catenoid is a very degenerate situation with respect to our study. Indeed, on the catenoid, the fundamental equation \eqref{om1}, and in particular the pair $(\alfa(u),\beta(u))$, does not carry meaningful information on it, due to the fact that the metric $e^{2\omega}$ only depends on $u$. More specifically, for any horizontal curvature line $\gamma=\psi(u_0,v)$ of any catenoid $\Sigma_c$, there exists a sphere $S(p,R)$ with center in the axis of $\Sigma_c$ that intersects $\Sigma_c$ orthogonally along $\gamma$. This means in our situation that we cannot read the limit values of the height map $\mathfrak{h}(r_1,r_2)$ directly from the geometry of the limit catenoid, and we need to follow a more indirect argument.

For any $(r_1,r_2)\in \Omega_0$, let $\Sigma_{\tau}=\Sigma_{\tau}(r_1,r_2)$ denote the minimal surface introduced in Definition \ref{sita}. Let 
$(\alfa(u),\beta(u))$ be the solution to \eqref{system} associated to $\Sigma_{\tau}(r_1,r_2)$. By Remark \ref{iniconi}, it has the initial conditions \begin{equation}\label{inicondi}
\alfa(0)=\beta(0) =0.
\end{equation}
Therefore, for each $(r_1,r_2)$ we have $\tau\in (0,\omega_1)$ and $\beta(\tau)=0$. Note that $\tau=\tau(r_1,r_2)$. \emph{We denote by $\hat{\alfa}=\hat{\alfa}(r_1,r_2)$ the value of $\alfa(\tau)$ at $\tau(r_1,r_2)$.} As usual, we consider the special orbit $\Gamma_0$ of the associated system \eqref{system2} (see Definition \ref{def:gamma0}), determined by the initial conditions \eqref{inicondi}.

On the other hand, given $r\in (-\sqrt[3]{2},-1]$, let $(\bar{\alfa}(u),\bar{\beta}(u))$ be the unique solution to the system \eqref{system2} for the choices \eqref{degecon} for $(r_1,r_2)$, with the trivial initial conditions $\bar{\alfa}(0)= \bar{\beta}(0)=0$.  This is the degenerate case studied in Section \ref{sec:dege}. 

As $(r_1,r_2)\to (r,r)$, the solutions $(\alfa(u),\beta(u))$ converge to $(\bar{\alfa}(u),\bar{\beta}(u))$. In the same way, as $(r_1,r_2)\to (r,r)$, the orbits $\Gamma_0(r_1,r_2)$ converge to the corresponding orbit $\Gamma_0=\Gamma_0(r)$ of the limit system, and the values $\hat{\alfa}(r_1,r_2)$ converge to $\hat{\alfa}(r)$ as introduced in Definition \ref{def:dege}. Note that the value $\hat{\alfa}(r)$ is explicitly computed in Theorem \ref{th:dege}. Thus, there exists a finite number $\tau(r)>0$ such that $\bar{\alfa}(\tau(r))= \hat{\alfa}(r)$ and $\bar{\beta}(\tau(r)) =0$. In this way, the quantities $\tau(r_1,r_2)$ defined above converge to $\tau(r)$.

Evaluating \eqref{om1} for $(r_1,r_2)$ at $(\tau,0)$, we obtain $$2 \omega_u (\tau,0) e^{\omega(\tau,0)} = \alfa(\tau)e^{2\omega(\tau,0)} =  \hat{\alfa}(r_1,r_2)e^{2\omega(\tau,0)} .$$ Taking limits, we have by \eqref{eome2}, \eqref{centro5} and \eqref{weidatalimit}
\begin{equation}\label{alta}
\hat{\alfa}(r)= \frac{-2\sinh\left(\frac{\tau(r)}{r}\right)}{r^2 \cosh^2\left(\frac{\tau(r)}{r}\right)},
\end{equation}
which determines $\hat{\alfa}(r)$ in terms of $\tau(r)$, and viceversa.

Coming now back to \eqref{centro}, we have for its last term 
\begin{equation}\label{2term}
  \frac{(1+g(\tau)^2) \, \phi(\tau)^2}{(1+g(\tau)^2)\phi'(\tau) + g(\tau)^2 -1} \longrightarrow r^2 \coth\left(\frac{\tau(r)}{r}\right), \hspace{0.6cm} \text{ as } (r_1,r_2)\to (r,r).
\end{equation}
So, by \eqref{weidatalimit} and the first equality in \eqref{heta}, we deduce that 
%
%
$\mathfrak{h}(r_1,r_2)$ in \eqref{heimap} satisfies
\begin{equation}\label{limh}
\mathfrak{h}(r_1,r_2) \longrightarrow -r^2 \left\{ \frac{\tau(r)}{r} - \coth\left(\frac{\tau(r)}{r}\right)\right\}, \hspace{0.6cm} \text{ as } (r_1,r_2)\to (r,r).
\end{equation}
This proves in particular that $\mathfrak{h}(r_1,r_2)$ extends continuously to the segment $L_0$. We now prove properties (1) and (2) in the statement of the theorem.

To start, let us assume that $r\in (-\sqrt[3]{2},r^{\sharp}]$, where $r^{\sharp}$ is the value in Theorem \ref{th:dege}. We prove next that $\mathfrak{h}(r,r)\neq 0$ in that interval. 

Arguing by contradiction, if $\mathfrak{h}(r,r)=0$, then by \eqref{limh} we have $\tau(r)/r=x_0<0$, where $x_0$ is the only negative solution to $x_0=\coth(x_0)$. So, by \eqref{alta}, $$\hat{\alfa}(r)= \frac{-2}{r^2 x_0 \cosh(x_0)}.$$ 
By item (ii) of Theorem \ref{th:dege}, we have $\hat{\cG}(r)=-1$, where 
$$\hat{\cG}(x):=\cG\left(\frac{-2}{x_0 \cosh(x_0) x^2 }\right).$$ However, one can check that $\hat{\cG}(x)$ does not take the value $-1$ in the interval $[-\sqrt[3]{2},r^{\sharp}]$; indeed, the only solution to $\hat{\cG}(x)=-1$ happens near the value $r\approx -1.0584$, outside the interval. This contradiction shows that $\mathfrak{h}(r,r)\neq 0$ in $(-\sqrt[3]{2},r^{\sharp}]$. As a matter of fact, we easily show that $\mathfrak{h}(r,r)< 0$ there, since by Theorem \ref{th:dege} we can compute explicitly $$\mathfrak{h}(r^{\sharp},r^{\sharp})= -(r^{\sharp})^2 \left(\sqrt{2} - \cosh^{-1} (\sqrt{2})\right)<0.$$

We now can make a similar argument on the interval $[r^{\sharp},-1]$. This time, if $\mathfrak{h}(r,r)=0$ for some $r\in [r^{\sharp},-1]$, we can use item (i) of Theorem \ref{th:dege} to obtain that $\hat{\cH}(r)=-1$, where $$\hat{\cH}(x):= \cH\left(\frac{-2}{x_0 \cosh(x_0) x^2}\right).$$ Now, the function $\hat{\cH}(x)$ is strictly increasing in $[r^{\sharp},-1]$, with $\hat{\cH}(-1)>-1$ and $\hat{\cH}(r^{\sharp})<-1$. Thus, there exists a unique value $r^*\approx -1.078124$ in $[r^{\sharp},-1]$ for which $\hat{\cH}(r^*)=-1$. This value $r^*$ is thus the unique solution to $\mathfrak{h}(r,r)=0$ in $[-\sqrt[3]{2},-1]$. We also have $\mathfrak{h}(r,r)>0$ for $r\in (r^*,0)$ by direct evaluation at $r=-1$. This proves items (1) and (2) of the theorem.

Item (3) is a direct consequence of items (1), (2) and the analyticity of $\mathfrak{h}(r_1,r_2)$ in $\Omega_0$.

To prove item (4), assume that the period function ${\rm Per}(r_1,r_2)$ is constant along $\gamma^*$. Then, by \eqref{per:boundary} and Theorem \ref{th:peri}, $\gamma^*\subset \Omega_0$ is a piece of the level set of ${\rm Per}$ given by 
\begin{equation}\label{levelper}
{\rm Per} (r_1,r_2)=c^*:= \frac{1}{\sqrt{1-(r^*)^3}}.
\end{equation}
Note that, by \eqref{perdire}, this level curve is explicitly given by 
 \begin{equation}\label{levelper2}
-\frac{1}{\pi}\int_{r_1}^{r_2} \frac{1}{t\sqrt{-q(t)}} \, dt = c^*.
\end{equation}

Take any $(r_1,r_2)\in {\rm Per}^{-1}(c^*)$. By Proposition \ref{monocen}, there exists a unique $u^*=u^*(r_1,r_2)\in (0,\omega_1)$ for which the center $c(u^*)$ of the sphere where $\psi(u^*,v)$ is contained has height zero, i.e. $c_3(u^*)=0$. Clearly, along the curve $\gamma^*$ we have $u^*=\tau$, and so $\beta(u^*)=0$ for every $(r_1,r_2)\in \gamma^*$. By analyticity, we have $\beta(u^*)=0$ for every $(r_1,r_2)\in {\rm Per}^{-1}(c^*)$, where $u^*=u^*(r_1,r_2)$. We now prove that this is not possible.

To start, let us recall that, by Theorem \ref{th:peri}, the level sets of ${\rm Per}(r_1,r_2)$ are regular, connected, real analytic curves that start from the diagonal $r_1=r_2$ and intersect at most once every line of the form $r_1-r_2 ={\rm const}$, since ${\rm Per}(r_1,r_2)$ is strictly increasing along any such line.


Let $(\hat{r}_1,\hat{r}_2)$ be the unique intersection point of the level curve \eqref{levelper2} and the line $r_1-r_2=-8$; thus, $\hat{r}_2\approx -0.222455$. If for $(\hat{r}_1,\hat{r}_2)$ we take, for instance, the value $u_0:=4/5\in (0,\omega_1)$, a computation following the explicit procedure described at the end of Section \ref{sec:wei} shows that $\beta(u_0)\approx 0.615>0$ while $c_3(u_0)\approx -5.1<0$. By Proposition \ref{monocen}, $c_3(u)$ is a decreasing bijection from $(0,\omega_1)$ into $\R$. Then, since $c_3(u_0)<0$ we deduce that $u^*\in (0,u_0)$. Also, $\beta(u^*)=0$, since $(r_1,r_2)\in {\rm Per}^{-1}(c^*)$. But next observe that, by Remark \ref{rema:orbits}, the equality $\beta=0$ happens exactly once (note that $(\hat{r}_1,\hat{r}_2)\in \Omega\setminus \Omega_0$, so we are in the conditions of this remark). The fact that $\beta(u_0)>0$ implies then that $u_0<u^*$, a contradiction. This completes the proof of Theorem \ref{th:hei}.

\end{proof}

{\it Proof of Theorem \ref{main}:}

Let $\gamma^*\subset \Omega_0$ be the real analytic curve constructed in Theorem \ref{th:hei}. For each $(r_1,r_2)\in \gamma^*$ such that ${\rm Per} (r_1,r_2)\in \Q$, we can consider the compact minimal annulus $\Sigma_{\tau}^*$ constructed in Proposition \ref{2centers}. Note that $\Sigma_{\tau}^*$ has a prismatic symmetry group, as detailed in item (iii) of Proposition \ref{2centers} or in Corollary \ref{rem:symmetries}. Now, the curve $\gamma^*$ lies in the nodal set $\mathfrak{h}^{-1}(0)$ of the height map $\mathfrak{h}(r_1,r_2)$. Thus, by item (4) of Proposition \ref{2centers}, a homothety and translation of $\Sigma_{\tau}^*$ defines a compact minimal annulus with free boundary in $\B^3$, and all the desired properties.

\vspace{0.2cm}

{\it Proof of Corollary \ref{cor:intro}:}

For any $(r_1,r_2)\in \gamma^*$ with ${\rm Per}(r_1,r_2)\in \R\setminus\Q$, the surface $\Sigma_{\tau}$ in Proposition \ref{2centers} is a complete (with boundary), non-compact minimal strip $\Sigma$ with the desired properties. See Remark \ref{rem:fb}.

\subsection{Examples, discussion and open problems}\label{subsec:discussion}

The most interesting examples of free boundary minimal annuli of our family are those associated to periods $\frac{m}{n}\in (0,1)$ where both $m,n$ are as small as possible. Indeed, $m$ gives the rotation index of the Gauss map along the orthogonal intersection of the minimal annulus with the plane $x_3=0$, while $n$ gives the number of periods that are necessary for the annulus to close, and determines its symmetry group. See Corollary \ref{rem:symmetries}.

The value of the period for the free boundary minimal annuli $\Sigma_{\tau}^*$ in $\B^3$ obtained in Theorem \ref{main} cannot be equal to $1/2$, or more generally, to $1/n$, $n\in \N$. We do not detail the complete argument. We merely indicate that if the period of $\Sigma_{\tau}^*$ was $1/n$, then its Gauss map would be a diffeomorphism onto its image in $\S^2$ (observe that it is a diffeomorphism along its central planar geodesic, by item (ii) of Proposition \ref{2centers}), and in these conditions, $\Sigma_{\tau}^*$ would be embedded ($\Sigma_{\tau}^*$ would actually be a radial graph, as explained in \cite{So}). Since $\Sigma_{\tau}^*$ is symmetric with respect to the three coordinate planes of $\R^3$, this contradicts McGrath's theorem in \cite{M}. We note that the main theorem in \cite{M} has been extended by Kusner-McGrath \cite{KM} to the case of antipodal symmetry, and by Seo \cite{Seo} to the case of two arbitrary planes of reflective symmetry. These results are based in part on a characterization by Fraser and Schoen \cite{FS} of the critical catenoid in terms of Steklov eigenvalues, and on the two-piece property of embedded free boundary minimal surfaces by Lima and Menezes \cite{LM}.

%
%
%
%

We have not been able to find any free boundary minimal annulus in $\B^3$ within our family with a period $\geq 2/3$. When we approach the diagonal $r_1=r_2$ along the curve ${\rm Per}^{-1}(2/3)$, the value $\mathfrak{h}(r_1,r_2)$ of the height map is very small, but always positive. In this sense, let us observe that if $r^*$ is the value at which $\gamma^*$ intersects $r_1=r_2$ (see Theorem \ref{th:hei}), then $${\rm Per}(r^*,r^*)=\frac{1}{\sqrt{1-(r^*)^3}} \approx 0.6662 <\frac{2}{3}.$$ We believe that there should exist examples of free boundary minimal annuli in $\B^3$ for all values of the period in the interval $(1/2,{\rm Per}(r^*,r^*))$.

%

The example in Figures \ref{fig:eje1} and \ref{fig:eje2} corresponds to a period equal to $3/5$. We next prove the existence of this example. By Theorem \ref{th:hei}, we have $\mathfrak{h}(r,r)<0$ at the point $(r,r)$ of intersection of ${\rm Per}^{-1}(3/5)$ with $r_1=r_2$. Thus, by monotonicity of $c_3(u)$, we have at $(r,r)$ that the value $\tau$ at which $\beta(\tau)=0$ and the value $u^*$ at which $c_3(u^*)=0$ satisfy $\tau<u^*$. In this way, $\beta(u^*)<0$ at $(r,r)$, by our orbit analysis in Section \ref{sec:dege}.

We now consider the unique point of intersection $(\hat{r}_1,\hat{r}_2)$ of ${\rm Per}^{-1}(3/5)$ with the line $r_1-r_2=-8$. This point lies in $\Omega\setminus \Omega_0$, and so by Remark \ref{rema:orbits}, we have at $(\hat{r}_1,\hat{r}_2)$ that either $\beta(u)>0$ for every $u\in (0,\omega_1)$ or there is a value $\xi\in (0,\omega_1)$ such that $\beta(u)>0$ if $u\in (0,\xi)$ and $\beta(u)<0$ if $u\in (\xi,\omega_1)$. But now, for $u_0= 4/5\in (0,\omega_1)$, a computation following the process described at the end of the proof of Theorem \ref{th:hei} shows that $\beta(u_0)>0$ while $c_3(u_0)<0$. This implies that $\beta(u^*)>0$ at $(\hat{r}_1,\hat{r}_2)$, since $c_3(u)$ is decreasing. Thus, there is a point $(r_1,r_2)$ in the connected curve ${\rm Per}^{-1}(3/5)$ where $\beta(u^*)=0$. Thus, by Proposition \ref{2centers}, we obtain a free boundary minimal annulus in $\B^3$. This corresponds to the example in Figures \ref{fig:eje1} and \ref{fig:eje2}.

%
%
%
%


Similarly, one can prove existence of free boundary minimal annuli in $\B^3$ with other low periods, like $4/7$, $5/8$ or $5/9$.  

Note that there obviously exist different points in $\mathfrak{h}^{-1}(0)$ whose associated periods are rational numbers with the same (maybe large) denominator, say $n$. This provides examples of non-congruent free boundary minimal annuli in $\B^3$ with the same symmetry group (the prismatic group of order $4n$). The authors thank Mario B. Schulz for this remark. This is related to a recent topological non-uniqueness theorem by Carlotto, Schulz and Wiygul \cite{CSW}, who constructed examples of embedded non-congruent free boundary minimal surfaces in $\B^3$ with the same topology and symmetry group.

%
%

In any case, \emph{the classification of all free boundary minimal annuli in $\B^3$ with spherical curvature lines is far from complete}. For instance, there seem to exist such free boundary minimal annuli that are not symmetric with respect to a central planar geodesic, and so, their isometry group would not be prismatic (it would be isomorphic to $D_n$).

The critical catenoid is stable (since it is a radial graph), and is characterized as the unique free boundary minimal annulus in $\B^3$ with index $4$ (as a free boundary minimal surface in $\B^3$), see \cite{D,SZ,T}. It is an interesting problem to study the related stability properties of the minimal annuli in $\B^3$ constructed in Theorem \ref{main}.

\begin{remark}
The minimal annuli of Theorem \ref{main} are, to the authors' knowledge, the first examples of non-embedded free boundary minimal surfaces in $\B^3$ (excluding trivial coverings of embedded ones). They are also the first non-trivial examples of free boundary minimal surfaces in $\B^3$ whose first Steklov eigenvalue $\sigma_1$ is known to be smaller than $1$; see \cite{FS,L} for a discussion on the meaning and importance of this eigenvalue condition. We remark that Fraser and Schoen proved that $\sigma_1\leq 1$, and that if $\sigma_1=1$ holds for an immersed free boundary minimal annulus in $\B^3$, then this annulus is the critical catenoid; thus, $\sigma_1<1$ holds for our minimal annuli here.
\end{remark}

A challenging but very interesting open problem is to \emph{classify all the free boundary minimal annuli in $\B^3$}. A more specific problem is the following intriguing question:


\noindent {\bf Open problem:} \emph{Are there free boundary minimal Möbius bands in $\B^3$?}

By topological reasons, such examples will never be embedded. We note that Fraser and Schoen found in \cite{FS} an embedded minimal Möbius band with free boundary in the unit ball $\B^4$ of $\R^4$, the so-called \emph{critical Möbius band}; see also Fraser and Sargent \cite{FSa}.

\section{Capillary minimal annuli embedded in the unit ball}\label{sec:examples}

\subsection{Embedded minimal annuli in $\B^3$}

The family of minimal annuli foliated by spherical curvature lines that we have constructed also has the remarkable property that it produces compact embedded minimal annuli in the unit ball, like the ones in Figure \ref{fig:eje3}. 
\begin{theorem}\label{capillary}
For any $n\in \N$, $n>1$, there exists a real analytic family $\{\mathbb{A}_n(\mu) : \mu \geq 0\}$ of compact minimal annuli in the unit ball $\B^3$ such that:
\begin{enumerate}
\item
$\mathbb{A}_n(\mu)$ intersects $\parc \B^3$ at a constant angle along $\parc \mathbb{A}_n(\mu)\subset \parc \B^3$, i.e., $\mathbb{A}_n(\mu)$ is a capillary minimal annulus in $\B^3$.
\item
Each $\mathbb{A}_n(\mu)$ is foliated by spherical curvature lines. 
 \item
Each $\mathbb{A}_n(\mu)$ is symmetric with respect to the planes $x_2=0$ and $x_3=0$, and with respect to the rotations around the $x_3$-axis of angles $2k\pi/n$, for $k\in \{1,\dots, n-1\}$. Thus, $\mathbb{A}_n(\mu)$ has a prismatic symmetry group of order $4n$.
\item
The planar geodesic $\mathbb{A}_n(\mu)\cap \{x_3=0\}$ is a closed strictly convex planar curve.
 \item
$\mathbb{A}_n(0)$ is a compact piece of a catenoid. $\mathbb{A}_n(\mu)$ is not rotational if $\mu \neq 0$.
\item
There exists $\ep>0$ such that, for every $\mu\in [0,\ep)$, the capillary minimal annulus $\mathbb{A}_n(\mu)$ is embedded in $\B^3$.
\end{enumerate}
\end{theorem}
\begin{proof}
By Theorem \ref{th:peri}, the level sets of ${\rm Per}$ are regular, connected, real analytic curves, and for any $n\in \N$ the level curve ${\rm Per}^{-1}(1/n)$ intersects the diagonal $r_1=r_2$ at the point $(\bar{r}_n,\bar{r}_n)$, with $\bar{r}_n=\sqrt[3]{1-n^2}$. 

By Corollary \ref{rem:symmetries}, every $(r_1,r_2)\in {\rm Per}^{-1}(1/n)$ together with the choice of initial condition $\hat{g}_0=1$ for the Gauss map $g$ determines a minimal annulus $\Sigma_n^*=\Sigma_n^*(r_1,r_2)$ in $\R^3$ foliated by spherical curvature lines, with the symmetries specified in item (3) above, and so, in particular, symmetric with respect to the $x_3=0$ plane. The height with respect to $x_3=0$ of the centers $c(u)$ of the spheres where the spherical curves $\psi(u_0,v)$ of $\Sigma_n^*$ lie decreases from $\8$ to $-\8$ as we move $u_0$ from $0$ to the half-period $\omega_1$; see Proposition \ref{monocen}. Thus, there exists some intermediate value $u_n^*\in (0,\omega_1)$ such $c_3(u_n^*)=0$, and therefore the restriction of $\Sigma_n^*$ to $[-u_n^*,u_n^*]\times \R$ produces (via the quotient \eqref{quotient}) a compact minimal annulus $\mathcal{A}_n$ in a ball $B$ of $\R^3$, with center at some point of the plane $x_3=0$. 

Consider next the planar geodesic $\Sigma_n^*\cap \{x_3=0\}$ of this example. It is a closed, locally convex planar curve that has rotation index equal to $1$, since the period is $1/n$. Thus, it is globally convex, and the Gauss map along it defines a diffeomorphism onto the horizontal great circle $\S^2\cap \{x_3=0\}$. We note, however, that the larger, non-compact annulus $\Sigma_n^*$ loses its embeddedness as we approach $u\to \omega_1$, due to the apparition of a flat end at $(\omega_1,0)$ (note that $\phi$ has a pole at $\omega_1$). We want to understand next the embeddedness of the compact piece $\mathcal{A}_n$ of $\Sigma_n^*$ when $(r_1,r_2)$ is close to the diagonal.

As $(r_1,r_2)\flecha (\bar{r}_n,\bar{r}_n)$ along ${\rm Per}^{-1}(1/n)$, the meromorphic functions $\phi$ in \eqref{Falt} that define the minimal annuli $\Sigma_n^*$ converge to the constant $\bar{r}_n$, see the beginning of the proof of Theorem \ref{th:hei}. Thus, taking into account \eqref{omein}, the annuli $\Sigma_n^*$ converge uniformly on compact sets to a subset of the minimal annulus with Weierstrass data on the conformal cylinder $\C/(2n\omega_2 \Z)$ given by 
 \begin{equation}\label{limitwdata}
\phi(z)=\bar{r}_n, \hspace{0.5cm} g(z)={\rm exp}(z/\bar{r}_n),
 \end{equation}
where $\omega_2$ is as in \eqref{omein}, i.e. $\omega_2 = -i\pi \bar{r}_n/n$. These are the Weierstrass  data of a (singly-covered, embedded) catenoid $C_n$ with vertical axis and necksize $\bar{r}_n^2$.

Assume for one moment that the limit of the compact annuli $\mathcal{A}_n=\mathcal{A}_n(r_1,r_2)$ is a compact set of this limit catenoid $C_n$, as $(r_1,r_2)\to (\bar{r}_n,\bar{r}_n)$. Then, for $(r_1,r_2)\in {\rm Per}^{-1}(1/n)$ close enough to $(\bar{r}_n,\bar{r}_n)$, all the annuli $\mathcal{A}_n$ are embedded. Thus, they are compact embedded minimal annuli foliated by spherical curvature lines, whose boundary lies in a ball of $\R^3$. By Theorem \ref{2centers}, they have a prismatic symmetry group $D_{n}\times \Z_2$, see Corollary \ref{rem:symmetries}. 

Therefore, after a homothety and a translation, we would obtain a real analytic family $\{\mathbb{A}_n(\mu)\}$, $\mu\geq 0$, of minimal annuli in $\B^3$ with all the properties stated in Theorem \ref{capillary}. Here the  parameter $\mu$ of this family is just a parametrization of the level curve ${\rm Per}^{-1}(1/n)$, so that $\mu=0$ corresponds to the point $(\bar{r}_n,\bar{r}_n)$.

Thus, it only remains to show that the limit of the annuli $\cA_n$ is a compact set of $C_n$. For that, it suffices to show that the values $u_n^*\in (0,\omega_1)$ that define $\cA_n$ as a subset of $\Sigma_n^*$ are bounded in $(0,\8)$; here, one should recall that $\omega_1\to \8$ as $(r_1,r_2)\to (\bar{r}_n,\bar{r}_n)$, see \eqref{omein}.

Let $(\alfa(u),\beta(u))$ be the solution to system \eqref{system} associated to $\Sigma_n^*$, which is determined by the initial values $\alfa(0)=\beta(0)=0$, with $\alfa'(0)>0$. For any fixed $u\in (0,\omega_1)$, the height $c_3(u)$ of the center of the sphere $S(c(u),R(u))$ where $\psi(u,v)$ lies is given by \eqref{3corn}. 
Also note that $c_3(u_n^*)=0$, by definition. 

We now make $(r_1,r_2)\flecha (\bar{r}_n,\bar{r}_n)$ in that expression. The pair $(\alfa(u),\beta(u))$ converges to the solution $(\bar{\alfa}(u),\bar{\beta}(u))$ to system \eqref{system} with the initial conditions $\bar{\alfa}(0)=\bar{\beta}(0)=0$, for the choice of constants $r_1=r_2=\bar{r}_n$, $r_3= 1/\bar{r}_n^2$.
Then, by \eqref{limitwdata} and \eqref{3corn} we obtain
\begin{equation}\label{limfin1}
c_3(u) \longrightarrow \bar{h}(u):= \bar{r}_n u + \frac{2-\bar{\beta}(u) \sinh(u/\bar{r}_n)}{\bar{\alfa}(u) \cosh(u/\bar{r}_n)}.
\end{equation}
We now prove that the function $\bar{h}(u)$ in \eqref{limfin1} satisfies $\bar{h}(u)\to -\8$ as $u\to \8$. 

Consider the system \eqref{system2} for $(s(\landa),t(\landa))$ associated to $(\bar{\alfa}(u),\bar{\beta}(u))$. In our situation, due to the double root of the polynomial $q(x)=-(x-\bar{r}_n)^2(x-1/\bar{r}_n^2)$, we easily obtain from \eqref{repa} and \eqref{system2} that $u(\landa)\to \8$ as $\landa\to \8$. In addition, recall from Section \ref{sec:dege} that any orbit $(s(\landa),t(\landa))$ of \eqref{system2} converges to $(0,r_2)$ as $\landa\to \8$ in that degenerate $r_1=r_2$ case (for us, $r_2=\bar{r}_n)$. Taking this into account, by \eqref{change} we have
$$\lim_{u\to \8} \frac{\bar{\beta}(u)}{\bar{\alfa}(u)} = \lim_{\landa\to \8}\left(-\frac{1}{s(\landa)} - \frac{1}{t(\landa)}\right) = -\8.$$This implies that $\bar{h}(u)\to -\8$ as $u\to \8$, as claimed (recall that $\bar{r}_n<0$). In particular, there exists $\bar{u}>0$ such that $\bar{h}(u)<0$ for every $u\geq \bar{u}$.

Now, we come back to the values $u_n^*$ that we wanted to show are bounded as $(r_1,r_2)\to (\bar{r}_n,\bar{r}_n)$. By their own definition, we have $c_3(u_n^*)=0$. If the $u_n^*$ were not bounded, we could consider a sequence $\{(r_1^k,r_2^k)\}_{k\in \N}\to (\bar{r}_n,\bar{r}_n)$ inside the level curve ${\rm Per}^{-1}(1/n)$ and values $\varrho_k \in (0,u_n^*)$, where here $u_n^*=u_n^*(r_1^k,r_2^k)$, such that $\varrho_k\to \bar{u}$ as $k\to \8$. By the monotonicity of $c_3(u)$, see Proposition \ref{monocen}, we have that $c_3(\varrho_k)>0$ for every $k$, since $\varrho_k<u_n^*(r_1^k,r_2^k)$. Thus, by \eqref{limfin1}, $\bar{h}(\bar{u})\geq 0$, a contradiction.

This contradiction shows that the $u_n^*$ must be bounded. As discussed previously, this finishes the proof of Theorem \ref{capillary}.
\end{proof}

\subsection{Discussion and open problems for the capillary case}\label{subsec:discussion2}

While the geometry of free boundary minimal surfaces in $\B^3$ has received a great number of contributions in the past decade, the more general situation of capillary minimal surfaces in $\B^3$ has remained, in contrast, largely unexplored. The examples that we construct in Theorem \ref{capillary} seem to indicate that there should exist an interesting bifurcation theory for capillary minimal surfaces in $\B^3$, analogous to the CMC case. For instance, we expect the following behavior regarding the minimal annuli of Theorem \ref{capillary}.

\noindent {\bf Conjecture.} \emph{The capillary minimal annuli $\mathbb{A}_n(\mu)$ in $\B^3$ constructed in Theorem \ref{capillary} are embedded for every $n>1$ and every $\mu\geq 0$. Moreover, as $\mu\to \8$, each family $\mathbb{A}_n(\mu)$ converges to a \emph{necklace} of $n$ flat vertical disks in $\B^3$ whose projection is a regular $n$-polygon inscribed in the unit circle of the $x_3=0$ plane.}

\begin{figure}
\begin{center}
\includegraphics[height=6.8cm]{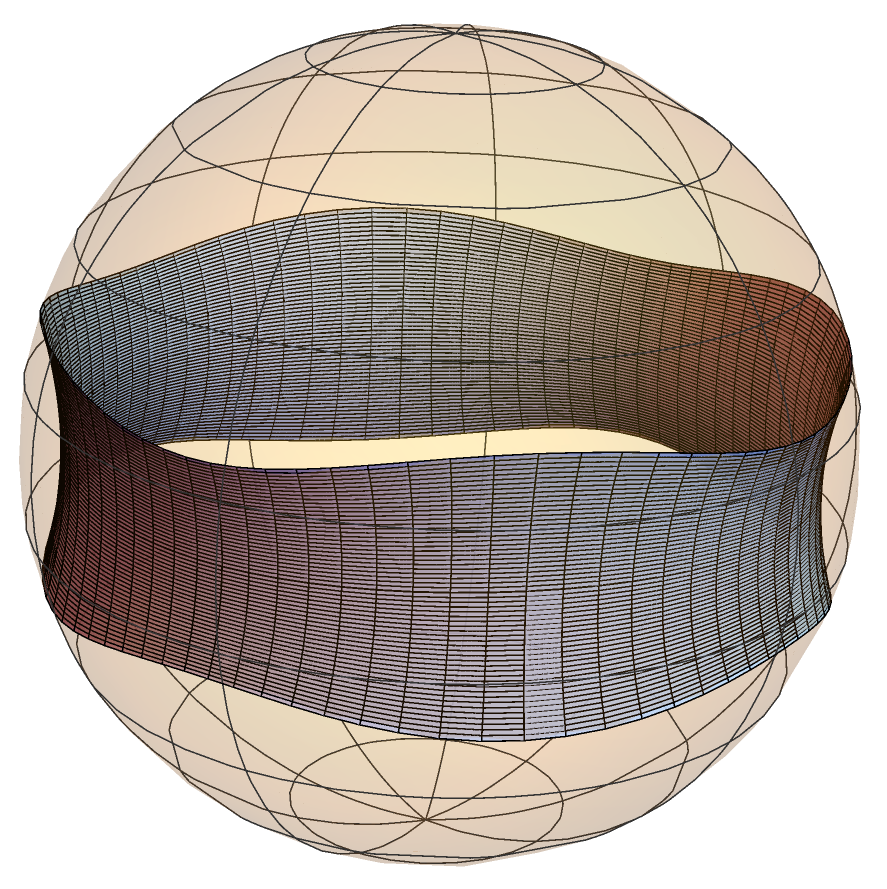} \hspace{0.2cm}
\includegraphics[height=6.8cm]{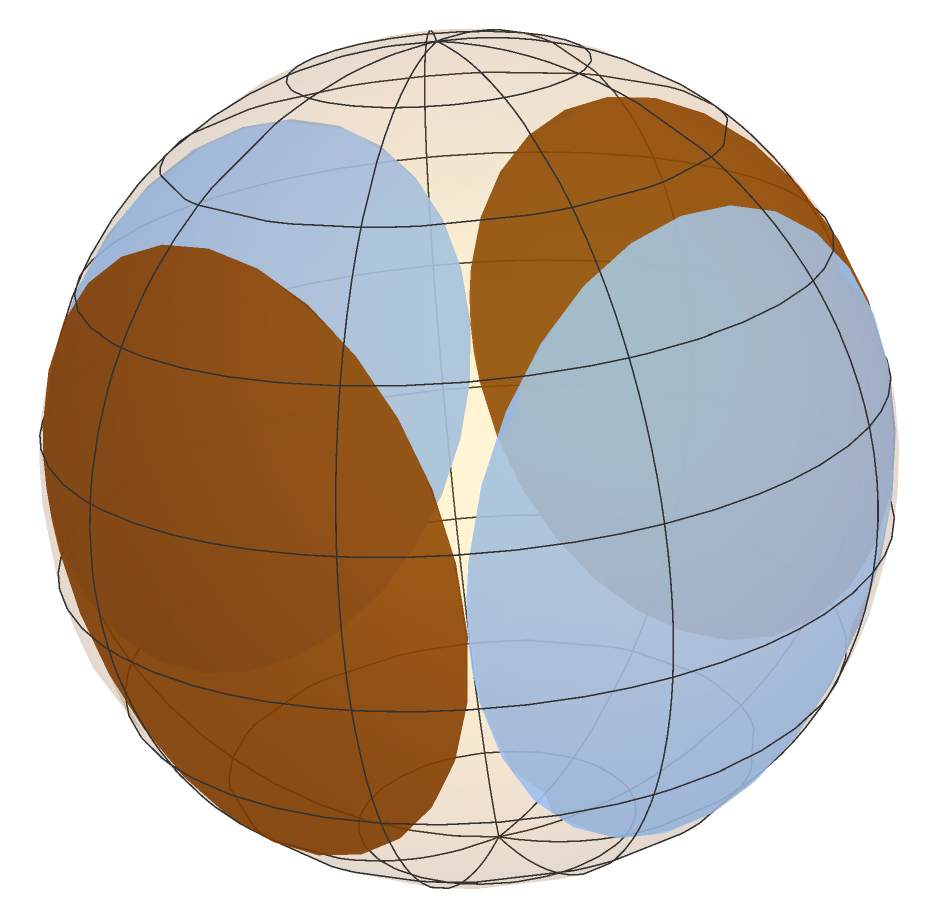} 
\caption{Left: an embedded capillary minimal annulus $\mathbb{A}_n(\mu)$ with $n=4$ and $\mu$ close to $0$. Right: expected limit behavior of the minimal annuli $\mathbb{A}_4(\mu)$ as $\mu\to \8$.}\label{fig:eje5}
\end{center}
\end{figure}

In the $n=2$ case, the above conjecture claims that the minimal annuli $\A_2(\mu)$ are always embedded and converge as $\mu\to \8$ to a doubly covered vertical equator of $\B^3$. In particular, an affirmative answer to the conjecture would imply the existence of non-rotational embedded capillary minimal annuli in $\B^3$ with intersection angles $\theta < \pi/2$ converging to $\pi/2$.

After Theorem \ref{capillary}, one can update the problem by Wente discussed in the introduction as follows: \emph{is any embedded capillary minimal annulus in $\B^3$ always foliated by spherical curvature lines?} Observe that, in the view of Theorem \ref{th:uni}, this problem can be regarded as a generalization of the critical catenoid conjecture. 

Each of the minimal annuli $\mathbb{A}_n(\mu)$ has at least three planes of reflective symmetry. Thus, McGrath's characterization \cite{M} of the critical catenoid in the free boundary case does not extend to the general capillary situation. Motivated by \cite{CSW}, it is natural to ask if any embedded capillary minimal annulus in $\B^3$ with prismatic symmetry group of order $4n$ is one of the annuli $\mathbb{A}_n(\mu)$.

In this line, J. Choe proposed to the authors the following conjecture: \emph{any embedded capillary minimal annulus in $\B^3$ that is symmetric with respect to the three coordinate planes of $\R^3$ is either a compact piece of a catenoid, or one of the minimal annuli $\mathbb{A}_{2n}(\mu)$.}

\section{Uniqueness: proof of Theorem \ref{th:uni}}\label{sec:uniqueness}

Let $\Sigma$ be a compact immersed minimal annulus in $\R^3$ whose boundary $\parc \Sigma$ is composed of two closed curvature lines of $\Sigma$. It is then well known that $\Sigma$ has no umbilic points, and that there exists a foliation $\cF$ of $\Sigma$ by regular curvature lines, so that both connected components of $\parc \Sigma$ are elements of $\cF$. We say that $\Sigma$ is \emph{foliated by spherical curvature lines} if each curve of the foliation $\cF$ lies in some sphere of $\R^3$.

We will prove in Theorem \ref{th:embedan} below that any free boundary minimal annulus $\Sigma$ embedded in $\B^3$ and foliated by spherical curvature lines has two planes of symmetry. Therefore, by Seo's theorem \cite{Seo}, $\Sigma$ is the critical catenoid, and this proves Theorem \ref{th:uni}.

Note that in Theorem \ref{th:embedan} we do not really assume a free boundary condition.

\begin{theorem}\label{th:embedan}
Let $\Sigma$ be a minimal annulus embedded in the unit ball $\B^3$, with  $\parc\Sigma\subset \S^2$ and foliated by spherical curvature lines. Then, $\Sigma$ is symmetric with respect to two planes of $\R^3$.
\end{theorem}
\begin{proof}
To start, let us consider the case that $\Sigma$ is foliated by \emph{planar} curvature lines. In that situation, each of the boundary curves of $\Sigma$ is a planar curve contained in the unit sphere, i.e., it is a circle, and $\Sigma$ intersects $\S^2$ with constant angle along it. So, $\Sigma$ is a compact piece of a catenoid, by uniqueness of the solution to Björling's problem. The result is then trivial in this case.

So, from now on, we assume that $\Sigma$ is not foliated by planar curvature lines. Therefore, by our discussion in Section \ref{sec:prelim}, and up to an ambient isometry and a homothety, we can represent $\Sigma$ (which after these normalizations is now is embedded in some unspecified ball of $\R^3$) as a subset of a complete minimal surface $\Sigma'$ parametrized as in \eqref{weipar0} in terms of Weierstrass data given by 
 \begin{equation}\label{weida9}
\phi(z)=b -4 \wp (z), \hspace{1cm} g(z)=\hat{g}_0 \,  {\rm exp} \left(\int_0^z \frac{1}{\phi(w)} dw \right).
 \end{equation} 
Here, $\hat{g}_0>0$, $b\in \R$ satisfies the cubic equation \eqref{cubicb} and $\wp(z)$ is a (possibly degenerate) Weierstrass P-function that satisifies the ODE 
\begin{equation}\label{odewp}
\wp'(z)^2 = 4 \wp(z)^3 -g_2 \wp(z) - g_3
\end{equation}
with respect to some $g_2,g_3\in \R$.

More specifically, if $\Lambda$ is the set of poles of $\wp(z)$ in $\C$, then $\Sigma'$ is given by a conformal parametrization $\psi(u,v):\C\setminus \Lambda \flecha \R^3$ as in \eqref{conf} so that the $v$-curves are spherical curvature lines, contained in spheres $S(c(u),R(u))$. We allow the possibility that $R(u_0)=\8$ for some $u_0$, which corresponds to the situation in which $\psi(u_0,v)$ is a planar curvature line. All the centers $c(u)$ lie in a common vertical line of $\R^3$. The minimal annulus $\Sigma$ is then obtained as the quotient of an adequate closed subset of $\Sigma'$ by the deck transformation $(u,v)\mapsto (u,v+T)$ for some $T>0$. In particular, $\psi(u,v)$ is $T$-periodic with respect to $v$.

So, we view the universal cover of $\Sigma$ as parametrized by $\psi(u,v):[c,d]\times \R\flecha \R^3$, with $\psi(u,v+2T)=\psi(u,v)$ for some $T>0$. By the compactness of $\Sigma$, there are no poles of $\phi$ (i.e., no points of $\Lambda$) in the vertical strip $[c,d]\times i\R$, since any such point would generate a flat end for $\Sigma$.

We divide the proof into steps.

{\bf Step 1:} \emph{$\Sigma$ has a planar curvature line $\psi(u_0,v)$, for some $u_0\in [c,d]$.}

To see this, we first note that by \eqref{radan}, the planar curvature lines of $\Sigma$ correspond to the values $u$ where $\alfa(u)=0$. Assume that $\alfa(u)\neq 0$ for every $u\in [c,d]$. The height of the centers $c_3(u)$ of the curvature lines $\psi(u,v)$ satisfy \eqref{wentecenter} and \eqref{wence}. Thus, $c_3(u)$ is monotonic and finite when $u\in [c,d]$. In particular, $\Sigma$ cannot lie in the unit ball unless at least one of its two boundary curves is contained in two spheres of different centers along the $x_3$-axis. This is only possible in the present situation if that intersection is a circle, and so $\Sigma$ is a compact piece of a catenoid, again the by uniqueness of the Björling problem. Since our assumption was that $\Sigma$ is not foliated by planar curvature lines, we reach a contradiction. This proves Step 1.

{\bf Step 2:} \emph{The modular discriminant $\Delta_{\rm mod}:=g_2^3 -27 g_3^2$ cannot be negative.}

Up to a horizontal translation in the $(u,v)$-plane, we may assume that the planar curvature line of $\Sigma$ whose existence was shown in Step 1 is placed along the $i\R$ axis, i.e., $\psi(0,v)$ is such a curvature line. After this translation, $\Sigma$ is parametrized by the Weierstrass data 
 \begin{equation}\label{weida17}
\phi(z)=b -4 \wp (z+a), \hspace{1cm} g(z)=\hat{g}_0 \,  {\rm exp} \left(\int_0^z \frac{1}{\phi(w)} dw \right),
 \end{equation} 
for some $a\in \R$. Now, since $\Sigma$ intersects the $x_3=0$ plane at a constant angle along $\psi(0,v)$, it follows that $|g(iv)|\equiv \hat{g}_0$, constant along $i\R$. This implies by $g/g'=\phi$ and $b\in \R$ that $\phi(iv) \in \R$ for every $v\in \R$. Therefore, $\wp(a+iv)\in \R$ for every $v\in \R$. 

But now we recall that, as explained in Section \ref{sec:were}, the condition $\Delta_{\rm mod}<0$ implies that $\Lambda$ is a rhombic lattice. In such a lattice, it is well known that $\wp(z)$ only takes real values along the horizontal and vertical lines in $\C$ that pass through points of $\Lambda$. Therefore, the line $a+i\R$ intersects the lattice $\Lambda$. So, by \eqref{weida17}, $\phi$ has poles along the $i\R$ axis. Since at those points $\Sigma$ would have flat ends, this contradicts the compactness of $\Sigma$. This rules out the case $\Delta_{\rm mod}<0$.

{\bf Step 3:} \emph{$\Delta_{\rm mod}$ cannot be zero.}

If $g_2=g_3=0$, then $\wp(z)=1/z^2$, which only takes real vales along $\R$ and $i\R$. Thus, the argument of Step 2 applies.

If $g_2^3 = 27 g_3^2 \neq 0$, the function $\wp(z)$ is singly periodic, with either a real or a purely imaginary period. Since $\psi(u,v)$ is $T$-periodic in the $v$-direction, we obtain from the form of the Weierstrass data \eqref{weida9} that this period is purely imaginary, i.e., $\wp(z+\omega_0)=\wp(z)$ for some $\omega_0\in i\R$. The degenerate Weiestrass function $\wp(z)$ is explicitly given in these conditions by (see \cite{Ch}, p. 46)
$$\wp(z) = \left(\frac{\pi}{\omega_0}\right)^2 \frac{1}{\sin^2(\pi z/\omega_0)}- \frac{1}{3} \left(\frac{\pi}{\omega_0}\right)^2.$$ It follows from this expression that, again, $\wp(z)$ cannot be simultaneously real and finite along a vertical line, and the argument of Step 2 applies. This shows that $\Delta_{\rm mod}=0$ is impossible for $\Sigma$.

{\bf Step 4:} If \emph{$\Delta_{\rm mod}>0$, then $\Sigma$ has two vertical symmetry planes.}
 
 Arguing as in Step 2, after a translation in the $u$-variable, we parametrize $\Sigma$ by the data \eqref{weida17}, so that $\psi(0,v)$ is a planar curvature line. In this way, $\wp(a+iv)\in \R$ for every $v\in \R$. Now, this time, by $\Delta_{\rm mod}>0$, the set $\Lambda$ is a rectangular lattice in $\C$, with generators $\{2\omega_1,2\omega_2\}$, being $\omega_1>0$ and $\omega_2 \in i\R$, with ${\rm Im}(\omega_2)>0$; see Section \ref{sec:were}.
 
In this situation, it is well known that $\wp(z)$ is real along a vertical line $a+i\R$ if and only if $a\in \R$ is a multiple of the real half-period $\omega_1$, i.e. $a$ is of the form $j\omega_1$ for some $j\in \Z$. If $j$ is even, the line $a+i\R$ intersects the lattice $\Lambda$, and this contradicts the compactness of $\Sigma$ as in Step 2. So, $j$ must be odd, and by periodicity we can assume $j=1$, i.e., $a=\omega_1$. So, $\Sigma$ can be parametrized by $\psi(u,v):[c,d]\times \R\flecha \R^3$ in terms of the Weierstrass data 
 \begin{equation}\label{weida10}
\phi(z)=b -4 \wp (z+\omega_1), \hspace{1cm} g(z)=\hat{g}_0 \,  {\rm exp} \left(\int_0^z \frac{1}{\phi(w)} dw \right),
 \end{equation} 
where $[c,d]$ is contained in $(-\omega_1,\omega_1)$, since $\phi(z)$ cannot have poles on $[c,d]\times i\R$. Also, recall that $\psi(u,v)$ is $T$-periodic in the $v$-direction. This means by \eqref{weida10} that $T=2n {\rm Im}(\omega_2)$ for some $n\in \N$. Note that, in particular, $g(2n \omega_2)=g(0)$.

Consider next the planar curvature line $\gamma(v):=\psi(0,v)$, which is contained in the $x_3=0$ plane. Since the principal curvatures of $\Sigma$ along the $v$-curves are $\kappa_2=e^{-2\omega}>0$, see Section \ref{sec:prelim}, we see that $\gamma(v)$ is a locally convex planar curve. Since $\Sigma$ is embedded, $\gamma(v)$ is thus globally convex. Since $|g(iv)|\equiv \hat{g}_0$ is constant, this global convexity implies that the Gauss map of $\Sigma$ along $\gamma(v)$ lies in a circle, and has rotation index equal to $1$. In other words, the map $g(iv):[0,2n{\rm Im}(\omega_2))\flecha \C$ is a regular injective parametrization of the circle $|z|=\hat{g}_0$ in $\C$. 

Now, since $\phi(z)$ is $2\omega_2$-periodic, it follows from \eqref{weida10} that, in order to have $g(2n\omega_2)=g(0)$, by the rotation index discussion above, it must happen that $g(2\omega_2)=g(0)e^{2\pi i/n}$.

Assume for one moment that $n>1$ holds. In that situation, by the form of the Weierstrass data \eqref{weida10}, the annulus $\Sigma$ is invariant by the rotations around the $x_3$-axis of angles $2k\pi/n$, with $k\in \{1,\dots, n-1\}$. Also, since both $\phi,g$ are real along the real axis, it also follows from the Weierstrass representation that $\Sigma$ is invariant by reflection with respect to the plane $x_2=0$. Thus, $\Sigma$ is invariant by at least two reflections with respect to vertical planes of $\R^3$ containing the $x_3$-axis. This would prove Theorem \ref{th:embedan}.

We next prove that $n>1$. For this, we will consider below two separate cases.

\begin{figure}
\begin{center}
\includegraphics[height=4.2cm]{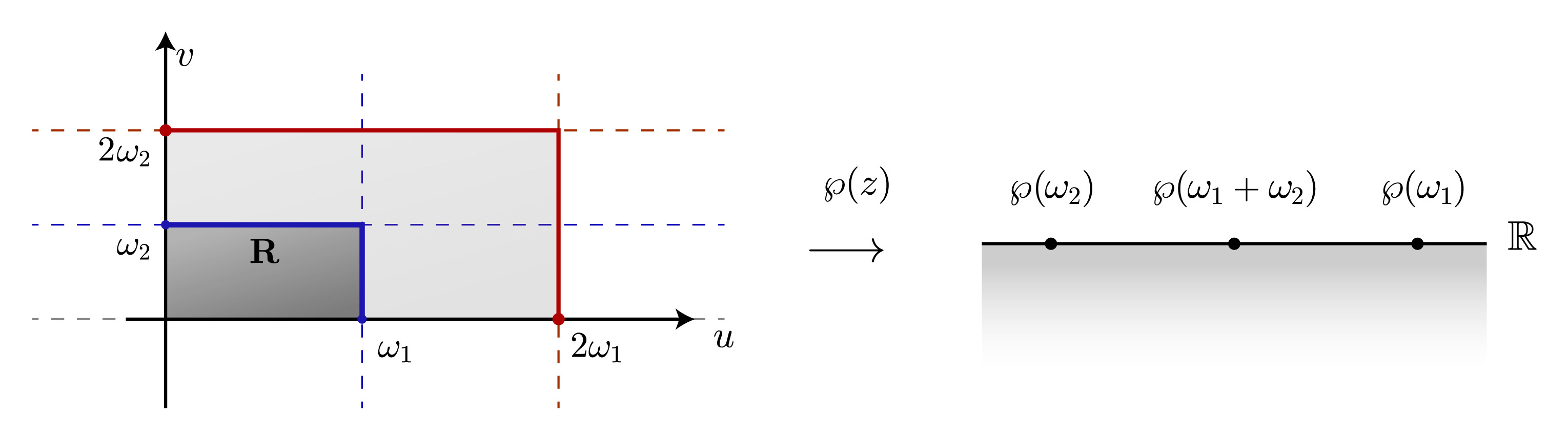} 
\caption{Values of the Weierstrass function $\wp(z)$ along the rectangle $\mathbf{R}$.}\label{fig:fig5}
\end{center}
\end{figure}

Let $\mathbf{R}$ be the boundary of the rectangle generated by the half-periods, i.e., the rectangle in $\C$ with vertices $\{0,\omega_1,\omega_2,\omega_1+\omega_2\}$. It is a standard result of Weierstrass functions that $\wp(z)$ maps $\mathbf{R}$ bijectively onto the extended real line $\R\cup \{\8\}$; see Figure \ref{fig:fig5}. In particular, there exists a unique value $\mu \in \mathbf{R}$ such that $4\wp(\mu)=b$, where $b\in \R$ is the constant appearing in \eqref{weida10}. If we now use \eqref{odewp} and \eqref{cubicb}, we deduce that $\wp'(\mu) = \pm 1/4$. Since $\wp'(z)\in i\R$ along the vertical sides of $\mathbf{R}$, this implies that either $\mu \in (0,\omega_1)$, or $\mu$ lies in the open segment of $\mathbf{R}$ between $\omega_2$ and $\omega_1+\omega_2$; here, we should recall that $\wp'(z)=0$ at the vertices $\omega_1,\omega_2$ and $\omega_1 + \omega_2$ of $\mathbf{R}$.

Consider first of all the case that $\mu=x^*+\omega_2$, where $x^*\in (0,\omega_1)$. Then, $\wp'(\mu)=1/4$. In that case, $g(z)$ can be represented as in Lemma \ref{gex}, i.e. by \eqref{w80}. Since $g(2\omega_2)=g(0)e^{2\pi i/n}$, the condition $n=1$ is equivalent to $g(2\omega_2)=g(0)$. Now, by an application of \eqref{forsigma} for $z:=\mu-(\omega_1+2 \omega_2)$, $j=1$ and $k=2$, we deduce after a computation from \eqref{w80}, \eqref{w100} and \eqref{ffsig} that the condition $g(2\omega_2)=g(0)$ is equivalent to
\begin{equation}\label{condiper1}
\omega_2 \zeta(\mu) - \mu \zeta(\omega_2)\in  \frac{i\pi}{2} \Z.
\end{equation} 
We should point out here that any number of the form $\omega_2 \zeta(z)-z \zeta(\omega_2)$, where $z=x+\omega_2$, $x\in \R$, is always purely imaginary.

In order to see that $n>1$, we consider the function $$f(x)=\frac{2}{\pi i} (\omega_2 \zeta(x+\omega_2) - (x+\omega_2) \eta_2) :(0,\omega_1)\flecha \R,$$ where were are using the standard notation $\eta_2:=\zeta(\omega_2)\in i\R$. We have $f(0)=0$, and $f(\omega_1)=1$, due to the Legendre relation $\omega_2 \zeta(\omega_1+\omega_2) - (\omega_1+\omega_2)\eta_2 = \frac{i\pi}{2}$, see e.g. \cite[p. 446]{WW}. Now, let us check that $f'>0$, what would prove that \eqref{condiper1} cannot hold, and thus, that $n>1$, as wished.

Otherwise, there would exist $x_0\in (0,\omega_1)$ such that $\wp(x_0+\omega_2)=-\zeta'(x_0+\omega_2)=-\eta_2 /\omega_2$. But on the other hand, by the mean value theorem applied to $\zeta(z)$ over the vertical segment of $\mathbf{R}$ of points of the form $\omega_1+y$, $y\in (0,\omega_2)\subset i\R$, there exists some $\omega_1+y_0$ where $\wp(\omega_1+y_0)=-\eta_2/\omega_2$. For that, we use that $\zeta(\omega_1+\omega_2)=\zeta(\omega_1)+\zeta(\omega_2)$. Since $\wp$ is injective over the rectangle $\mathbf{R}$, we conclude that such a point $x_0+\omega_2$ cannot exist.

Consider next the remaining case $\mu\in (0,\omega_1)$. This time $\wp'(\mu)=-1/4$ and by a similar computation to the one in Lemma \ref{gex} we obtain the following expression for $g$: $$g(z) = g_0\,  {\rm exp} \left(2 \zeta(\mu) (z+\omega_1)\right) \frac{\sigma(\mu-(z+\omega_1))}{\sigma(\mu + (z+\omega_1))},$$ where this time, due to \eqref{ffsig}, we have $$g(0)=-g_0 \, {\rm exp} \big(2(\zeta(\mu) \omega_1 - \mu \zeta(\omega_1))\big).$$ The condition that $n=1$, i.e., that $g(2\omega_2)=g(0)$, is then again written as \eqref{condiper1}, by an application of \eqref{forsigma} for $z:=\mu+\omega_1+2 \omega_2$, $j=-1$ and $k=-2$.

This time, we consider the function $$f(x)=\frac{2}{\pi i} \left( \omega_2 \,\zeta(x) - x\, \eta_2\right):(0,\omega_1)\flecha \R.$$ We have now $f(x)\to \8$ as $x\to 0^+$ and $f(\omega_1)=1$, again by the Legendre relation for Weierstrass functions. Moreover, if $f'(x_0)=0$ for some $x_0\in (0,\omega_1)$, then we would have $\wp(x_0)=-\eta_2 /\omega_2$. The same argument as in the first case shows that this contradicts the fact that $\wp$ is injective over $\mathbf{R}$. Thus, $f'<0$ on $(0,\omega_1)$, and in particular $f(\mu)> 1$. Therefore, \eqref{condiper1} does not hold, and so $n>1$ as wished. This completes the proof of Theorem \ref{th:embedan}.
\end{proof}

\def\refname{References}

\vskip 0.2cm

\noindent Isabel Fernández

\noindent Departamento de Matemática Aplicada I,\\ Instituto de Matemáticas IMUS \\ Universidad de Sevilla (Spain).

\noindent  e-mail: {\tt isafer@us.es}

\vskip 0.2cm

\noindent Laurent Hauswirth

\noindent Laboratoire d'Analyse et de Matématiques Apliquées \\ Université Gustave Eiffel, Paris (France)).

\noindent  e-mail: {\tt laurent.hauswirth@u-pem.fr}

\vskip 0.2cm

\noindent Pablo Mira

\noindent Departamento de Matemática Aplicada y Estadística,\\ Universidad Politécnica de Cartagena (Spain).

\noindent  e-mail: {\tt pablo.mira@upct.es}

\vskip 0.4cm

\noindent This research has been financially supported by Project PID2020-118137GB-I00 funded by MCIN/AEI /10.13039/501100011033.


\begin{thebibliography}{9}

\bibitem{Ab} U. Abresch, Constant mean curvature tori in terms of elliptic functions, {\it J. Reine Angew. Math.} {\bf 394} (1987), 169--192.

\bibitem{AN} L. Ambrozio, I. Nunes, A gap theorem for free boundary minimal surfaces in the three-ball, {\it Comm. Anal. Geom.} {\bf 29} (2021), 283--292.

\bibitem{CFS} A. Carlotto, G. Franz, M.B. Schulz, Free boundary minimal surfaces with connected boundary and arbitrary genus. {\it Cambridge J. Math.}, to appear.

\bibitem{CSW} A. Carlotto, M.B. Schulz, D. Wiygul, Infinitely many pairs of free boundary minimal surfaces with the same topology and symmetry group, preprint, arXiv:2205.04861.

\bibitem{Ch} K. Chandrasekharan, Elliptic Functions, Grundlehren der mathematischen Wissenschaften, {\bf 281}, Springer Verlag, 1985.

\bibitem{Choe} J. Choe, The periodic Plateau problem and its applications, arXiv:2104.09087

\bibitem{De} B. Devyver, Index of the critical catenoid, {\it Geom. Dedicata} {\bf 199} (2019), 355--371.

\bibitem{D} H. Dobriner, Die Minimalflächen mit Einem System Sphärischer Krümmungslinien. {\it Acta Math.} {\bf 10} (1887), 145--152.

\bibitem{F} A. Fraser, Extremal Eigenvalue Problems and Free Boundary Minimal Surfaces in the Ball, in: Geometric Analysis, Lecture Notes in Mathematics {\bf 2263}, Springer. 

\bibitem{FL} A. Fraser, M. Li, Compactness of the space of embedded minimal surfaces with free boundary in three-manifolds with nonnegative Ricci curvature and convex boundary, {\it J. Differential Geom.} {\bf 96} (2014), 183--200.

\bibitem{FSa} A. Fraser, P. Sargent, Existence and classification of $S^1$-invariant free boundary annuli and Möbius bands in $\B^n$, {\it J. Geom. Anal.}  {\bf 31} (2021), 2703--2725.

\bibitem{FS0} A. Fraser, R, Schoen, The first Steklov eigenvalue, conformal geometry, and minimal surfaces., {\it Adv. Math.} {\bf 226} (2011), 4011--4030.

\bibitem{FS} A. Fraser, R. Schoen, Sharp eigenvalue bounds and minimal surfaces in the ball, {\it Invent. Math.} {\bf 203} (2016), 823--890.

\bibitem{FPZ} A. Folha, F. Pacard, T. Zolotareva, Free boundary minimal surfaces in the unit 3-ball, {\it Manuscripta Math.} {\bf 154} (2017), 359--409.


\bibitem{KL} N. Kapouleas, M. Li, Free boundary minimal surfaces in the unit three-ball via desingularization of the critical catenoid and the equatorial disc, {\it J. Reine Angew. Math.} {\bf 776} (2021), 201--254.

\bibitem{KW} N. Kapouleas, D. Wiygul, Free boundary minimal surfaces with connected boundary in the $3$-ball by tripling the equatorial disc, {\it J. Differential Geom.}, to appear; arXiv:1711.00818.

\bibitem{KZ} N. Kapouleas, J. Zou, Free Boundary Minimal surfaces in the Euclidean Three-Ball close to the boundary, preprint, arXiv:2111.11308.

\bibitem{Ke} D. Ketover, Free boundary minimal surfaces of unbounded genus, preprint, arXiv:1612.08691.

\bibitem{KM} R. Kusner, P. McGrath, On free boundary minimal annuli embedded in the unit ball, {\it Amer. J. Math.}, to appear, arXiv:2011.06884.

\bibitem{LY} J. Lee, E. Yeon, A new approach to the Fraser-Li conjecture with the Weierstrass representation formula, {\it Proc. Amer. Math. Soc.} {\bf 149} (2021), 5331--5345.

\bibitem{L} M. Li, Free boundary minimal surfaces in the unit ball: recent advances and open questions. In: Proceedings of the International Consortium of Chinese Mathematicians, 2017 (First Annual Meeting), International Press of Boston, Inc., pp. 401--436 (2020).

\bibitem{LM} V. Lima, A. Menezes, A two-piece property for free boundary minimal surfaces in the ball, {\it Trans. Amer. Math. Soc.} {\bf 374} (2021), 1661--1686.

\bibitem{M} P. McGrath, A characterization of the critical catenoid, {\it Indiana Univ. Math. J.} {\bf 67} (2018),  889--897.
 
\bibitem{MPR} W.H. Meeks, J. Pérez, A. Ros, Properly embedded minimal planar domains, {\it Ann. Math.} {\bf 181} (2015),  473--546.
 
\bibitem{Nit} J.C.C. Nitsche, Stationary partitioning of convex bodies, \emph{Arch. Rat. Mech. Anal.} {\bf 89} (1985), 1--19.

\bibitem{Pa} S.H. Park, A constant mean curvature annulus tangent to two identical spheres is Delaunay, {\it Pacific J. Math.} {\bf 251} (2011),197--206.

\bibitem{RS} A. Ros, R. Souam, On stability of capillary surfaces in a ball, {\it Pacific J. Math.} {\bf 178} (1997), 345--361.

\bibitem{Seo} D.H. Seo, Sufficient symmetry conditions for free boundary minimal annuli to be the critical catenoid, preprint, arXiv:2112.11877

\bibitem{Shi} M. Shiffman, On surfaces of stationary area bounded by two circles, or convex curves, in parallel planes, {\it Ann. Math.} {\bf 63} (1956), 77--90.

\bibitem{SZ} G. Smith, D. Zhou, The Morse index of the critical catenoid, {\it Geom. Dedicata} {\bf 201} (2019), 13--19. 

\bibitem{So} R. Souam, Schiffer's Problem and an isoperimetric inequality for the first buckling eigenvalue of domains on $\S^2$, {\it Ann. Global Anal. Geom.} {\bf 27} (2005), 341--354.

\bibitem{T} H. Tran, Index Characterization for Free Boundary Minimal Surfaces, {\it Comm. Anal. Geom.} {\bf 28} (2020), 189--222.

\bibitem{Wa1} R. Walter, Explicit examples to the $H$-problem of Heinz Hopf, {\it Geom. Dedicata} {\bf 23} (1987), 187--213.

\bibitem{Wa2} R. Walter, Constant mean curvature tori with spherical curvature lines in noneuclidean geometry, {\it Manuscripta Math.} {\bf 63} (1989), 343--363.


\bibitem{W0} H.C. Wente, Counterexample to a conjecture of H. Hopf, {\it Pacific J. Math.} {\bf 121} (1986), 193--243.

\bibitem{W} H.C. Wente, Constant mean curvature immersions of Enneper type, {\it Mem. Amer. Math. Soc.} {\bf 478} (1992).

\bibitem{W2} H.C. Wente, Tubular capillary surfaces in a convex body. Advances in geometric analysis and continuum mechanics, Edited by P. Concus and K. Lancaster, International Press 1995, 288--298.

\bibitem{WW} E.T. Whittaker, G.N. Watson, A course of modern analysis, Cambridge University Press.

\end{thebibliography}
\end{document}